\documentclass[smallextended,nospthms]{svjour3}
\usepackage[utf8]{inputenc}
\usepackage[T1]{fontenc}
\usepackage{indentfirst}
\usepackage{textcomp}
\usepackage[english]{babel}
\usepackage{amsmath,amsfonts,amssymb,amsopn,amscd,amsthm}
\usepackage{mathrsfs,stmaryrd}
\usepackage[enableskew]{youngtab}
\usepackage{shuffle}
\usepackage{graphicx}
\usepackage[dvipsnames]{xcolor}
\usepackage{xy}  \xyoption{all}
\usepackage{pb-diagram,pb-xy}
\usepackage{enumerate}
\usepackage[colorlinks=true,citecolor=DarkOrchid,linkcolor=NavyBlue]{hyperref}
\hypersetup{pdftitle={A central limit theorem for the characters of the infinite symmetric group and of the infinite Hecke algebra}}

\newtheorem{theorem}{Theorem}
\newtheorem{proposition}[theorem]{Proposition}

\newtheorem{lemma}[theorem]{Lemma}
\newtheorem{conjecture}[theorem]{Conjecture}
\theoremstyle{remark}
\newtheorem*{example}{Example}
\newtheorem*{remark}{Remark}

\newcommand{\E}{\mathrm{e}}
\newcommand{\I}{\mathrm{i}}
\newcommand{\N}{\mathbb{N}}
\newcommand{\Z}{\mathbb{Z}}
\newcommand{\Q}{\mathbb{Q}}
\newcommand{\R}{\mathbb{R}}
\newcommand{\C}{\mathbb{C}}
\newcommand{\hecke}{\mathscr{H}}

\newcommand{\sym}{\mathfrak{S}}
\newcommand{\comp}{\mathfrak{C}}
\newcommand{\setpart}{\mathfrak{P}}
\newcommand{\Sym}{\mathbf{Sym}}
\newcommand{\QSym}{\mathbf{QSym}}
\newcommand{\FQSym}{\mathbf{FQSym}}
\newcommand{\ym}{\mathfrak{Y}}
\newcommand{\card}{\mathrm{card}\,} 
\newcommand{\GL}{\mathrm{GL}}

\newcommand{\proba}{\mathbb{P}}
\newcommand{\qproba}{\mathbb{Q}}
\newcommand{\esper}{\mathbb{E}}
\newcommand{\eps}{\varepsilon}
\newcommand{\lle}{\left[\!\left[} 
\newcommand{\rre}{\right]\!\right]} 
\newcommand{\obs}{\mathscr{O}}
\newcommand{\alg}{\mathscr{A}}
\newcommand{\blg}{\mathscr{B}}
\newcommand{\id}{\mathrm{id}}
\newcommand{\tr}{\mathrm{tr}\,}
\newcommand{\arra}{\mathfrak{A}}

\newcommand{\scal}[2]{\left\langle #1\vphantom{#2}\,\right |\left.#2 \vphantom{#1}\right\rangle}
\newcommand{\figcap}[2]{\begin{figure}[ht] \begin{center} {\footnotesize{#1}} \caption{#2} \end{center} \end{figure}}

\newcommand{\comment}[1]{}

\begin{document}
\author{Pierre-Lo\"ic M\'eliot}
\institute{The Gaspard--Monge Institute of electronics and computer science,
University of Marne-La-Vall\'ee Paris-Est,
77454 Marne-la-Vall\'ee Cedex 2, France}
\email{meliot@phare.normalesup.org}

\title{A central limit theorem for the characters of the infinite symmetric group and of the infinite Hecke algebra}
\titlerunning{A central limit theorem for the characters of $\sym_{\infty}$ and $\hecke(\sym_{\infty})$}

\begin{abstract} In this paper, we review the representation theory of the infinite symmetric group, and we extend the works of Kerov and Vershik by proving that the irreducible characters of $\sym_{\infty}$ always satisfy a central limit theorem. Hence, for any point of the Thoma simplex, the corresponding measures on the levels $\ym_{n}$ of the Young graph have a property of gaussian concentration. By using the Robinson-Schensted-Knuth algorithm and the theory of Pitman operators, we relate these results to the properties of certain random permutations obtained by riffle shuffles, and to the behaviour of random walks conditioned to stay in a Weyl chamber.
\end{abstract}

\maketitle

For $n \geq 0$, we denote by $\sym_{n}$ the \textbf{symmetric group} of order $n$, that is to say the group of permutations of $\lle 1,n\rre=\{1,2,\ldots,n\}$. We also denote by $\sym_{\infty}$ the inductive limit $\bigcup_{n =0}^{\infty}\uparrow \sym_{n}$. The representation theory of $\sym_{\infty}$ is quite difficult, because it is a wild group, \emph{i.e.}, an infinite group with factor representations of type II or III. However, the classification of the finite factor representations of $\sym_{\infty}$ is entirely known since the works of E. Thoma (\emph{cf.} \cite{Tho64}). Hence, these representations are labelled by the pairs $\omega=(\alpha,\beta)$ of non-increasing sequences of non-negative real numbers
$$\alpha=(\alpha_{1}\geq \alpha_{2}\geq \cdots \geq \alpha_{n}\geq \cdots)\qquad;\qquad \beta=(\beta_{1}\geq \beta_{2}\geq \cdots \geq \beta_{n}\geq \cdots)$$
such that $\sum_{i=1}^{\infty}\alpha_{i}+\sum_{i=1}^{\infty}\beta_{i}\leq 1$. Moreover, if $\chi^{\omega}$ is the normalized character of $\sym_{\infty}$ corresponding to $\omega$, and if $\sigma \in \sym_{\infty}$ is a product of disjoint cycles of lengths $l_{1},\ldots,l_{r}\geq 2$, then
$$\chi^{\omega}(\sigma)=\prod_{j=1}^{r}\left(\sum_{i=1}^{\infty} (\alpha_{i})^{l_{j}}+ (-1)^{l_{j}-1} \sum_{i=1}^{\infty} (\beta_{i})^{l_{j}} \right).$$ \bigskip

Now, if one considers a restriction $\chi^{\omega}_{|\sym_{n}}$ of $\chi^{\omega}$ to a finite symmetric group, it is not in general an irreducible character of $\sym_{n}$, and one can expand it as a linear combination
$$\chi^{\omega}_{|\sym_{n}}=\sum_{\lambda\in \ym_{n}} \proba_{n,\omega}[\lambda]\,\chi^{\lambda}$$
of the normalized irreducible characters of $\sym_{n}$. The sum runs over the set $\ym_{n}$ of (isomorphism classes of) irreducible representations of $\sym_{n}$, which happen to correspond bijectively to the \textbf{integer partitions} of size $n$, that is to say the non-decreasing sequences of positive integers that sum to $n$:
$$\ym_{n}=\left\{\lambda=(\lambda_{1}\geq \lambda_{2}\geq \cdots \geq \lambda_{r})\,\,\bigg|\,\,|\lambda|=\sum_{i=1}^{r}\lambda_{i}=n\right\}.$$
Hence, for any pair $\omega=(\alpha,\beta)$ and any integer $n$, one obtains a probability measure $\proba_{n,\omega}$ on $\ym_{n}$. It has been shown by Kerov and Vershik (see \cite{KV81a,KV81b}) that these systems of probability measures, called coherent systems, satisfy a law of large numbers. Thus, the irreducible characters $\chi^{\lambda}$ (viewed as random functions on the symmetric groups) and the renormalized coordinates of the random partitions $\lambda$ under $\proba_{n,\omega}$ both converge in probability. In this paper, we will refine this result by showing that the convergence is always gaussian, see our main Theorems \ref{mainalg}, \ref{mainmeasure} and \ref{maingeom}.\bigskip\bigskip

The article is organized as follows. In the first part of the paper, we recall the results of E. Thoma, using extensively the formalism of symmetric functions. In Section \ref{randompermutation}, we explain how coherent systems of measures are related to models of random permutations obtained by generalized riffle shuffles; here we use well-known arguments from the theory of symmetric and quasi-symmetric functions. In \S\ref{obs}, we present a ``method of non-commutative moments'' that is due to Kerov and Olshanski (see \cite{KO94,IO02,Sni06b}), and that is extremely useful in the setting of asymptotic representation theory of the symmetric groups. \bigskip

We apply this method in Sections \ref{largenumbers} and \ref{clt} to reprove the law of large numbers of Kerov and Vershik and to establish our central limit theorems. The arguments are exactly those that the author used with V. F\'eray in \cite{FM10} for the asymptotic study of the $q$-Plancherel measures; in our general setting, they shall seem somewhat simpler. However, when looking at the fluctuations of the rows and columns of the random partitions taken with respect to the probability $\proba_{n,\omega}$, new difficulties arise when some coordinates of the parameter $\omega$ in the Thoma simplex are equal. We address these difficulties in Section \ref{oconnell} by using and generalizing arguments due to N. O'Connell (\emph{cf.} \cite{OC03}, see also \cite{Litt95,OCY02,BBO05,Bia09}); hence, we relate a random walk on $\ym=\bigsqcup_{n \in \N} \ym_{n}$ whose marginale laws are the $\proba_{n,\omega}$'s, to a multidimensional random walk  conditioned to stay in a Weyl chamber. This link ensures that the fluctuations of rows or columns corresponding to the same parameter are the same as those of the eigenvalues of a gaussian hermitian matrix. Incidentally, the link between random partitions and random walks provides a new proof of some results of Section \ref{clt}, in particular Theorem \ref{almostgeom}.\bigskip

Finally, in \S\ref{qtplancherel}, we explain how our results behave with respect to the quantization of the symmetric group, that is to say that we replace $\sym_{\infty}$ by its Hecke algebra $\hecke(\sym_{\infty})$. In particular, we perform all the computations in the case when the irreducible character of $\hecke_{q}(\sym_{\infty})$ is the so-called Jones-Ocneanu trace (\cite{Jon87}) of parameter $z=-(1-q)(1-t)$, thereby providing an asymptotic study of a $(q,t)$-deformation of the Plancherel measures of the symmetric groups.\bigskip

We have tried to make this paper accessible to a wide audience, and self-contained. In particular, we will recall many well-known results on symmetric groups, symmetric functions and (free) quasi-symmetric functions.

\bigskip
\begin{acknowledgements}
The author would like to express his gratitude to P. Biane, M. Bo\.{z}ejko, V. F\'eray, S. Giraudo, F. Hivert, A. Nikeghbali, J.-C. Novelli, P. \'Sniady, J.-Y. Thibon and A. Vershik for various discussions or suggestions.
\end{acknowledgements}
\bigskip

\section{Representation theory of the infinite symmetric group}
In this first paragraph, we recall the representation theory of the infinite symmetric group, and a basic tool that we will use throughout the paper, namely, the Hopf algebra of symmetric functions (later, we shall also need the algebras of quasi-symmetric and free quasi-symmetric functions).

\subsection{Representation theory of the symmetric groups $\sym_{n}$} To begin with, let us recall the representation theory of the finite symmetric groups $\sym_{n}=\sym(\lle 1,n\rre)$. If $\sigma \in \sym_{n}$ is a permutation, it can be written uniquely as a product of disjoint cycles; for example, $$41728635=(1,4,2)(3,7)(5,8)(6).$$ Then, two permutations are conjugated in $\sym_{n}$ if and only if the lengths or their cycles form the same integer partition (here, $(3,2,2,1)$). Thus, the conjugacy classes of $\sym_{n}$ are labelled by the elements of $\ym_{n}$. The irreducible representations of $\sym_{n}$ are also labelled by integer partitions of size $n$. Indeed, if $\lambda=(\lambda_{1},\lambda_{2},\ldots,\lambda_{r})$, let us denote 
\begin{align*}\Delta^{\lambda}(x_{1},\ldots,x_{r})&=\prod_{i=1}^{r}\Delta(x_{\lambda_{1}+\cdots+\lambda_{i-1}+1},\ldots,x_{\lambda_{1}+\cdots+\lambda_{i}})\\
&=\prod_{i=1}^{r} \left(\prod_{\lambda_{1}+\cdots+\lambda_{i-1}+1\leq j<k\leq \lambda_{1}+\cdots+\lambda_{i}}x_{j}-x_{k}\right).
\end{align*}
The symmetric group acts (on the right) on $\C[x_{1},\ldots,x_{n}]$ by permutation of the variables, and the complex vector space $V$ generated by the $(\Delta^{\lambda})\cdot \sigma$ with $\sigma \in \sym_{n}$ happens to be an irreducible module over $\sym_{n}$; moreover, all the isomorphism classes of irreducibles can be obtained this way (\emph{cf.} \cite{JK81}).\bigskip

If $\lambda$ is a partition of size $n$, we represent it by its \textbf{Young diagram}: this is the array of $n$ boxes with $\lambda_{1}$ boxes on the first row, $\lambda_{2}$ boxes on the second row, \emph{etc.} 
$$(5,4,2,2) \quad \leftrightarrow \quad \yng(2,2,4,5)\,.$$
The conjugate of a partition $\lambda$ is the partition $\lambda'$ obtained by symmetrizing its Young diagram with respect to the first diagonal; for instance, $(5,4,2,2)'=(4,4,2,2,1)$. In the following, we shall denote by $V^{\lambda}$ the irreducible module over $\sym_{n}$ associated to the polynomial $\Delta^{\lambda'}$; this is the usual convention for encoding the values of the irreducible characters in the algebra of symmetric functions, see \S\ref{symfunc}. The two one-dimensional representations of $\sym_{n}$ correspond to the partitions $(n)$ and $(1,1,\ldots,1)=1^{n}$: the first one gives the trivial representation and the second one is the sign representation.
\bigskip

The Young diagrams yield a simple description of the branching rules of the irreducible modules of $\sym_{n}$ and $\sym_{n+1}$. If $\lambda \in \ym_{n}$ and $\Lambda \in \ym_{n+1}$, we denote $\lambda \nearrow \Lambda$ if $\Lambda$ can be obtained from $\lambda$ by adding a single box to the Young diagram. 
\begin{example} If $\lambda=(5,4,2,2)$, the partitions $\Lambda$ of size $14=13+1$ such that $\lambda \nearrow \Lambda$ are $(6,4,2,2)$, $(5,5,2,2)$, $(5,4,3,2)$ and $(5,4,2,2,1)$. 
\end{example}
\noindent The \textbf{branching rules} of the symmetric groups are the following (see \cite{OV04}, \cite{Zel81}): 
\begin{align*}\forall \lambda \in \ym_{n},&\,\,\mathrm{Ind}_{\sym_{n}}^{\sym_{n+1}}(V^{\lambda})=\sum_{\Lambda \in \ym_{n+1}\,\,|\,\,\lambda \nearrow \Lambda}V^{\Lambda}\\
 \forall \Lambda \in \ym_{n+1},&\,\,\mathrm{Res}_{\sym_{n}}^{\sym_{n+1}}(V^{\Lambda})=\sum_{\lambda \in \ym_{n}\,\,|\,\,\lambda \nearrow \Lambda}V^{\lambda}.
 \end{align*}
Since the restriction of representations preserves the dimension, by applying $n$ times the second branching rule, one sees that
$$\dim \lambda=\dim V^{\lambda}=\card\{\text{standard tableaux of shape }\lambda\},$$
where a standard tableau of shape $\lambda$ is a numbering of the boxes of the Young diagram of $\lambda$ that is strictly increasing along the rows and along the columns.

\begin{example}
The five standard tableaux of shape $(3,2)$ are $$\young(45,123)\quad;\quad\young(35,124)\quad;\quad\young(25,134)\quad;\quad\young(34,125)\quad;\quad\young(24,135)\,.$$ Hence, $\dim (3,2)=5$.
\end{example}
\bigskip

There is also a ``closed'' formula for $\dim \lambda$ known as the \textbf{hook length formula}. If $\oblong=(i,j)$ is a box of the Young diagram of $\lambda$, its hook length is $h(\oblong)=\lambda_{i}+\lambda_{j}'-i-j+1$. For instance, in $(5,4,2,2)$, the hook length of the box $(3,1)$ appearing on the first row and the third column is $4$. Then, 
$$\dim \lambda =\frac{n!}{\prod_{\oblong \in \lambda} h(\oblong)},$$
see for instance \cite[\S1.7]{Mac95}. This can be proved by relating the character theory of the symmetric groups to the theory of symmetric functions that we are going to review in the next paragraph.

\subsection{The Hopf algebra of symmetric functions}\label{symfunc}
A formal power series\footnote{In the following, all the coefficients of the functions and power series are supposed complex, though one could define everything over $\Q$.} $f$ of bounded degree in an infinite sequence of variables $X=\{x_{1},\ldots,x_{n},\ldots\}$ is called a \textbf{symmetric function} if, for any finite permutation $\sigma \in \sym_{\infty}$, 
$$f(x_{\sigma(1)},x_{\sigma(2)},\ldots)=f(x_{1},x_{2},\ldots)=f(X).$$
In other words, for any $i_{1}\neq i_{2}\neq \cdots \neq i_{r}$ and $j_{1}\neq j_{2}\neq \cdots \neq j_{r}$, and any $(a_{1},\ldots,a_{r})\in \N^{r}$,
$$[x_{i_{1}}^{a_{1}}x_{i_{2}}^{a_{2}}\cdots x_{i_{r}}^{a_{r}}](f)=[x_{j_{1}}^{a_{1}}x_{j_{2}}^{a_{2}}\cdots x_{j_{r}}^{a_{r}}](f).$$
The set of symmetric functions will be denoted by $\Sym$; this is an infinite-dimensional complex commutative algebra. A transcendance basis of $\Sym$ consists in the \textbf{power sums}
$$p_{k}(X)=\sum_{i=1}^{\infty} \,(x_{i})^{k}.$$
If $\mu=(\mu_{1},\ldots,\mu_{r})$ is a partition, we shall denote $p_{\mu}=p_{\mu_{1}}\cdots p_{\mu_{r}}$; a linear basis of $\Sym$ is therefore $(p_{\mu})_{\mu}$, where $\mu$ runs over $\ym=\bigsqcup_{n \in \N}\ym_{n}$ (\cite[\S1.2]{Mac95}).\bigskip

Another important basis of $\Sym$ is given by the \textbf{Schur functions}, see \cite[\S1.3]{Mac95}. If $\lambda$ is a partition, $\ell(\lambda)$ will denote its number of parts. Then, given variables $x_{1},\ldots,x_{N}$ with $N \geq \ell(\lambda)$, we define the antisymmetric polynomial of type $\lambda$ by
$$a_{\lambda}(x_{1},x_{2},\ldots,x_{N})=\sum_{\sigma \in \sym_{N}} \eps(\sigma)\, x^{\sigma(\lambda)},$$
where $x^{\sigma(\lambda)}=x_{1}^{\lambda_{\sigma(1)}}x_{2}^{\lambda_{\sigma(2)}}\cdots x_{N}^{\lambda_{\sigma(N)}}$ --- here, we complete $\lambda$ with parts of size $0$ to obtain a sequence of length $N$. Finally, the Schur polynomial is defined as the quotient
$$s_{\lambda,N}(x_{1},x_{2},\ldots,x_{N})=\frac{a_{\lambda+\delta}(x_{1},x_{2},\ldots,x_{N})}{a_{\delta}(x_{1},x_{2},\ldots,x_{N})},$$
where $\delta=(N-1,N-2,\ldots,1,0)$, and 
$$\lambda+\delta=(\lambda_{1}+N-1,\lambda_{2}+N-2,\ldots,\lambda_{N-1}+1,\lambda_{N}).$$ If $M$ is a matrix in $\GL(N,\C)$ with eigenvalues $x_{1},\ldots,x_{N}$, then by Weyl's character formula, $s_{\lambda,N}(x_{1},\ldots,x_{N})$ is the character of the action of $M$ on the irreducible $\GL(N,\C)$-module of highest weight $\lambda$. On the other hand, for any $N$, $$s_{\lambda,N+1}(x_{1},\ldots,x_{N},x_{N+1}=0)=s_{\lambda,N}(x_{1},\ldots,x_{N}),$$ so there is a formal power series $s_{\lambda}(x_{1},x_{2},\ldots)$ that is the inductive limit of the $s_{\lambda,N}$'s. Since these polynomials are defined as quotients of antisymmetric polynomials, $s_{\lambda}$ is always a symmetric function in $\Sym$. Actually, $(s_{\lambda})_{\lambda \in \ym}$ is a linear basis of $\Sym$, and by using the Schur-Weyl duality between $\GL(N,\C)$ and $\sym_{n}$ on $(\C^{N})^{\otimes n}$, one can show that the transition matrix between the two bases $(s_{\lambda})_{\lambda \in \ym}$ and $(p_{\mu})_{\mu \in \ym}$ is given by the tables of characters of the symmetric groups (see \cite[\S1.7]{Mac95}).
\begin{proposition}\label{frobenius}
For $\lambda$ and $\mu$ in $\ym_{n}$, we denote by $\zeta^{\lambda}(\mu)$ the value of the (non-normalized) character of $V^{\lambda}$ on a permutation of cycle type $\mu$. Then, for any partition $\mu$,
$$p_{\mu}(X)=\sum_{\lambda \in \ym_{n}} \zeta^{\lambda}(\mu)\,s_{\lambda}(X).$$
\end{proposition}
\noindent This result is the celebrated \textbf{Frobenius formula}.
\bigskip\bigskip

The aforementioned link between symmetric functions and the representation theory of the symmetric groups gives rise to an isomorphism between two self-adjoint \textbf{Hopf algebras}, see \cite{Zel81}. On the one hand, let us denote by $K(\sym_{n})$ the complex vector space built from the Grothendieck group of the category of representations of $\sym_{n}$; it means that $K(\sym_{n})$ has for linear basis the irreducible modules $V^{\lambda}$ with $\lambda \in \ym_{n}$. The graded vector space $K(\sym)=\bigoplus_{n \in \N}K(\sym_{n})$ is a graded connected Hopf algebra for the following operations:
\begin{align*}\text{product: }&\forall \lambda \in \ym_{p}, \,\,\,\forall \mu \in \ym_{q},\,\,\,V^{\lambda}\times V^{\mu}= \mathrm{Ind}_{\sym_{p}\times \sym_{q}}^{\sym_{p+q}}(V^{\lambda}\otimes V^{\mu}) \,;\\
\text{coproduct: }&\forall \lambda \in \ym_{n},\,\,\,\Delta(V^{\lambda})=\sum_{p+q=n} \mathrm{Res}_{\sym_{p}\times \sym_{q}}^{\sym_{n}}(V^{\lambda})\,;\\
\text{involution: }&\forall \lambda \in \ym_{n},\,\,\,\nu(V^{\lambda})=V^{1^{n}}\otimes_{\C} V^{\lambda}=V^{\lambda'}.
\end{align*}
The involution $\nu$ defined above is equal to the antipode of the Hopf algebra up to sign. Moreover, there is a scalar product on $K(\sym)$ defined by $\scal{V^{\lambda}}{V^{\mu}}=\delta_{\lambda\mu}$, and $K(\sym)$ is a self-adjoint Hopf algebra with respect to this pairing. On the other hand, $\Sym$ has also a structure of self-adjoint Hopf algebra if one sets:
$$\Delta(p_{k})=p_{k}\otimes 1 + 1 \otimes p_{k}\qquad;\qquad \nu(p_{k})=(-1)^{k-1}p_{k}\qquad;\qquad \scal{p_{\mu}}{p_{\rho}}=\delta_{\mu\rho}\,z_{\mu}$$
where $z_{\mu}=\prod_{i\geq 1}i^{m_{i}}\,m_{i}!$ if $\mu=1^{m_{1}}2^{m_{2}}\cdots s^{m_{s}}$. Then, it is shown in \cite{Zel81} that:
\begin{proposition}
The rule $V^{\lambda}\mapsto s_{\lambda}$ provides an isomorphism of self-adjoint graded Hopf algebras between $K(\sym)$ and $\Sym$. With some additional conditions, this isomorphism is unique up to the multiplication by the antipode.
\end{proposition}
\noindent In particular, the Schur functions form an orthonormal basis of $\Sym$, and the character values of the symmetric groups satisfy $\zeta^{\lambda}(\mu)=\scal{s_{\lambda}}{p_{\mu}}$.
\bigskip
\bigskip

In $\Sym$, the structure of Hopf algebra can be related to operations on alphabets. Indeed, given two infinite alphabets $X$ and $Y$, if $\overline{X}=\{-x_{1},-x_{2},\ldots\}$ and $X+Y=\{x_{1},x_{2},\ldots,y_{1},y_{2},\ldots\}$, then
$$(\Delta p_{k})(X,Y)=p_{k}(X)+p_{k}(Y)=p_{k}(X+Y)\qquad;\qquad (\nu( p_{k}))(X)=-p_{k}(\overline{X}).$$
Hence, the coproduct $\Delta$ in $\Sym$ corresponds simply to the addition of alphabets, and if the (formal) difference of two alphabets is defined by $p_{k}(X-Y)=p_{k}(X)-p_{k}(Y)$ for any $k\geq 1$, then for any symmetric function $f$,
$$f(X-Y)=[(1\otimes \nu) \circ \Delta]f(X,\overline{Y}).$$\bigskip

To conclude this paragraph, let us speak about \textbf{specializations}. By a formal specialization of $\Sym$, we just mean a morphism of $\C$-algebras between $\Sym$ and $\C$. Since $\Sym=\C[p_{1},p_{2},\ldots,p_{k},\ldots]$, it amounts to give the images of the $p_{k}$'s. Another way to define a specialization is to give values to the variables of the alphabet $X$; however, not all the specializations of $\Sym$ can be obtained in this way. 
\begin{example} We set
$$X_{q}=\frac{1}{1-[q]}=(1,q,q^{2},q^{3},\ldots).$$
Then, $p_{k}(X_{q})=\frac{1}{1-q^{k}}$ for any $k$, and there is a hook-length formula for the specializations $s_{\lambda}(X_{q})$ of the Schur functions corresponding to this alphabet:
$$s_{\lambda}(X_{q})=\frac{q^{n(\lambda)}}{\prod_{\oblong \in \lambda}1-q^{h(\oblong)}},$$
where $n(\lambda)=\sum_{i=1}^{\ell(\lambda)}(i-1)\lambda_{i}$. 
\end{example}\bigskip

Given two specializations denoted by formal alphabets $A$ and $B$, one can as before construct new specializations $A+B$ and $A-B$ by setting:
$$\forall k\geq 1,\,\,\,p_{k}(A+B)=p_{k}(A)+p_{k}(B)\qquad;\qquad p_{k}(A-B)=p_{k}(A)-p_{k}(B).$$
Even if $A$ and $B$ are true alphabets, the specialization $A-B$ may not come from a true alphabet; hence, differences of alphabets give rise to ``new'' specializations of $\Sym$, and we will see hereafter that this operation is important in the setting of the representation theory of $\sym_{\infty}$. Another important specialization of $\Sym$ that will be used in the next paragraph is the so-called exponential alphabet $E$. This formal alphabet satisfies:
$$p_{1}(E)=1\qquad;\qquad\forall k \geq 2,\,\,\,p_{k}(E)=0.$$
Using the reciprocal of the Frobenius formula, one sees that for any $\lambda \in \ym_{n}$, $s_{\lambda}(E)=\frac{\dim \lambda}{n!}$.

\subsection{Coherent systems of measures and Thoma's theorem for the characters of $\sym_{\infty}$}
 We denote by $\sym_{\infty}$ the set of finite permutations of $\N^{*}=\lle 1,+\infty\rre$; this is also the inductive limit of the $\sym_{n}$'s with respect to the natural embeddings $\sym_{n}\hookrightarrow \sym_{n+k}$. As explained in the introduction, there is no hope to have a complete knowledge of the representation theory of this group: indeed, the existence of factor representations of type II or III implies that a unitary representation of $\sym_{\infty}$ does not in general have a \emph{unique} expansion as a direct integral of factors. However, one can classify the \textbf{characters} of $\sym_{\infty}$, that is to say the functions $\chi : \sym_{\infty} \to \C$ that satisfy the three following conditions:\vspace{2mm}
\begin{enumerate}
\item $\chi$ is constant on conjugacy classes: $\forall \sigma,\tau,\,\,\chi(\tau\sigma\tau^{-1})=\chi(\sigma)$.\vspace{2mm}
\item $\chi$ is normalized so that $\chi(1)=1$. \vspace{2mm}
\item $\chi$ is non-negative definite: $\forall \sigma_{1},\ldots,\sigma_{n},\,\,(\chi(\sigma_{i}\sigma_{j}^{-1}))_{1\leq i,j \leq n}$ is hermitian and non-negative definite.\vspace{2mm}
\end{enumerate} 
Such a character is always the trace of a unitary representation of $\sym_{\infty}$ in a Von Neumann algebra of finite type. If one considers only finite type \emph{factor} representations of $\sym_{\infty}$, then one obtained the \textbf{extremal characters} that satisfy the additional condition:\vspace{2mm}
\begin{enumerate}
\setcounter{enumi}{3}
\item $\chi$ cannot be written as a combination $a\chi_{1}+(1-a)\chi_{2}$ with $\chi_{1}\neq \chi_{2}$ characters and $a \in\, ]0,1[$.\vspace{2mm}
\end{enumerate}
We will denote by $\mathcal{X}(\sym_{\infty})$ the set of characters of $\sym_{\infty}$, and by $\mathcal{X}^{*}(\sym_{\infty})$ the set of extremal characters. Then, $\mathcal{X}(\sym_{\infty})$ is a convex compact set (for the topology of simple convergence), and 
\begin{align*}\mathcal{X}^{*}(\sym_{\infty})&=\mathrm{Extr}(\mathcal{X}(\sym_{\infty}))=\{\text{extremal points of }\mathcal{X}(\sym_{\infty})\}\,;\\
\mathcal{X}(\sym_{\infty})&=\mathscr{P}(\mathcal{X}^{*}(\sym_{\infty}))=\{\text{borelian probability measures on }\mathcal{X}^{*}(\sym_{\infty})\}.
\end{align*}
Hence, in order to understand the character theory of $\sym_{\infty}$, it is sufficient to determine the set of extremal (or irreducible) characters of $\sym_{\infty}$. As explained beautifully in \cite{KOV04,Oko97}, this has been done by E. Thoma (\cite{Tho64}), who relates this problem to the classification of totally positive sequences. We define the \textbf{Thoma simplex} $\Omega$ as the set of pairs $\omega=(\alpha,\beta)$ of non-increasing sequences of non-negative real numbers
$$\alpha=(\alpha_{1}\geq \alpha_{2}\geq \cdots \geq 0)\qquad;\qquad\beta=(\beta_{1}\geq \beta_{2}\geq \cdots \geq 0)$$
such that $\sum_{i=1}^{\infty}\alpha_{i}+\sum_{i=1}^{\infty}\beta_{i}=1-\gamma$ with $\gamma \geq0$. To any point $\omega$ of this infinite dimensional simplex, we associate the formal alphabet $A-B+\gamma E$, where $A=\alpha=\{\alpha_{1},\alpha_{2},\ldots\}$, $B=\overline{\beta}=\{-\beta_{1},-\beta_{2},\ldots\}$, and $E$ is the exponential alphabet. Thus, 
\begin{align*}
p_{1}(A-B+\gamma E)&=p_{1}(\alpha)-p_{1}(\overline{\beta})+\gamma p_{1}(E)=p_{1}(\alpha)+p_{1}(\beta)+\gamma=1\,;\\
p_{k\geq 2}(A-B+\gamma E)&=p_{k}(\alpha)-p_{k}(\overline{\beta})=p_{k}(\alpha)+(-1)^{k-1}p_{k}(\beta)\\
&=\sum_{i=1}^{\infty}(\alpha_{i})^{k}+(-1)^{k-1}\sum_{i=1}^{\infty}(\beta_{i})^{k}.
\end{align*}
\begin{theorem}\label{thomatheorem}
The set of extremal characters $\mathcal{X}^{*}(\sym_{\infty})$ can be identified (homeomorphically) with the Thoma simplex $\Omega$, and the character $\chi^{\omega}$ corresponding to a point $\omega=(\alpha,\beta) \in \Omega$ writes as
$$\chi^{\omega}(\sigma)=p_{\mu}(A-B+\gamma E)$$
if $\sigma \in \sym_{n}$ is of cycle type $\mu \in \ym_{n}$ --- since $p_{1}(A-B+\gamma E)=1$, this definition does not depend on the choice of $n$.
\end{theorem}
\noindent In particular, a remarkable property of extremal characters of $\sym_{\infty}$ is that they factorize on disjoint cycles. Actually, it has been proved that these two conditions are  equivalent, see \cite{Tho64,KV84,Oko97}, and also \cite{GK10} for a proof that relates this property to the notion of exchangeability of (non-commutative) random variables. Assuming only this multiplicativity, the methods of Kerov and Vershik (which are reviewed in Section \ref{largenumbers}) allow to reprove Thoma's theorem. That said, we shall assume in the following that Theorem \ref{thomatheorem} is already known, and use it as the starting point of our new results.\bigskip
\bigskip

Another way of describing $\mathcal{X}(\sym_{\infty})$ and $\mathcal{X}^{*}(\sym_{\infty})$ is by using the notion of \textbf{coherent systems of measures} on $\ym=\bigsqcup_{n \in \N}\ym_{n}$. A function $H$ on $\ym$ is said harmonic if, for any partition $\lambda$,
$$H(\lambda)=\sum_{\lambda \nearrow \Lambda}H(\Lambda).$$
Let us suppose that $H$ is positive and that $H(\emptyset)=1$; we then set $P_{n}(\lambda)=H(\lambda)\times \dim\lambda $. In this setting, it is an easy recurrence to show that on each level $\ym_{n}$ of the Young graph $\ym$,  $P_{n}$ is a probability measure. Indeed, for any $n$,
\begin{align*}
\sum_{\Lambda \in \ym_{n+1}}P_{n+1}(\Lambda)&=\sum_{\Lambda \in \ym_{n+1}} H(\Lambda)\,\dim \Lambda=\sum_{\Lambda \in \ym_{n+1}} \sum_{\lambda\nearrow \Lambda} H(\Lambda)\,\dim\lambda\\
&=\sum_{\lambda \in \ym_{n}} H(\lambda)\,\dim \lambda=\sum_{\lambda \in \ym_{n}} P_{n}(\lambda).
\end{align*}

We call coherent system of probability measures on $\ym$ a sequence of probability measures $(P_{n})_{n \in \N}$ coming from a positive normalized harmonic function on $\ym$.
\begin{proposition}
The coherent systems of probability measures provide a bijection between $\mathcal{X}(\sym_{\infty})$ and the set of positive normalized harmonic functions on $\ym$.
\end{proposition}
\begin{proof}
If $\chi$ is character of $\sym_{\infty}$, its restrictions $\chi_{n}=\chi_{| \sym_{n}}$ are non-negative normalized characters of the symmetric groups, so they are positive linear combinations of the normalized irreducible characters $\chi^{\lambda}$:
$$\forall n,\,\,\,\chi_{|\sym_{n}}=\sum_{\lambda \in \ym_{n}} P_{n,\chi}(\lambda)\,\chi^{\lambda}.$$
Since $(\chi_{|\sym_{n+1}})_{|\sym_{n}}=\chi_{|\sym_{n}}$, the $P_{n,\chi}$'s form a coherent system. Conversely, given a coherent system of probability measures, one can define on each group $\sym_{n}$ a non-negative character by using the formula written above. The coherence of the system implies that these definitions are compatible, whence a well-defined character of $\sym_{\infty}$.
\end{proof}\bigskip

Let us denote by $(\proba_{n,\omega})_{n\in \N}$ the coherent system associated to a point $\omega$ in the Thoma simplex. For any permutation $\sigma \in \sym_{n}\subset \sym_{\infty}$ of cycle type $\mu \in \ym_{n}$,
\begin{align*}
\sum_{\lambda \in \ym_{n}}\proba_{n,\omega}[\lambda]\,\,\chi^{\lambda}(\mu)&=\chi^{\omega}_{|\sym_{n}}(\sigma)=\chi^{\omega}(\sigma)=p_{\mu}(A-B+\gamma E)\\
&=_{(\text{Frobenius formula})} \sum_{\lambda \in \ym_{n}} s_{\lambda}(A-B+\gamma E)\,\zeta^{\lambda}(\mu).
\end{align*}
Since this is true for any permutation, the coherent system of probability measures on $\ym$ associated to $\omega$ is therefore given by the specialization $A-B+\gamma E$ of the Schur functions:
\begin{proposition}\label{csspezialschur}
$\proba_{n,\omega}[\lambda]=(\dim \lambda)\,s_{\lambda}(A-B+\gamma E)$.
\end{proposition}
\begin{example}
Take $\omega=\mathbf{0}=((0,0,\ldots),(0,0,\ldots))$. Then, $\chi^{\mathbf{0}}$ is the regular trace $\tau$ of the infinite symmetric group:
$$\tau(\omega)=\begin{cases}
1 &\text{if }\omega=\id_{\N},\\
0&\text{otherwise}.
\end{cases}$$
The corresponding measures are the so-called \textbf{Plancherel measures} of the symmetric groups:
$$\proba_{n}[\lambda]=\proba_{n,\mathbf{0}}[\lambda]=(\dim\lambda)\,s_{\lambda}(E)=\frac{(\dim \lambda)^{2}}{n!}.$$
The asymptotics of these measures have been studied in \cite{LS77,KV77,Ker93,IO02}. It turns out that this point of the Thoma simplex will be the only one for which our methods won't give a sufficiently precise result (more on this at the end of the paper).
\end{example}\bigskip

Proposition \ref{csspezialschur} is very important from a combinatorial point of view: it enables us to relate the coherent systems of probability measures coming from characters of $\sym_{\infty}$ to very general models of random permutations. The goal of the next section is to explain this link.
\bigskip

\section{Quasi-symmetric functions and a link with random permutations}\label{randompermutation}
In this section, we shall explain how the coherent systems of measures on $\ym$ are related to models of random permutations that can be defined by using generalized riffle shuffles in the sense of \cite{Stan01,Ful02}.

\subsection{The Hopf algebras $\QSym$ and $\FQSym$}\label{fqsym}
To begin with, we have to present two new combinatorial Hopf algebras, namely, $\QSym$ and $\FQSym$. In the following, $X=\{x_{1},x_{2},\ldots\}$ is an infinite ordered alphabet, but this time with non-commuting letters. The algebra $\C\langle X\rangle$ generated by these noncommutative variables has for linear basis the words $w=x_{i_{1}}x_{i_{2}}\ldots x_{i_{r}}$. The algebra of \textbf{free quasi-symmetric functions}, first introduced in \cite{MR95} and redefined in \cite{NCSF6}, is a subalgebra of $\C\langle X\rangle$ that happens to have a natural structure of Hopf algebra and to be related to numerous other combinatorial Hopf algebras. It is defined by using the \textbf{standardization} of words. Given a word $w=x_{i_{1}}x_{i_{2}}\ldots x_{i_{r}}$, let us denote by $\{j_{1} <j_{2}<\cdots <j_{s}\}$ the set of distinct indices $i$ appearing in $w$, and by $m_{1},\ldots,m_{s}$ their multiplicities. In $w$, we replace from left to right the occurrences of $x_{j_{1}}$ by $1,2,\ldots, m_{1}$; then, we replace from left to right the occurrences of $x_{j_{2}}$ by $m_{1}+1,m_{1}+2,\ldots,m_{1}+m_{2}$; \emph{etc.} We thus obtain a permutation of size $n$, called the standardization of the word $w$.
\begin{example}
If $w=x_{4}x_{3}x_{1}x_{5}x_{1}x_{1}x_{3}x_{4}$, then $\mathrm{Std}(w)=64182357$.
\end{example}\bigskip

Now, for any permutation $\sigma \in \sym_{n}$, we define elements $F_{\sigma}$ and $G_{\sigma}$ in $\C\langle X \rangle$ by
$$G_{\sigma}=\sum_{w\,\,|\,\,\mathrm{Std}(w)=\sigma}w\qquad;\qquad F_{\sigma}=G_{\sigma^{-1}}.$$
It turns out that the $F_{\sigma}$'s with $\sigma \in \sym=\bigsqcup_{n \in\N}\sym_{n}$ generate linearly a subalgebra of $\C\langle X \rangle$. For two permutations $\sigma=\sigma(1)\cdots \sigma(p)$ and $\tau=\tau(1)\cdots \tau(q)$, we define their shifted \textbf{shuffle product} $\sigma \cshuffle \tau$ as the set of permutations $\upsilon$ of size $p+q$ such that $\sigma(1),\sigma(2),\ldots, \sigma(p)$ appear in this order in the word of $\upsilon$, and $\tau(1)+p,\tau(2)+p,\ldots,\tau(q)+p$ also appear in this order in the word of $\upsilon$.
\begin{example}
If $\sigma=21$ and $\tau=231$, then 
\begin{align*}
\sigma \cshuffle \tau&= 21 \shuffle 453\\
&\!\!\!=\{21453,24153,24513,24531,42153,42513,42531,45213,45231,45321\}.
\end{align*}

\end{example}
\noindent For all permutations $\sigma,\tau \in \sym$, by sorting the words appearing in the product $F_{\sigma}F_{\tau}$ according to their standardization, one sees readily that
$F_{\sigma}F_{\tau}=\sum_{\upsilon \in \sigma\cshuffle\tau}F_{\upsilon}.$\bigskip

Let us denote by $\FQSym$ the algebra generated by the $F_{\sigma}$'s. The ordered sum of ordered alphabets $X+Y$ defines a coproduct on $\FQSym$: $\Delta f(X,Y)=f(X+Y)$. Endowed with this new operation, $\FQSym$ becomes an Hopf algebra (see \cite[\S3]{NCSF6} and \cite{AS05}) which is neither commutative nor cocommutative. Moreover, $\FQSym$ is self-dual with respect to the pairing
$$\scal{F_{\sigma}}{G_{\tau}}=\delta_{\sigma\tau},$$
and in the basis $(F_{\sigma})_{\sigma \in \sym}$, the coproduct of $\FQSym$ can be written in terms of deconcatenation and standardization. More precisely,
$$\Delta F_{\sigma}=\sum_{\sigma=\tau\cdot \nu}F_{\mathrm{Std}(\tau)}\otimes F_{\mathrm{Std}(\nu)}.$$
\begin{example}
If $\sigma$ is in $\sym_{n}$, one will always obtain $n+1$ terms in $\Delta(F_{\sigma})$. For instance,
$$\Delta(F_{4132})=1\otimes F_{4132}+F_{1}\otimes F_{132}+F_{21}\otimes F_{21}+F_{312}\otimes F_{1}+F_{4132}\otimes 1.$$
\end{example}
\bigskip
\bigskip

If $\sigma\in \sym_{n}$, its \textbf{descents} are the $i$'s in $\lle 1,n-1\rre$ such that $\sigma(i)>\sigma(i+1)$, and its \textbf{recoils} are the descents of $\sigma^{-1}$. For instance, if $\sigma=64182357$, then 
$$D(\sigma)=\{1,2,4\}\qquad;\qquad R(\sigma)=\{3,5,7\}.$$
To a subset $I \subset \lle 1,n-1\rre$, one can associate bijectively a \textbf{composition} of size $n$, that is to say a sequence of positive integers $c=(c_{1},c_{2},\ldots,c_{r})$ such that $c_{1}+\cdots+c_{r}=n$. Indeed, such a sequence is entirely determined by the integers
$$c_{1}<c_{1}+c_{2}<\cdots<c_{1}+\cdots+c_{r-1},$$
and these integers form an arbitrary subset of $\lle 1,n-1\rre$. We will denote by $c(\sigma)$ the composition associated to the descent set of a permutation $\sigma$. It is easy to compute $c(\sigma)$ by writing down the \textbf{ribbon} of $\sigma$, which is the unique ribbon tableau with reading word $\sigma$ and increasing rows and columns. 
\begin{example}
The ribbon of $64182357$ is 
$$\young(6,4,18,:2357)\,.$$
\end{example}
\noindent Then, the parts of $c(\sigma)$ are the lengths of the rows of the ribbon of $\sigma$. In the following, for any integer $n$, we will denote $\comp_{n}$ the set of compositions of size $n$, and if $c\in \comp_{n}$, we will denote $D(c)$ the associated subset of $\lle 1,n-1\rre$ --- it is called the descent set of $c$.
\bigskip
\bigskip

The algebra of \textbf{quasi-symmetric functions} $\QSym$ (see for instance \cite{Ges84} and \cite[\S7.19]{Stan91}) is obtained from $\FQSym$ by making the variables $x_{i}$ commutative. If $X$ is now a commutative alphabet, then it is not hard to convince oneself that $G_{\sigma}$ and $G_{\tau}$ become equal if and only if $\sigma$ and $\tau$ have the same recoils. More precisely, under the morphism $\C\langle X \rangle \to \C[X]$, a free quasi-symmetric function $F_{\sigma}$ becomes $L_{c(\sigma)}$, where for any composition $c \in \comp_{n}$, $L_{c}$ is defined as
$$L_{c}=\sum_{\substack{i_{1}\leq i_{2}\leq \cdots \leq i_{n}\\ i_{j}<i_{j+1}\text{ if }j \in D(c)}} x_{i_{1}}x_{i_{2}}\cdots x_{i_{n}}.$$
A formal power series $f$ in commutative variables $x_{1},x_{2},\ldots$ belongs to the vector space generated by the $L_{c}$'s, $c \in \comp=\bigsqcup_{n \in \N}\comp_{n}$ if and only if it is quasi-symmetric, meaning that for any $i_{1}< i_{2}< \cdots < i_{r}$ and $j_{1}< j_{2}< \cdots < j_{r}$, and any $(a_{1},\ldots,a_{r})\in \N^{r}$,
$$[x_{i_{1}}^{a_{1}}x_{i_{2}}^{a_{2}}\cdots x_{i_{r}}^{a_{r}}](f)=[x_{j_{1}}^{a_{1}}x_{j_{2}}^{a_{2}}\cdots x_{j_{r}}^{a_{r}}](f).$$
The reader should compare this definition of quasi-symmetric functions with the definition of symmetric functions given on page \pageref{symfunc}. The space $\QSym$ of quasi-symmetric functions is stable by product, and in fact, it is again an Hopf algebra; it contains $\Sym$ as a Hopf subalgebra, and it is the quotient of $\FQSym$ by the congruences 
$$F_{\sigma}\equiv F_{\tau} \quad\iff \quad c(\sigma)=c(\tau).$$
The product of two fundamental quasi-symmetric functions $L_{c}$ and $L_{d}$ is not in general multiplicity-free as in $\FQSym$\footnote{This is one the main advantage of doing computations in $\FQSym$.}; however, the formula for a coproduct $\Delta(L_{c})$ remains essentially the same as the one given in $\FQSym$ for $\Delta(F_{\sigma})$. More precisely, for $c\in \comp_{n}$ and $i,j \in \lle 0,n\rre$, let us denote $c_{\lle i+1,j\rre}$ the composition of size $j-i$ that satisfies:
$$D(c_{\lle i+1,j\rre})=\big(D(c)\cap \lle i+1,j-1\rre \big)-i.$$
In terms of ribbons, it amounts to cut the ribbon of $c$ before the box $i+1$ and after the box $j$.
\begin{example}
If $c=(1,1,2,4)$ and $\lle i+1,j\rre=\lle 3,7\rre$, then the ribbon of $c$ is $$\young(~,~,~~,:~~~~)$$
and it is then easy to see that $c_{\lle 3,7\rre}=(2,3)$.
\end{example}
\noindent The translation in $\QSym$ of the formula given previously for a coproduct $\Delta(F_{\sigma})$ is quite obviously:
$$\Delta(L_{c})=\sum_{i=0}^{n} L_{c_{\lle 1,i\rre}}\otimes L_{c_{\lle i+1,n\rre}}.$$
\begin{example}
$$\Delta(L_{1,2,1})=L_{\emptyset}\otimes L_{1,2,1}+L_{1}\otimes L_{2,1}+L_{1,1}\otimes L_{1,1}+L_{1,2}\otimes L_{1}+L_{1,2,1}\otimes L_{\emptyset}.$$
\end{example}
\bigskip

We shall also need the formula for the antipode of $\QSym$. If $c=(c_{1},\ldots,c_{r})$ is a composition, let us denote by $c^{*}$ the composition whose descent set is $(c_{r},c_{r}+c_{r-1},\ldots,c_{r}+\cdots+c_{2})$; this is the reverse composition of $c$. Then, we define the conjugate $\overline{c}$ of $c$ by $D(\overline{c})=\lle 1,n-1\rre \setminus D(c^{*})$. The operation $c \mapsto \overline{c}$ may seem complicate, but in terms of ribbons, it amounts simply to symmetrize $c$ with respect to the first diagonal. 
\begin{example}
The conjugate of $c=(1,1,2,4)$ has for ribbon
$$\young(~,~,~,~~,:~~~)$$
so $\overline{c}=(1,1,1,2,3)$.
\end{example}
\noindent Then, it can be shown that the antipode of $\QSym$ is up to sign equal to $\nu$, where $\nu(L_{c})=L_{\overline{c}}$ for any composition. In $\FQSym$, the antipode of is much more difficult to compute; see \cite[\S5]{AS05}. Fortunately, we shall only need the following property:
 $$\nu(F_{12\ldots n})=F_{n\ldots 21}$$
 which is an immediate consequence\footnote{This partial rule is compatible with the following identity in $\QSym$: for any permutation $\sigma \in \sym_{n}$, $\nu(L_{c(\sigma)})=L_{c(\sigma\omega_{0})}$, where $\omega_{0}$ is the permutation of maximal length in $\sym_{n}$.} of \cite[Theorem 5.4]{AS05}.
\bigskip
\bigskip

To conclude this paragraph, let us explain how the bases of $\Sym$, $\QSym$ and $\FQSym$ are related. If $\sigma \in \sym_{n}$ is a permutation, the Robinson-Schensted-Knuth (RSK) algorithm (\cite{Rob38,Sch61}) provides a pair of standard tableaux $(P(\sigma),Q(\sigma))$ of same shape $\Lambda(\sigma) \in \ym_{n}$, and it is actually a bijection between $\sym_{n}$ and $\bigsqcup_{\lambda \in \ym_{n}}\mathrm{ST}(\lambda)^{2}$. Given a permutation\footnote{Or more generally, a word $w$; then, $P(w)$ is a \emph{semi-standard} tableau whose entries are the letters of $w$.} $\sigma=w_{1}w_{2}\cdots w_{n}$, one constructs the first tableau $P(\sigma)$ as follows:\vspace{2mm}
\begin{enumerate}
\item One reads the word from left to right, and one inserts the letters successively in the standard tableau $P$, starting from $P=\emptyset$.\vspace{2mm}
\item When inserting the letter $a$, one compare $a$ to the entries of the first row. If $a$ is greater than all the entries of the first row, one adds $a$ at the end of the row. Otherwise, if $b$ is the first entry strictly greater than $a$, one replaces $b$ by $a$ in the first row, and one inserts $b$ in the next row (with the same rules).\vspace{2mm}
\end{enumerate}
The tableau $Q(\sigma)$ is constructed by recording the order of creation of the boxes of $P$. The correspondence $\sigma \mapsto (P(\sigma),Q(\sigma))$ has the following additional properties (\cite{Ful97}):\vspace{2mm}
\begin{enumerate}
\item For any permutation $\sigma$, $P(\sigma)=Q(\sigma^{-1})$.\vspace{2mm}
\item The parts of the common shape $\Lambda(\sigma)$ of $P(\sigma)$ and $Q(\sigma)$ are related to the so-called \textbf{Greene invariants} of the permutation (see \emph{e.g.} \cite{LLT02}). Hence, for any $i \leq \ell(\Lambda(\sigma))$,
$$\Lambda_{1}+\Lambda_{2}+\cdots+\Lambda_{i}=\max\{\ell(iw_{1})+\ell(iw_{2})+\cdots+\ell(iw_{i})\}$$
where the maximum is taken on the set of $i$-tuples of disjoint increasing subwords of $\sigma$. In particular, $\Lambda_{1}$ is the length $\ell(\sigma)$ of a longest increasing subword of $\sigma$. One has the same result if one replaces $\Lambda(\sigma)$ by its conjugate and the longest increasing subwords by the longest decreasing subwords.\vspace{2mm}
\item A descent of a standard tableau $T$ is an index $i$ such that $i+1$ appears in a row strictly above the one that contains $i$. This notion allows to associate to $T$ a subset of $\lle 1,n-1\rre$, and therefore a composition $c(T)$. Then, for any permutation $\sigma$, $c(\sigma)=c(Q(\sigma))$.\vspace{2mm}
\end{enumerate}
\begin{example}
If $\sigma=64182357$, the two corresponding tableaux are
$$\young(6,48,12357)\quad\text{and}\quad\young(3,25,14678)\,.$$
The length of the first part is $5$, and it is indeed the length of a longest increasing subword, here, $12357$. The descents of $Q(\sigma)$ are $\{1,2,4\}$, and they are indeed the descents of $\sigma$.
\end{example}\bigskip

Let us denote $\pi: \FQSym \to \QSym$  the quotient map coming from the morphism $\C\langle X\rangle \to \C[X]$. For any $n$, notice that in $\FQSym$,
$$(F_{1})^{n}=\sum_{\sigma \in \sym_{n}}F_{\sigma}.$$
We decompose this sum in blocks according to the shape $\Lambda(\sigma)$ of the two standard tableaux corresponding to $\sigma$:
$$(F_{1})^{n}=\sum_{\lambda \in \ym_{n}} \left(\sum_{\sigma \,\,|\,\,\Lambda(\sigma)=\lambda}F_{\sigma}\right).$$
If one applies $\pi$ to the left-hand side of this expansion, one obtains $(L_{(1)})^{n}$. On the other hand, if one applies $\pi$ to a block $B_{\lambda}$ in the right hand side, one gets 
$$\pi(B_{\lambda})=\sum_{P,Q \in \mathrm{ST}(\lambda)}L_{c(Q)}=(\dim \lambda)\,\sum_{Q \in \mathrm{ST}(\lambda)}L_{c(Q)}.$$
But in fact, after applying $\pi$, both members of the equation are symmetric functions (not just \emph{quasi-symmetric}), and the decomposition is simply the Frobenius formula for $(p_{1})^{n}$:
$$(p_{1})^{n}=\sum_{\lambda \in \ym_{n}}(\dim \lambda)\,s_{\lambda}.$$
In particular, $s_{\lambda}=\sum_{Q \in \mathrm{ST}(\lambda)}L_{c(Q)}$, see \cite[p. 361]{Stan91}. This simple fact allows to lift the coherent systems of measures to systems of probability measures on the symmetric groups $\sym_{n}$.

\subsection{Coherent systems as push-forwards of measures on the symmetric groups}\label{pushforward} 
To do this, we will lift the specializations $A-B+\gamma E$ to specializations of $\FQSym$. First, remark that for any symmetric function $f$, one can write:
$$f(A-B+\gamma E)= \Delta f(A-B,\gamma E)=[(\id \otimes \nu \otimes \id)\circ (\Delta \otimes \id)\circ \Delta]f(A,\overline{B},\gamma E).$$
Suppose that the specializations $A$, $\overline{B}$ and $\gamma E$ are lifted up to $\QSym$; this means that one has morphisms of algebras $A : \QSym \to \C$, $B : \QSym \to \C$ and $\gamma E : \QSym \to \C$, whose restrictions to $\Sym \subset \QSym$ are those defined in Theorem \ref{thomatheorem}. Then, since $\nu$ and $\Delta$ are morphisms of the commutative algebra $\QSym$, the formula above defines a specialization of $\QSym$ that is compatible with the specialization $A-B+\gamma E$ of $\Sym$. Moreover, by composition by the quotient map $\pi: \FQSym \to \QSym$, one gets a specialization of $\FQSym$. Hence, it is sufficient to define $f(A)$, $f(\overline{B})$ and $f(\gamma E)$ for any quasi-symmetric function $f$.\bigskip

For the two first, exactly as in the case of $\Sym$, it amounts to replace the variables $X=\{x_{1},x_{2},\ldots\}$ by either $A=\{\alpha_{1},\alpha_{2},\ldots,\}$ or $\overline{B}=\{\beta_{1},\beta_{2},\ldots\}$. Notice that since $\sum_{i=1}^{\infty}\alpha_{i}+\sum_{i=1}^{\infty}\beta_{i}\leq 1$, for any $f \in \QSym$, the series $f(A)$ and $f(\overline{B})$ converge absolutely. As for the exponential alphabet, in $\FQSym$ or $\QSym$, it is defined by the rule:
$$F_{\sigma}(\gamma E)=L_{c(\sigma)}(\gamma E)=\frac{\gamma^{n}}{n!}.$$
Since a product $F_{\sigma}F_{\tau}$ involves $\binom{|\sigma|+|\tau|}{|\sigma|,|\tau|}$ terms, this rule gives indeed a morphism of algebras. So, the specialization $A-B+\gamma E$ is well-defined on $\QSym$ and $\FQSym$. Moreover, notice that for any permutation $\sigma$, $F_{\sigma}(A)$, $F_{\sigma}(\overline{B})$ and $F_{\sigma}(\gamma E)$ are non-negative. Since $\Delta$ and $\nu$ are positive operators on $\QSym$ with respect to the basis $(L_{c})_{c \in \comp}$, 
$$\forall \omega=(\alpha,\beta)\in \Omega,\,\,\,\forall \sigma \in \sym,\,\,\,F_{\sigma}(A-B+\gamma E)>0.$$
On the other hand, since $F_{1}(A-B+\gamma E)=p_{1}(A-B+\gamma E)=1$,
$$\sum_{\sigma \in \sym_{n}}G_{\sigma}(A-B+\gamma E)=\sum_{\sigma \in \sym_{n}}F_{\sigma}(A-B+\gamma E)=(F_{1}(A-B+\gamma E))^{n}=1.$$
Thus, by using the Hopf algebras $\QSym$ and $\FQSym$, for any point $\omega$ of the Thoma simplex, one can define a (positive) probability measure $\qproba_{n,\omega}$ on the symmetric group $\sym_{n}$ corresponding \emph{via} RSK to the coherent measure $\proba_{n,\omega}$:
\begin{proposition}\label{frompermtopar}
For any $\omega \in \Omega$, and any permutation $\sigma \in \sym$, set $\qproba_{n,\omega}[\sigma]=G_{\sigma}(A-B+\gamma E)$.\vspace{2mm}
\begin{enumerate}
\item For any $n$, $\qproba_{n,\omega}$ is a probability measure on $\sym_{n}$, and for any partition $\lambda$,
$$\proba_{n,\omega}[\lambda]=\Lambda_{*}\qproba_{n,\omega}[\lambda]=\sum_{\sigma\,\,|\,\,\Lambda(\sigma)=\lambda}\qproba_{n,\omega}[\sigma].$$\vspace{2mm}
\item If $c$ is the composition of size $n$ encoding the recoils of $\sigma$, then
$$\qproba_{n,\omega}[\sigma]=\sum_{i+j+k=n}\frac{\gamma^{k}}{k!}\,L_{c_{\lle 1,i\rre}}(\alpha_{1},\alpha_{2},\ldots)\,L_{\overline{c_{\lle i+1,i+j\rre}}}(\beta_{1},\beta_{2},\ldots).$$\vspace{2mm}
\end{enumerate}
\end{proposition}
\begin{proof}
We have seen just before that $\qproba_{n,\omega}[\sigma]=G_{\sigma}(A-B+\gamma E)$ defines a probability measure on $\sym_{n}$. Now, for any partition $\lambda \in \ym_{n}$, we can write:
\begin{align*}
\sum_{\Lambda(\sigma)=\lambda}\qproba_{n,\omega}[\sigma]&=\sum_{\Lambda(\sigma)=\lambda} L_{c(\sigma^{-1})}(A-B+\gamma E)=\sum_{P,Q \in \mathrm{ST}(\lambda)}L_{c(P)}(A-B+\gamma E)\\
&=(\dim \lambda)\,\,s_{\lambda}(A-B+\gamma E)=\proba_{n,\omega}[\lambda].
\end{align*}
Finally, the specialization $A-B+\gamma E$ of $\FQSym$ factors through the quotient map $\pi : \FQSym \to \QSym$, and in this latter algebra, it is not difficult to compute it:
\begin{align*}
\qproba_{n,\omega}[\sigma]&=G_{\sigma}(A-B+\gamma E)=L_{c(\sigma^{-1})}(A-B+\gamma E)\\
&=[(\id \otimes \nu \otimes \id)\circ (\Delta \otimes \id)\circ \Delta]L_{c}(A,\overline{B},\gamma E)\\
&=\sum_{i+j+k=n} L_{c_{\lle 1,i\rre}}(A)\,(\nu L_{c_{\lle i+1,i+j\rre}})(\overline{B})\,L_{c_{\lle i+j+1,n\rre}}(\gamma E)\\
&=\sum_{i+j+k=n}L_{c_{\lle 1,i\rre}}(\alpha)\,L_{\overline{c_{\lle i+1,i+j\rre}}}(\beta)\,\frac{\gamma^{k}}{k!}.
\end{align*}
\end{proof}
\begin{example}
If $\omega=((0,0,\ldots),(0,0,\ldots))$, then $\gamma =1$ and $\qproba_{n,\omega}[\sigma]=\frac{1}{n!}$ for any permutation $\sigma \in \sym_{n}$ (in the triple sum, the only non-zero term corresponds to $i=j=0$ and $k=n$). Thus, one obtains the uniform measures.\bigskip

\noindent Similarly, suppose that $\omega=((1-\gamma,0,0,\ldots),(0,0,\ldots))$ with $\gamma >0$. Then, $L_{c_{\lle 1,i\rre}}(\alpha)$ equals $0$ unless $c_{\lle 1,i\rre}$ has no descents, which is equivalent to ask that the first recoil of $\sigma$ is after $i$. So, if $\mathrm{FR}(\sigma)$ denotes this first recoil, then
$$\mathbb{Q}_{n,\omega}[\sigma]=\sum_{i=0}^{\mathrm{FR}(\sigma)}\frac{(1-\gamma)^{i}\,\gamma^{n-i}}{n-i!}.$$\bigskip

\noindent Suppose now that $\omega=((1/d,1/d,\ldots,1/d,0,0,\ldots),(0,0,\ldots))$, with $d$ terms $1/d$ in $\alpha$. Then, 
$$\qproba_{n,\omega}[\sigma]=L_{c(\sigma)^{-1}}(\alpha)=\!\!\!\!\sum_{\substack{i_{1}\leq i_{2} \leq \ldots\leq i_{n} \leq d\\
i_{j}<i_{j+1}\text{ if }j \in D(c(\sigma^{-1}))}}\!\! \frac{1}{d^{n}}=\frac{1}{d^{n}}\,\binom{n-1+d-\card R(\sigma)}{n}.$$
For the two last examples, exchanging $\alpha$ and $\beta$ amounts to replace $\sigma$ by $\sigma\omega_{0}$, where $\omega_{0}=(n,n-1,\ldots,1)$.\bigskip

\noindent  Finally, suppose that $\omega=(((1-q),(1-q)q,(1-q)q^{2},\ldots),(0,0,\ldots))$. Then,
$$\qproba_{n,\omega}[\sigma]=L_{c(\sigma)^{-1}}(\alpha)=(1-q)^{n}\!\!\!\!\!\!\sum_{\substack{i_{1} \leq \ldots\leq i_{n} \leq d\\
i_{j}<i_{j+1}\text{ if }j \in D(c(\sigma^{-1}))}}\!\!\!\!\! q^{i_{1}+\cdots+i_{n}-n}=\frac{q^{\mathrm{comaj}(\sigma^{-1})}}{\{n!\}_{q}},$$
where $\mathrm{comaj}(\sigma)=\sum_{d \in D(\sigma)}n-d$ and $\{n!\}_{q}=\prod_{i=1}^{n }\frac{1-q^{i}}{1-q}.$
\end{example}

\subsection{Link with the generalized riffle shuffles}\label{riffleshuffle}
A remarkable fact is that the probability measures $\qproba_{n,\omega}$ defined in the previous paragraph can be described practically by random \textbf{generalized riffle shuffles}; see \cite{Ful02}. We fix a point $\omega=(\alpha,\beta)$ of the Thoma simplex and an integer $n$, and we consider a deck of cards ordered from $1$ to $n$. We then construct a random permutation $\sigma \in \sym_{n}$ by using the following algorithm:\vspace{2mm}
\begin{enumerate}
\item First, we split the deck in blocks of size $n_{1}+n_{2}+\cdots+m_{1}+m_{2}+\cdots+l=n$, the sizes of the blocks being chosen randomly according to a multinomial law of parameter $(\alpha,\beta,\gamma)$:
$$\proba[n_{1},\ldots,m_{1},\ldots,l]=\binom{n}{n_{1},\ldots,m_{1},\ldots,l}\,\left(\prod_{i\geq 1}\alpha_{i}^{n_{i}}\right)\,\left(\prod_{j\geq 1}\beta_{j}^{m_{j}}\right)\,\gamma^{l}.$$\vspace{2mm}
\item In the blocks of sizes $m_{1},m_{2},\ldots$, we reverse the order of the cards. In the block of size $l$, we randomize the order of the cards, so that every permutation of the $l$ last cards becomes equiprobable.\vspace{2mm}
\item Finally, we randomly shuffle the blocks with the following rule. If two blocks of size $k_{1}$ and $k_{2}$ are shuffled together, then the first card of the new deck comes from the top of the first block with probability $\frac{k_{1}}{k_{1}+k_{2}}$, and from the top of the second block with probability $\frac{k_{2}}{k_{1}+k_{2}}$; then we start again with the remaining $k_{1}+k_{2}-1$ cards. Notice that this recursive rule is equivalent to ask that all the $\binom{k_{1}+k_{2}}{k_{1},k_{2}}$ shuffles of the blocks are equiprobable.
\end{enumerate}
This algorithm is called in \cite{Ful02} an $(\alpha,\beta,\gamma)$-shuffle.
\begin{proposition}\label{apocalyptica}
The probability of a permutation $\sigma$ obtained by the $\omega$-shuffle presented above is exactly $\qproba_{n,\omega}[\sigma]$.
\end{proposition}
\begin{proof}
Let us give a purely algebraic proof. If $X$ is a formal alphabet, we denote by $\phi_{X} : \FQSym \to \C$ the corresponding specialization, and by $u_{X}$ and $t_{X}$ the maps 
\begin{align*}
u_{X}:\FQSym &\to \FQSym\qquad\qquad;\qquad\qquad t_{X}:\FQSym \to \FQSym\\
F_{\sigma} &\mapsto F_{\sigma}(X)\, F_{12\ldots n}\qquad\qquad\qquad\qquad\quad\,\qquad F_{\sigma}\mapsto F_{\sigma}(X)\,F_{\sigma}.
\end{align*}
Notice that if $X=(\zeta)$ is a single variable, then $u_{(\zeta)}=t_{(\zeta)}$. 
We start from the identity $\qproba_{n,\omega}[\sigma]=G_{\sigma}(A-B+\gamma E)$, and we expand it:
\begin{align*}
\qproba_{n,\omega}[\sigma]&=\phi_{A-B+\gamma E}(G_{\sigma})=m^{3}\circ (\phi_{A}\otimes \phi_{-B}\otimes \phi_{\gamma E})\circ \Delta^{3} (G_{\sigma})\\
&=m^{3}\circ (\phi_{A}\otimes \phi_{\overline{B}}\otimes \phi_{\gamma E})\circ(1\otimes \nu \otimes 1)\circ \Delta^{3} (G_{\sigma})
\end{align*}
The operator applied to $G_{\sigma}$ can be further expanded in 
$$m^{2\infty +1}\circ \left(\left(\otimes_{i}\,\phi_{(\alpha_{i})}\right) \otimes \left(\otimes_{j}\,\phi_{(\beta_{j})}\right) \otimes \phi_{\gamma E}\right)\circ\left(1^{\otimes \infty}\otimes \nu^{\otimes \infty} \otimes 1\right)\circ \Delta^{2\infty +1}$$
where the meaning of the symbols $2\infty+1$ is quite clear from what precedes. Let us denote $I^{2\infty+1} = 1^{\otimes \infty}\otimes \nu^{\otimes \infty}\otimes 1$. Then, $\mathbb{Q}_{n,\omega}[\sigma]$ can be rewritten as:
\begin{align*}
&\scal{\Delta^{2 \infty +1}(F_{12\ldots n})}{\left(\left(\otimes_{i}\,u_{(\alpha_{i})}\right) \otimes \left(\otimes_{j}\,u_{(\beta_{j})}\right) \otimes u_{\gamma E}\right)\circ I^{2\infty+1} \circ \Delta^{2\infty +1} (G_{\sigma})}\\
&=\scal{\Delta^{2 \infty +1}(F_{12\ldots n})}{\big(\big(\!\otimes_{i}\,t_{(\alpha_{i})}\big) \otimes \big(\!\otimes_{j}\,t_{(\beta_{j})}\big) \otimes u_{\gamma E}\big)\circ I^{2\infty+1} \!\circ \Delta^{2\infty +1} (G_{\sigma})}\\
&=\scal{m^{2\infty+1} \circ I^{2\infty+1} \circ \! \big(\big(\!\otimes_{i}\,t_{(\alpha_{i})}^{*}\big) \otimes \big(\!\otimes_{j}\,t_{(\beta_{j})}^{*}\big) \otimes u_{\gamma E}^{*}\big) \! \circ \Delta^{2 \infty +1}(F_{12\ldots n})}{G_{\sigma}}
\end{align*}
Here, we have used extensively the property of self-adjointness of $\FQSym$. Now, we have to check that the operator applied to $F_{12\ldots n}$ encodes the random steps of the $\omega$-shuffle. We describe a probability distribution $\proba$ on $\sym_{n}$ by an element $\sum_{\sigma \in \sym_{n}}\proba[\sigma]\,F_{\sigma} \in \FQSym$, and the effect of our random steps by matrices
$$M=\sum_{\sigma,\tau}c_{\sigma\tau}\,F_{\sigma}\otimes G_{\tau},$$
where the $G$'s are considered as elements of the dual space of $\FQSym$.  We also introduce the operator of uniformization
$$U=\sum_{n\in \N}\sum_{\sigma\in \sym_{n}}\frac{1}{n!}\,F_{\sigma}\otimes G_{12\ldots n}.$$
Notice that for any $n$,
$$u_{\gamma E}^{*}(F_{12\ldots n})=\frac{\gamma^{n}}{n!}\,F_{12\ldots n}=(U \circ t_{(\gamma)})(F_{12\ldots n})\qquad;\qquad F_{12\ldots n}=G_{12\ldots n}.$$
For this reason, we can rewrite the probability $\qproba_{n,\omega}$ as
$$\scal{m^{2\infty+1} \circ \widetilde{I}^{2\infty+1} \circ \big(\big(\otimes_{i}\,t_{(\alpha_{i})}\big) \otimes \big(\otimes_{j}\,t_{(\beta_{j})}\big) \otimes t_{(\gamma)}\big) \circ \Delta^{2 \infty +1}(F_{12\ldots n})}{G_{\sigma}},$$
where $\widetilde{I}^{2\infty+1}=1^{\otimes \infty}\otimes \nu^{\otimes \infty} \otimes U$.\vspace{2mm}
\begin{enumerate}
\item Starting from the state $F_{12\ldots n}$, the first step of our algorithm can be written
$$S_{1}=\left(\left(\otimes_{i}\,t_{(\alpha_{i})}\right) \otimes \left(\otimes_{j}\,t_{(\beta_{j})}\right) \otimes t_{(\gamma)}\right) \circ t_{E}^{\otimes 2\infty+1} \circ \Delta^{2\infty+1} \circ t_{E^{-1}}.$$
Indeed, $\Delta^{2\infty+1}$ enables us to cut the deck in every possible way 
$$n=n_{1}+n_{2}+\cdots+m_{1}+m_{2}+\cdots+l,$$ and the other operators take account of the probabilities (we use the $t_{(\zeta)}$'s for the powers and the $t_{E}$'s for the factorials).\vspace{2mm}
\item The second step is obviously $S_{2}=\widetilde{I}^{2\infty+1}=1^{\otimes \infty}\otimes \nu^{\otimes \infty} \otimes U$.
Here, we use the aforementioned fact that the involution $\nu$ reverses the term $F_{12\ldots n}$ in $F_{n(n-1)\ldots 1}$.\vspace{2mm}
\item Finally, since all the shuffles are equiprobable, by the very definition of the multiplication in $\FQSym$, the third step amounts to apply $$S_{3}=t_{E}\circ m^{2\infty +1} \circ t_{E^{-1}}^{\otimes 2\infty +1}.$$ \vspace{2mm}
\end{enumerate}
Finally, in $S_{3}\circ S_{2}\circ S_{1}$, an operator $t_{E}$ commutes with every degree-preserving operator, so one can simplify them. One gets therefore
$$S_{3}\circ S_{2}\circ S_{1}=m^{2\infty+1} \circ \widetilde{I}^{2\infty+1} \circ \big(\big(\otimes_{i}\,t_{(\alpha_{i})}\big) \otimes \big(\otimes_{j}\,t_{(\beta_{j})}\big) \otimes t_{(\gamma)}\big) \circ \Delta^{2 \infty +1}$$
and the proof is complete.
\end{proof}
\bigskip

\noindent Hence, any point $\omega$ of the Thoma simplex $\Omega$ corresponds to:\vspace{2mm}
\begin{enumerate}
\item an irreducible character $\chi^{\omega}$ of $\sym_{\infty}$ and a coherent system of probability measures $(\proba_{n,\omega})_{n \in \N}$ on the levels $\ym_{n}$ of the Young graph;\vspace{2mm}
\item a system of measures $(\qproba_{n,\omega})_{n\in \N}$ on the symmetric groups $\sym_{n}$ such that $\Lambda_{*}\qproba_{n,\omega}=\proba_{n,\omega}$;\vspace{2mm}
\item and a random algorithm of shuffling that produces permutations $\sigma \in \sym_{n}$ with distribution $\qproba_{n,\omega}$.\vspace{2mm}
\end{enumerate}
It is now time to introduce the tools that shall enable us to perform the asymptotic analysis of these objects.
\bigskip

\section{Observables of Young diagrams and their cumulants}\label{obs}
For the asymptotic analysis of Young diagrams, we are basically going to use a method of moments, but we have to define precisely what we mean by the moments of a partition. We shall essentially follow \cite{IO02}, and \cite{Sni06b,FM10} for the combinatorics of cumulants.

\subsection{Polynomial functions on Young diagrams}\label{polfunc}  If $\lambda$ is a partition, we associate to it two finite alphabets $A(\lambda)$ and $B(\lambda)$ called the \textbf{Frobenius coordinates}:
\begin{align*}
A(\lambda)&=\left\{\lambda_{1}-\frac{1}{2},\lambda_{2}-\frac{3}{2},\ldots,\lambda_{d}-\frac{2d-1}{2}\right\};\\
B(\lambda)&=A(\lambda')=\left\{\lambda_{1}'-\frac{1}{2},\lambda_{2}'-\frac{3}{2},\ldots,\lambda_{d}'-\frac{2d-1}{2}\right\}.
\end{align*}
Here, $d$ is the size of the diagonal of $\lambda$, that is to say the number of boxes of the Young diagram of $\lambda$ that belong to the first diagonal.
\begin{example}
If $\lambda=(6,3,2,2)$, its Young diagram is
$$\yng(2,2,3,6)$$
and the Frobenius coordinates are $A=\{\frac{11}{2},\frac{3}{2}\}$ and $B=\{\frac{7}{2},\frac{5}{2}\}$ --- they are easily computed by drawing the diagonal and measuring the rows and columns starting from the diagonal.
\end{example}
\noindent The \textbf{Frobenius moments} of $\lambda$ are the power sums
$$p_{k}(\lambda)=p_{k}\left(A(\lambda)-\overline{B}(\lambda)\right)=\left(\sum_{i=1}^{d}a_{i}^{k}\right)+(-1)^{k-1}\left(\sum_{i=1}^{d}b_{i}^{k}\right).$$
For instance, $p_{1}(\lambda)=|\lambda|$ for any partition. The functions $p_{k}$ generate freely over $\C$ an algebra called the algebra of \textbf{observables of diagrams}, and denoted $\obs$ (\emph{cf.} \cite{KO94,IO02}). This algebra is isomorphic to $\Sym$, and it is naturally graded by $\deg p_{k}=k$. It does not contain all the polynomial functions of the coordinates of the partitions, but it will prove sufficiently large\footnote{However, we should mention that there is a bigger algebra of observables of diagrams that contains $\obs$, is isomorphic to $\QSym$ and allows to decompose the observables of $\obs$ in combinatorial summands; see in particular \cite{DFS10}.} in our setting.
\bigskip
\bigskip

We now define new functions on Young diagrams by renormalizing the values of the irreducible characters of the symmetric groups. Recall that for any partitions $\lambda,\mu$ of the same size, $\chi^{\lambda}(\mu)$ denotes the value of the \emph{normalized} irreducible character of $\sym_{n}$ of label $\lambda$ on a permutation of cycle type $\mu$; thus, $\zeta^{\lambda}(\mu)=(\dim \lambda)\,\chi^{\lambda}(\mu)$. Then, for any partitions $\lambda,\mu$, we set:
$$\varSigma_{\mu}(\lambda)=\begin{cases}
n(n-1)\ldots (n-|\mu|+1)\,\chi^{\lambda}(\mu \sqcup 1^{n-|\mu|}) &\text{if }n=|\lambda|\geq |\mu|,\\
0&\text{otherwise}.
\end{cases}$$
For convenience, we will denote $n(n-1)\cdots (n-k+1)=n^{\downarrow k}$ for any $n$ and $k$ --- this is the Pochhammer symbol for falling factorials. Thus, $\varSigma_{\mu}(\lambda)=|\lambda|^{\downarrow |\mu|}\,\chi^{\lambda}(\mu \sqcup 1^{|\lambda|-|\mu|})$. An essential result exposed in details in \cite{IO02} is the following:
\begin{theorem}\label{ivanovolshanski}
Viewed as a function on Young diagrams, $\varSigma_{\mu}$ is an element of $\obs$, and is of degree $|\mu|$. Moreover,
$$\forall \mu,\,\,\,\varSigma_{\mu}=p_{\mu}+(\text{terms of degree at most } |\mu|-1),$$
see \cite[Proposition 3.4]{IO02}.
\end{theorem}
\begin{example}
Actually, there is an explicit formula for the $\varSigma_{k}$'s that is due to Wassermann (\cite{Was81}, \cite[Proposition 3.3]{IO02}). Hence, one can compute $\varSigma_{1}=p_{1}$, $\varSigma_{2}=p_{2}$, $\varSigma_{3}=p_{3}-\frac{3}{2}p_{11}+\frac{5}{4}p_{1}$, $\varSigma_{4}=p_{4}-4p_{21}+\frac{11}{2}p_{2}$.
\end{example}
\noindent Theorem \ref{ivanovolshanski} implies in particular that the symbols $\varSigma_{\mu}$ factorize in higher degree, that is to say that for any partitions $\mu$ and $\nu$,
$$\varSigma_{\mu}\,\varSigma_{\nu}=\varSigma_{\mu \sqcup \nu}+(\text{terms of degree at most }|\mu|+|\nu|-1).$$
The next paragraph is devoted to a better description of the products $\varSigma_{\mu}\,\varSigma_{\nu}$; we shall rely on the formalism of partial permutations (\emph{cf.} \cite{IK99}). To conclude this paragraph, let us compute the expectation of a symbol $\varSigma_{\mu}(\lambda)$ when $\lambda$ is chosen randomly according to a measure $\proba_{n,\omega}$ (we suppose that $n \geq |\mu|$):
\begin{align*}\esper_{n,\omega}[\varSigma_{\mu}(\lambda)]&=n^{\downarrow |\mu|} \, \sum_{\lambda \in \ym_{n}} \proba_{n,\omega}[\lambda]\,\,\chi^{\lambda}(\mu \sqcup 1^{n-|\mu|})\\
&=n^{\downarrow |\mu|} \, \sum_{\lambda \in \ym_{n}} s_{\lambda}(A-B+\gamma E)\,\,\zeta^{\lambda}(\mu \sqcup 1^{n-|\mu|})\\
&=_{(\text{Frobenius formula})}n^{\downarrow |\mu|}\,p_{\mu \sqcup 1^{n-|\mu|}}(A-B+\gamma E)\\
&=n^{\downarrow |\mu|}\,p_{\mu}(A-B+\gamma E).\end{align*}
Theorem \ref{ivanovolshanski} implies that $(\varSigma_{\mu})_{\mu \in \ym}$ is a linear basis of $\obs$; hence, by using this simple formula, we will be able to compute the expectation of \emph{any} observable of diagrams under a coherent measure $\proba_{n,\omega}$. In particular, we shall compute in Sections \ref{largenumbers} and \ref{clt} the joint moments of the Frobenius observables $p_{k}$; they will provide information on the geometry of the random partitions under the probability measures $\proba_{n,\omega}$.

\subsection{Ivanov-Kerov algebra of partial permutations}
If $\sigma \in \sym_{\infty}$, we call admissible set for $\sigma$ a finite subset $S \subset \N^{*}$ such that if $x \notin S$, then $\sigma(x)=x$. A \textbf{partial permutation} is a pair $(\sigma,S)$ with $S$ admissible set for $\sigma$. Said in other words, a partial permutation is a subset $S \subset\N^{*}$ and a permutation $\sigma \in \sym(S)$. One can multiply these objects by the rule
$$(\sigma,S)\times(\tau,T)=(\sigma\tau,S\cup T).$$
We denote $\deg (\sigma,S)=\card S$, and we construct a graded algebra $\alg_{\infty}$ by looking at all the formal linear combinations of partial permutations with bounded degree. For any $n$, this algebra $\alg_{\infty}$ can be projected onto the group algebra $\C\sym_{n}$:
\begin{align*}
\pi_{n} : \alg_{\infty}&\to \C\sym_{n}\\
(\sigma,S)&\mapsto \begin{cases}
\sigma &\text{if }S \subset \lle 1,n\rre,\\
0 &\text{otherwise}.
\end{cases}
\end{align*}
There is a natural action of $\sym_{\infty}$ on $\alg_{\infty}$: on the basis elements $(\sigma,S)$, it is defined by
$$(\sigma,S)^{\tau}=(\tau\sigma\tau^{-1},\tau(S)),$$
and it is compatible with the product of partial permutations. Moreover, two partial permutations $(\sigma,S)$ and $(\tau,T)$ are conjugated if and only if $\card S=\card T$, and the cyclic type of $\sigma_{|S}$ is the same as the cyclic type of $\tau_{|T}$. Thus, the invariant algebra $\blg_{\infty}=(\alg_{\infty})^{\sym_{\infty}}$ (called the \textbf{Ivanov-Kerov algebra}) has a linear basis labelled by partitions of arbitrary size:
$$\blg_{\infty}=\mathrm{Vect}((A_{\lambda})_{\lambda \in \ym}),\quad\text{with }A_{\mu}=\sum_{|S|=|\mu|}\sum_{\substack{\sigma \in \sym(S)\\ t(\sigma)=\mu}}(\sigma,S).$$
Actually, we would rather use a renormalized version of this basis, defined as follows:
$$\varSigma_{\mu}=\sum\left[(a_{11},\ldots,a_{1\mu_{1}})(a_{21},\ldots,a_{2\mu_{2}})\cdots(a_{r1},\ldots,a_{r\mu_{r}}), \{a_{ij} \}_{\substack{1 \leq i \leq \ell(\mu) \\ 1\leq j \leq \mu_{i}}}\right]$$
where the sum is taken over injective $\N^{*}$-valued functions $i,j \mapsto a_{ij}$ with $1 \leq i \leq \ell(\mu)$ and $1\leq j \leq \mu_{i}$.
\begin{proposition}\label{fringe}
The family $(\varSigma_{\mu})_{\mu \in \ym}$ is a linear basis of $\blg_{\infty}$, and for any $n$ and any partition $\lambda \in \ym_{n}$,
$$\varSigma_{\mu}(\lambda)=\chi^{\lambda}(\pi_{n}(\varSigma_{\mu})).$$
\end{proposition}
\begin{proof}
Here, an irreducible character $\chi^{\lambda}$ is extended by linearity from the group $\sym_{n}$ to the algebra $\C\sym_{n}$. It is not hard to see that $\varSigma_{\mu}$ is a multiple of $A_{\mu}$ for any partition $\mu$; hence, $(\varSigma_{\mu})_{\mu}$ is indeed a linear basis of $\blg_{\infty}$. Now, for any $n$, $\pi_{n}(\varSigma_{\mu})$ consists in $n^{\downarrow |\mu|}$ elements of $\sym_{n}$ of cycle type $\mu \sqcup 1^{n-|\mu|}$, corresponding to all the possible $|\mu|$-arrangements of the $a_{ij}$'s. Hence,
$$\chi^{\lambda}(\pi_{n}(\varSigma_{\mu}))=n^{\downarrow |\mu|}\,\chi^{\lambda}(\mu \sqcup 1^{n-|\mu|})=\varSigma_{\mu}(\lambda).$$
\end{proof}\bigskip

Now, Proposition \ref{fringe} allows to compute a product of two symbols $\varSigma$ in $\blg_{\infty}$ instead of $\obs$. Indeed, for any partitions $\lambda,\mu,\nu$,
\begin{align*}
\varSigma_{\mu}(\lambda) \,\varSigma_{\nu}(\lambda)&=\chi^{\lambda}(\pi_{n}(\varSigma_{\mu}))\,\,\chi^{\lambda}(\pi_{n}(\varSigma_{\nu}))=\chi^{\lambda}(\pi_{n}(\varSigma_{\mu}) \, \pi_{n}(\varSigma_{\nu}))\\
&=\chi^{\lambda}(\pi_{n}(\varSigma_{\mu}\times_{\blg_{\infty}} \varSigma_{\nu}))=(\varSigma_{\mu}\times_{\blg_{\infty}} \varSigma_{\nu})(\lambda)
\end{align*}
where we used the fact that for any elements $a,b$ of the center $Z(\C\sym_{n})$ of the group algebra, and any normalized irreducible character $\chi^{\lambda}$ of $\sym_{n}$, $\chi^{\lambda}(ab)=\chi^{\lambda}(a)\,\chi^{\lambda}(b)$ --- this is true for any finite group. With this new point of view, the combinatorics of the symbols $\varSigma$ has been investigated by P. \'Sniady in \cite{Sni06a,Sni06b}, and later by V. F\'eray and the author in \cite{FM10}. Thus, one can give a semi-explicit description of a product $\varSigma_{\mu}\,\varSigma_{\nu}$ by using \textbf{partial matchings}. To this purpose, let us introduce some additional notations. We denote by $\arra_{k}$ the set of $k$-arrangements of integers in $\N^{*}$, and if $A=(a_{1},\ldots,a_{k})\in \arra_{k}$, we denote by $C(A)$ the corresponding $k$-cycle, viewed as a partial permutation with support $\{a_{1},\ldots,a_{k}\}$. Then, $\varSigma_{\mu}$ can be rewritten as
$$\varSigma_{\mu}=\sum_{\substack{\forall i,\,\,A_{i} \in \arra_{\mu_{i}} \\ \forall i \neq j,\,\,A_{i} \cap A_{j}=\emptyset}}C(A_{1})\,C(A_{2})\,\cdots \,C(A_{\ell(\mu)}),$$
and given two partitions $\mu$ and $\nu$,
$$\varSigma_{\mu}\,\varSigma_{\nu}=\!\!\sum_{\substack{\forall i,\,\,A_{i} \in \arra_{\mu_{i}} \\ \forall i \neq j,\,\,A_{i} \cap A_{j}=\emptyset}}\,\,\,\sum_{\substack{\forall k,\,\,B_{k} \in \arra_{\nu_{k}} \\ \forall k \neq l,\,\,B_{k} \cap B_{l}=\emptyset}} \!\!C(A_{1})\,\cdots \,C(A_{\ell(\mu)})\,C(B_{1})\,\cdots \,C(B_{\ell(\nu)}).$$
To any sequences $(a_{ij})_{i,j}$ and $(b_{kl})_{k,l}$ appearing in this expression, we associate a partial matching of the set $I_{A}$ of the $(i,j)$'s with the set $I_{B}$ of the $(k,l)$'s: 
$$(i,j) \sim (k,l)\text{ if }a_{ij}=b_{kl}.$$
Then, given a partial matching $M$ of $I_{A}$ with $I_{B}$, if we sum the products 
$$C(A_{1})\,\cdots \,C(A_{\ell(\mu)})\,C(B_{1})\,\cdots \,C(B_{\ell(\nu)})$$
corresponding to this partial matching, we obtain a term $\varSigma_{\rho(M,\mu,\nu)}$: indeed, if one knows which of the $a_{ij}$'s are equal to which of the $b_{kl}$'s, then one can rewrite the product as a product of \emph{disjoint} cycles in a way that does not depend on the values of the $a$' and the $b$'s (\emph{cf.} \cite[\S3.4]{FM10}). Thus: 
\begin{lemma}\label{matching}
For any partition $\mu$ and $\nu$,
$$\varSigma_{\mu}\,\varSigma_{\nu}=\sum_{M} \varSigma_{\rho(M,\mu,\nu)}$$
where the sum runs over all the partial matchings $M$ of $I_{A}$ with $I_{B}$.
\end{lemma}
\noindent We refer again to \cite{FM10} for more details on this combinatorial lemma. The reader should also notice that the degree of a term $\varSigma_{\rho(M,\mu,\nu)}$ corresponding to a partial matching $M$ is $|\mu|+|\nu|-|M|$, where $|M|$ is the size of the partial matching, that is to say the number of pairs $(a \in I_{A},b \in I_{B})$ in $M$.
\begin{example}
Let us compute the term of higher degree and the subdominant term of a product $\varSigma_{\mu}\,\varSigma_{\nu}$.\vspace{2mm}
\begin{enumerate}
\setcounter{enumi}{-1}
\item If $\deg \varSigma_{\rho(M,\mu,\nu)}=|\mu|+|\nu|$, then $|M|=0$, and $M=\emptyset$. The cycles corresponding to the empty matching are disjoint ones; thus, the only term of degree $|\mu|+|\nu|$ in $\varSigma_{\mu}\,\varSigma_{\nu}$ is indeed $\varSigma_{\mu \sqcup \nu}$. \vspace{2mm}
\item If $\deg \varSigma_{\rho(M,\mu,\nu)}=|\mu|+|\nu|-1$, then $|M|=1$ and $M$ contains only one pair $(a,b)$. Suppose that $a$ is in a cycle $C(A_{i})$ of length $\mu_{i}$, and that $b$ is in a cycle $C(B_{j})$ of length $\nu_{j}$. Then, $C(A_{i})$ and $C(B_{j})$ intersect in exactly one point, so their product is a $(\mu_{i}+\nu_{j}-1)$-cycle; and all the other cycles $C(A_{i'})$ or $C(B_{j'})$ stay disjoint. Thus, such a matching yields a term
$$\varSigma_{(\mu\sqcup \nu \sqcup \{\mu_{i}+\nu_{j}-1\})\setminus \{\mu_{i},\nu_{j}\} },$$
and one has to sum over all possible matchings of size $1$ to get the subdominant term of $\varSigma_{\mu}\,\varSigma_{\nu}$.
\end{enumerate}
So, for any partitions $\mu=(\mu_{1},\ldots,\mu_{r})$ and $\nu=(\nu_{1},\ldots,\nu_{s})$,
$$\varSigma_{\mu}\,\varSigma_{\nu}=\varSigma_{\mu \sqcup \nu}+\sum_{i=1}^{r}\sum_{j=1}^{s} \mu_{i}\nu_{j}\,\varSigma_{(\mu\sqcup \nu \sqcup \{\mu_{i}+\nu_{j}-1\})\setminus \{\mu_{i},\nu_{j}\} }$$
plus an observable of degree at most $|\mu|+|\nu|-2$. In particular, for two symbols $\varSigma_{l}$ and $\varSigma_{m}$ corresponding to cycles,
$$\varSigma_{l}\,\varSigma_{m}=\varSigma_{l,m}+lm\,\varSigma_{l+m-1}+(\text{observable of degree at most }l+m-2).$$
\end{example}

\subsection{\'Sniady's theory of cumulants of observables}
In order to prove results of asymptotic gaussian behaviour in the setting of the representation theory of $\sym_{\infty}$, we shall use in Section \ref{clt} \'Sniady's theory of \textbf{cumulants of observables} (\emph{cf.} \cite{Sni06a,Sni06b}), though in a setting that is slightly different from the setting of his papers, and with another gradation on the algebra $\obs$. If $X_{1},\ldots,X_{r}$ are real random variables defined on a common probability space, we recall that their \textbf{joint cumulant} is defined by:
$$k(X_{1},X_{2},\ldots,X_{r})=\left.\frac{\partial^{r}}{\partial t_{1}\partial t_{2}\cdots \partial t_{r}} \right|_{t=0}\log \esper\left[\exp(t_{1}X_{1}+t_{2}X_{2}+\cdots+t_{r}X_{r})\right].$$
In particular, $k(X)=\esper [X]$ is the expectation and $k(X,Y)=\esper[XY]-\esper[X]\,\esper[Y]$ is the covariance. With the definition given above, it is obvious that a gaussian vector $X=(X_{1},\ldots,X_{n})$ is characterized by the vanishing of all the joint cumulants of the coordinates of order $r\geq 3$. On the other hand, one can give an equivalent combinatorial definition of the joint cumulants. For any integer $n$, we denote by $\setpart_{n}$ the set of \textbf{set partitions} of $\lle 1,n\rre$, and if $\pi=\pi_{1}\sqcup \pi_{2}\sqcup\cdots \sqcup \pi_{s}$ is such a set partition, we denote $\ell(\pi)=s$ and $\mu(\pi)=(-1)^{\ell(\pi)-1}\,(\ell(\pi)-1)!$. Then,
$$k(X_{1},X_{2},\ldots,X_{r})=\sum_{\pi \in \setpart_{r}} \mu(\pi) \,\left(\prod_{i=1}^{\ell(\pi)}\,\esper\!\left[\prod_{j \in \pi_{i}} X_{j}\right]\right).$$
\begin{example}
For instance, when $r=3$, one has to sum over the $5$ set partitions of $\setpart_{3}$, and $k(X,Y,Z)$ is equal to
$$\esper[XYZ]-\esper[XY]\,\esper[Z]-\esper[XZ]\,\esper[Y]-\esper[YZ]\,\esper[X]+2\,\esper[X]\,\esper[Y]\,\esper[Z].$$
\end{example}
\bigskip
\bigskip

Now, in the algebra of observables of diagrams $\obs$, we can in fact define three kind of cumulants. First, if $\proba$ is a probability law on partitions, then any $X \in \obs$ can be considered as a random variable $X(\lambda)$, and one can define the joint cumulant of observables $X_{1},\ldots,X_{r}$. That said, one can define on $\obs \simeq \blg_{\infty}$ a new product $\bullet$, called \textbf{disjoint product}, and given by the following rule:
$$\forall \mu,\nu,\,\,\,\varSigma_{\mu}\bullet \varSigma_{\nu}=\varSigma_{\mu\sqcup \nu}.$$
This new product is the restriction to $\blg_{\infty}$ of the disjoint product of partial permutations:
$$(\sigma,S)\bullet (\tau,T)=\begin{cases}(\sigma\tau,S\sqcup T)&\text{if }S\cap T=\emptyset,\\
0&\text{otherwise}.
\end{cases}$$
We will denote by $\obs^{\bullet}$ the algebra of observables endowed with this new product. The \textbf{disjoint cumulant} of some observables $X_{1},\ldots,X_{r}$ is then defined exactly in the same way as for the standard cumulant:
$$k^{\bullet}(X_{1},X_{2},\ldots,X_{r})=\sum_{\pi \in \setpart_{r}} \mu(\pi) \,\left(\prod_{i\leq \ell(\pi)}\,\esper\!\left[\prod_{j \in \pi_{i}}^{\bullet} X_{j}\right]\right).$$
Finally, we define the \textbf{identity cumulant} of observables $X_{1},\ldots,X_{r}$ as the new observable 
$$k^{\id} (X_{1},X_{2},\ldots,X_{r})=\sum_{\pi \in \setpart_{r}} \mu(\pi) \,\left(\prod_{i \leq \ell(\pi)}^{\bullet}\,\left[\prod_{j \in \pi_{i}} X_{j}\right]\right).$$
The three kind of cumulants are related by the following identity, that can be understood as a computation of expectation \emph{via} conditioning, see \cite{Bri69} and \cite[Lemma 4.8]{Sni06b}:
$$k(X_{1},\ldots,X_{r})=\sum_{\pi \in \setpart_{r}} k^{\bullet}\left(k^{\id}(X_{j \in \pi_{1}}),k^{\id}(X_{j \in \pi_{2}}),\ldots,k^{\id}(X_{j \in \pi_{\ell(\pi)}})\right).$$\bigskip
\bigskip

The expansion of an identity cumulant $k^{\id}(\varSigma_{l_{1}},\ldots,\varSigma_{l_{r}})$ in partial permutations has been computed in \cite[Lemma 19]{FM10}. Given arrangements $A_{1},\ldots,A_{r}$, we define a relation $\sim$ by $i \sim j \iff A_{i}\cap A_{j}\neq \emptyset$. The transitive cloture of $\sim$ is an equivalence relation, and we denote by $\pi(A_{1},\ldots,A_{r})$ the corresponding set partition in $\setpart_{r}$. 
\begin{lemma}\label{identitycumulant}
For any $l_{1},\ldots,l_{r}$,
$$k^{\id}(\varSigma_{l_{1}},\ldots,\varSigma_{l_{r}})=\sum_{\substack{\forall i,\,\,A_{i} \in \arra_{l_{i} } \\ \pi(A_{1},\ldots,A_{r})=\lle 1,r\rre }} C(A_{1})\,C(A_{2})\,\cdots\,C(A_{r}).$$
\end{lemma}
\begin{proof}
The three kind of cumulants can be defined inductively by the relations:
\begin{align*}
X_{1}X_{2}\cdots X_{r}&=\sum_{\pi \in \setpart_{r}} k^{\id}(X_{j \in \pi_{1}})\bullet k^{\id}(X_{j \in \pi_{2}})\bullet \cdots \bullet k^{\id}(X_{j \in \pi_{\ell(\pi)}})\,;\\
\esper[X_{1}X_{2}\cdots X_{r}]&=\sum_{\pi \in \setpart_{r}} k(X_{j \in \pi_{1}})\, k(X_{j \in \pi_{2}})\,\cdots \, k(X_{j \in \pi_{\ell(\pi)}})\,;\\
\esper[X_{1}\bullet X_{2}\bullet \cdots \bullet X_{r}]&=\sum_{\pi \in \setpart_{r}} k^{\bullet}(X_{j \in \pi_{1}})\, k^{\bullet}(X_{j \in \pi_{2}})\,\cdots \, k^{\bullet}(X_{j \in \pi_{\ell(\pi)}}).
\end{align*}
In particular, the first relation enables us to prove Proposition \ref{identitycumulant} by induction on $r$, see \cite[Lemma 19]{FM10}.\end{proof}
\bigskip

Now, the degree in $\obs \simeq \blg_{\infty}$ is the restriction of the degree $\deg(\sigma,S)=|S|$ of the algebra of partial permutations $\alg_{\infty}$. However, in any product of cycles appearing in the expansion of $k^{\id} (\varSigma_{l_{1}},\ldots,\varSigma_{l_{r}})$, the joint support of the cycles $C(A_{1}),\ldots,C(A_{r})$ has for cardinality at most $l_{1}+l_{2}+\cdots+l_{r}-(r-1)$ --- this is imposed by the condition $\pi(A_{1},\ldots,A_{r})=\lle 1,r\rre$. Hence, Proposition \ref{identitycumulant} implies that
$$\forall l_{1},\ldots,l_{r},\,\,\,\deg k^{\id}(\varSigma_{l_{1}},\ldots,\varSigma_{l_{r}}) \leq l_{1}+\cdots+l_{r}-(r-1).$$
This leads to the following estimate, that will be crucial in the proof of our central limit theorem for the irreducible characters of $\sym_{\infty}$.
\begin{theorem}\label{estimate}
For $\omega \in \Omega$, we denote by $k_{n,\omega}$ and $k_{n,\omega}^{\bullet}$ the cumulant and the disjoint cumulant of observables with respect to the expectation $\esper_{n,\omega}$. For any observables $X_{1},\ldots,X_{r}$, 
\begin{align*}k_{n,\omega}^{\bullet}(X_{1},\ldots,X_{r})&=O\left(n^{\deg X_{1} +\cdots+\deg X_{r}-r+1}\right);\\
k_{n,\omega}(X_{1},\ldots,X_{r})&=O\left(n^{\deg X_{1} +\cdots+\deg X_{r}-r+1}\right). 
\end{align*}
\end{theorem}
\begin{proof}
Because of the multilinearity of cumulants and of some compatibility relation with respect to the products (see \cite[Theorem 4.4]{Sni06b}), it is sufficient to prove these estimates for an algebraic basis of $\obs^{\bullet}$ or $\obs$, so we will suppose in the following that $X_{i}=\varSigma_{l_{i}}$ with positive integers $l_{1},\ldots,l_{r}\geq 1$. Now, notice that for any partition $\mu$,
$$\esper_{n,\omega}[\varSigma_{\mu}]=p_{\mu}(A-B+\gamma E)\,\,n^{\downarrow |\mu|}=p_{\mu}(A-B+\gamma E)\,\,\esper_{n,\omega}[\varSigma_{1^{|\mu|}}].$$
Thus, one can rewrite a disjoint cumulant $k_{n,\omega}^{\bullet}(\varSigma_{l_{1}},\ldots,\varSigma_{l_{r}})$ as follows:
\begin{align*}
k_{n,\omega}^{\bullet}(\varSigma_{l_{1}},\ldots,\varSigma_{l_{r}})&=\sum_{\pi \in \setpart_{r}}\mu(\pi)\,\left(\prod_{i\leq \ell(\pi)}\,\esper_{n,\omega}\!\left[\prod_{j \in \pi_{i}}^{\bullet} \varSigma_{l_{j}}\right]\right)\\
&\!\!\!\!\!\!\!\!\!\!\!\!=p_{l_{1},\ldots,l_{r}}(A-B+\gamma E) \,\,\sum_{\pi \in \setpart_{r}}\mu(\pi)\,\left(\prod_{i\leq \ell(\pi)}\,\esper_{n,\omega}\!\left[\prod_{j \in \pi_{i}}^{\bullet} \varSigma_{1^{l_{j}}}\right]\right)\\
&\!\!\!\!\!\!\!\!\!\!\!\!=p_{l_{1},\ldots,l_{r}}(A-B+\gamma E)\,\,k_{n,\omega}^{\bullet}(\varSigma_{1^{l_{1}}},\ldots,\varSigma_{1^{l_{r}}}).
\end{align*}
So, in order to prove the estimate for disjoint cumulants, it is sufficient to treat the case of observables $\varSigma_{1^{l}}$. However, for these observables, the expectation $\esper_{n,\omega}[\varSigma_{1^{l}}]=n^{\downarrow l}$ does not depend on $\omega$, and the estimate is simply an algebraic result on the degree of some polynomial combinations of falling factorials of $n$. It has been proved in \cite[Lemma 4.8]{Sni06b} that:
$$k^{\bullet}_{n}(\varSigma_{1^{l_{1}}},\ldots, \varSigma_{1^{l_{r}}})=O\left(n^{l_{1}+\cdots+l_{r}-r+1}\right).$$
Thus, the proof is complete for the disjoint cumulants. As for the standard cumulants, one uses the aforementioned formula:
$$k_{n,\omega}(X_{1},\ldots,X_{r})=\sum_{\pi \in \setpart_{r}} k^{\bullet}_{n,\omega}\left(k^{\id}(X_{j \in \pi_{1}}),k^{\id}(X_{j \in \pi_{2}}),\ldots,k^{\id}(X_{j \in \pi_{\ell(\pi)}})\right).$$
Since the estimate for disjoint cumulants has been proved, it suffices now to show that for any set partition $\pi \in \setpart_{r}$,
$$\left(\sum_{i=1}^{\ell(\pi)} \deg k^{\id}(X_{j} \in \pi_{i})\right)- \ell(\pi)+1 \leq \deg X_{1} +\cdots+\deg X_{r} -r+1.$$
This is trivial because of Proposition \ref{identitycumulant}:
\begin{align*}
\sum_{i=1}^{\ell(\pi)} \deg k^{\id}(X_{j} \in \pi_{i})&\leq \sum_{i=1}^{\ell(\pi)} \left(\sum_{j \in \pi_{i}} \deg X_{i}-|\pi_{i}|+1\right)\\
&\leq\deg X_{1}+\cdots+\deg X_{r}-r+\ell(\pi).
\end{align*}
Notice that the estimates of our Theorem can be made independent of $\omega$: indeed, for any integers $l_{1},\ldots,l_{r}$, $|p_{l_{1},\ldots,l_{r}}(A-B+\gamma E)|\leq 1$, and these quantities are the only ones that depend explicitly on $\omega$ in our proof.
\end{proof}
\noindent A particular case of Theorem \ref{estimate} is when $r=1$; thus, $\esper_{n,\omega}[X]=O(n^{\deg X})$ for any observable $X$. This simple fact and the multiplicativity of the irreducible characters of $\sym_{\infty}$ will now be used in Section \ref{largenumbers}, in which we shall recall the results of Kerov and Vershik (see \cite{KV81a}) and their proofs.
\bigskip

\section{Kerov-Vershik law of large numbers}\label{largenumbers}
In this section, we recall the results of Kerov and Vershik; although nothing is new here, we shall introduce the important concept of a (random) probability measure associated to a (random) diagram, that will later prove useful for the analysis of the fluctuations of the rows and columns, see \S\ref{gaussrow}.

\subsection{Law of large numbers for the irreducible characters}
Given an irreducible character $\chi^{\omega}$, we replace it on $\sym_{n}$ by the random function $\chi^{\lambda}$, where $\lambda$ is taken according to the probability measure $\proba_{n,\omega}$. The law of large numbers of Kerov and Vershik ensures that $\chi^{\lambda}$ converges to $\chi^{\omega}$:
\begin{theorem}\label{llnalg}
For any permutation $\sigma \in \sym_{\infty}$,
$$\chi^{\lambda}(\sigma) \longrightarrow_{\proba_{n,\omega}} \chi^{\omega}(\sigma),$$
where the long right arrow means that one has convergence in probability.
\end{theorem}
\begin{proof}
For any partition $\mu$, $\esper_{n,\omega}\!\left[\frac{\varSigma_{\mu}}{n^{|\mu|}}\right]\simeq p_{\mu}(A-B+\gamma E)$, and 
\begin{align*}
\esper_{n,\omega}\!\left[\left(\frac{\varSigma_{\mu}}{n^{|\mu|}}\right)^{2}\right]&=\esper_{n,\omega}\!\left[\frac{\varSigma_{\mu \sqcup \mu}}{n^{2|\mu|}}\right]+\frac{1}{n^{2|\mu|}}\,\esper_{n,\omega}[\text{obs. of degree at most }2|\mu|-1]\\
&= \left(p_{\mu}(A-B+\gamma E)\right)^{2}+O(n^{-1})\simeq \left(p_{\mu}(A-B+\gamma E)\right)^{2}.
\end{align*}
Here, we have used the estimate $\esper_{n,\omega}[X]=O(n^{\deg X})$ mentioned at the end of the previous paragraph, the factorization property of the symbols $\varSigma_{\mu}$ in higher degree, and the computation of the expectations $\esper_{n,\omega}[\varSigma_{\mu}]$ performed at the end of \S\ref{polfunc}. Because of Bienaym\'e-Chebyshev inequality, these two estimates ensure that $\frac{\varSigma_{\mu}}{n^{|\mu|}}\simeq \chi^{\lambda}(\mu \sqcup 1^{n-|\mu|})$ converges in probability towards $p_{\mu}(A-B+\gamma E)=\chi^{\omega}(\mu)$. 
\end{proof}
\bigskip

Again, Theorem \ref{llnalg} can be given an uniform flavour: namely, for any $\sigma \in \sym_{\infty}$, and any $\eps >0$,
$$\sup_{\omega \in \Omega} \proba_{n,\omega}[|\chi^{\lambda}(\sigma)-\chi^{\omega}(\sigma)|\geq \eps] \to_{n \to \infty} 0.$$
On the other hand, we shall see in Section \ref{clt} that in fact, a much stronger result holds: for any permutation $\sigma$, the renormalized difference $\sqrt{n}\,(\chi^{\lambda}(\sigma)-\chi^{\omega}(\sigma))$ converges in law towards a centered gaussian variable. Now, let us remark that for any $k$,
$$\esper_{n,\omega}\!\left[\left(\frac{p_{k}-\varSigma_{k}}{n^{k}}\right)^{2}\right]=\frac{\esper_{n,\omega}[\text{obs. of degree at most }2k-2]}{n^{2k}}=O(n^{-2}) \to 0.$$
So, Theorem \ref{llnalg} implies that for any $k\geq 1$, and any $\omega \in \Omega$,
$$\frac{p_{k}(\lambda)}{n^{k}} \longrightarrow_{\proba_{n,\omega}} p_{k}(A-B+\gamma E).$$
We are now going to interpret this result in geometric terms, thereby proving a law of large numbers for the first rows $\lambda_{1},\lambda_{2},\ldots$ and the first columns $\lambda_{1}',\lambda_{2}',\ldots$ of a random partition $\lambda$ chosen according to a law $\proba_{n,\omega}$.

\subsection{Random measures associated to random partitions}
We denote by $\mathscr{P}([-1,1])$ the set of borelian probability measures on the interval $[-1,1]$; this is a metrizable compact set for Skorohod's topology of weak convergence, see \cite[Chapter 1]{Bil69}. A distance compatible with this topology is provided by any dense sequence of functions in the space of continuous (real-valued) functions $\mathscr{C}([-1,1])$ --- by Stone-Weierstrass theorem, such a sequence $(f_{n})_{n \in \N}$ exists, and we can for instance take $\{f_{n}\}_{n \in \N}=\Q[x]$. Thus, 
$$d(m_{1},m_{2})=\sum_{n=0}^{\infty}\frac{1}{2^{n}} \min(1,|m_{1}(f_{n})-m_{2}(f_{n})|)$$
is a distance on $\mathscr{P}([-1,1])$ that defines the topology of weak convergence\footnote{Here, for any $f \in \mathscr{C}([-1,1])$ and any $m \in \mathscr{P}([-1,1])$, $m(f)=\int_{-1}^{1} f(x) \,m(dx)$.}. Now, to any partition $\lambda$ of size $n$, we associate a probability measure $X_{\lambda} \in \mathscr{P}([-1,1])$:
$$X_{\lambda}=\sum_{i=1}^{d} \frac{a_{i}}{n}\,\delta_{\left(\frac{a_{i}}{n}\right)}+\sum_{i=1}^{d} \frac{b_{i}}{n}\,\delta_{\left(-\frac{b_{i}}{n}\right)},$$
where the $a$'s and the $b$'s are the Frobenius coordinates of $\lambda$. Notice that the moments of $X_{\lambda}$ are up to a shift of index the renormalized Frobenius moments of $\lambda$:
$$\forall k \geq 0,\,\,\,X_{\lambda}(x^{k})=\frac{p_{k+1}(\lambda)}{n^{k+1}}.$$
If $\lambda$ is random, we will then consider $X_{\lambda}$ as a \textbf{random measure}, that is to say a random point of the space $\mathscr{P}([-1,1])$.\bigskip
\bigskip

If $\omega =((\alpha_{1},\alpha_{2},\ldots),(\beta_{1},\beta_{2},\ldots))$ is a point of the Thoma simplex $\Omega$, we can also associate to it a probability measure $X_{\omega} \in \mathscr{P}([-1,1])$:
$$X_{\omega}=\sum_{i=1}^{\infty} \alpha_{i}\,\delta_{(\alpha_{i})}+\sum_{i=1}^{\infty} \beta_{i}\,\delta_{(-\beta_{i})} + \gamma \,\delta_{0}.$$
The moments of $X_{\omega}$ are up to a shift of index the specializations of the power sums on the formal alphabet $A-B+\gamma E$:
$$\forall k \geq 0,\,\,\, X_{\omega}(x^{k})=p_{k+1}(A-B+\gamma E).$$
We have seen in the previous paragraph that under the coherent system of measures $\proba_{n,\omega}$, $\frac{p_{k}(\lambda)}{n^{k}}$ converges in probability towards $p_{k}(A-B+\gamma E)$ for any $k$. Consequently,
$$\forall k \geq 0,\,\,\, X_{\lambda}(x^{k}) \longrightarrow_{\proba_{n,\omega}} X_{\omega}(x^{k}).$$
By linearity, we have the same result for any evaluation of the measures on a polynomial in $\Q[x]$, and since $\Q[x]$ is dense in $\mathscr{C}([-1,1])$, we conclude that:
\begin{proposition}
For the topology of weak convergence,
$$X_{\lambda}\longrightarrow_{\proba_{n,\omega}} X_{\omega}.$$
\end{proposition}
\noindent More concretely, for \emph{any} function $f \in \mathscr{C}([-1,1])$ and any $\eps>0$, $\proba_{n,\omega}[|X_{\lambda}(f)-X_{\omega}(f)|\geq \eps] \to 0$. In fact, these convergences are again uniform on the Thoma simplex:
$$\sup_{\omega \in \Omega}\proba_{n,\omega}[|X_{\lambda}(f)-X_{\omega}(f)|\geq \eps]\to_{n \to \infty} 0.$$
\bigskip

Finally, let us translate the Kerov-Vershik law of large numbers in terms of the asymptotic geometry of the random partitions. From the result on random measures, it is easy to guess that the following should be true:
\begin{theorem}\label{llngeom}
Take $\lambda$ randomly according to a measure $\proba_{n,\omega}$, with $\omega=(\alpha,\beta)$.
$$\forall i \geq1,\,\,\,\frac{\lambda_{i}}{n}\simeq \frac{a_{i}}{n} \longrightarrow_{\proba_{n,\omega}} \alpha_{i}\qquad;\qquad \forall j \geq1,\,\,\,\frac{\lambda_{j}'}{n}\simeq \frac{b_{j}}{n} \longrightarrow_{\proba_{n,\omega}} \beta_{j}.$$
\end{theorem}
\begin{proof}
One of the main difficulty when recovering the shapes of the partitions is that we have to take account of the multiplicities of the coordinates $\alpha_{i},\beta_{j}$ of the point $\omega \in \Omega$. Let us prove rigorously that if $\alpha_{1}=\alpha_{2}=\cdots=\alpha_{m}>\alpha_{m+1}$ are the biggest coordinates of $\alpha$, then the renormalized $m$ first parts of $\lambda$ converge in probability to the value $\alpha_{1}$. We fix $\eps>0$, and we split the proof in two parts.\vspace{2mm}
\begin{enumerate}
\item First, let us prove that $\frac{a_{1}}{n}$ is smaller than $\alpha_{1}+2\eps$ with probability going to $1$ as $n$ goes to infinity (of course, we suppose $\alpha_{1}<1$, this case being obvious). We consider a test function $f_{\eps}^{+}$ that takes its values in $[0,1]$, is continuous, and is equal to $1$ after $\alpha_{1}+2\eps$ and to $0$ before $\alpha_{1}+\eps$ (see Figure \ref{functionplus}). On the one hand, $X_{\omega}(f_{\eps}^{+})=0$, and on the other hand,
$$X_{\lambda}(f_{\eps}^{+})\geq (\alpha_{1}+2\eps)\,\,\card \left\{i \,\,\big|\,\, \frac{a_{i}}{n} \geq \alpha_{1}+2\eps\right\}.$$
Since $X_{\lambda}(f_{\eps})$ converges in probability to $X_{\omega}(f_{\eps}^{+})=0$, this implies that the cardinality is equal to $0$ for $n$ big enough and with big probability.
\figcap{
\includegraphics{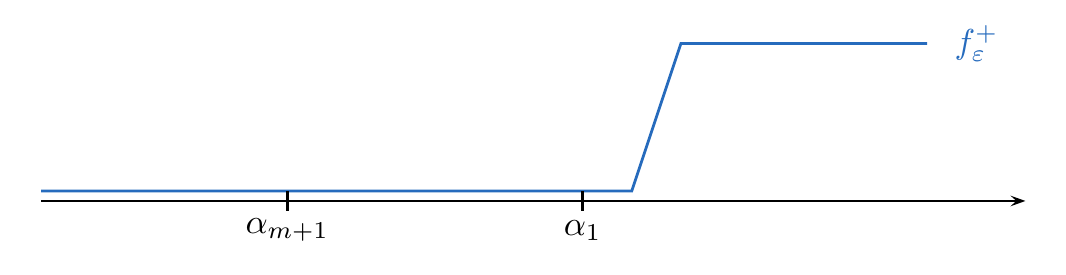}}{Test function $f_{\eps}^{+}$.\label{functionplus}}
\vspace{2mm}
\item Similarly, let us prove that $\frac{a_{m}}{n}$ is bigger than $\alpha_{1}-2\eps$ with probability going to $1$ as $n$ goes to infinity (again, we suppose that $\alpha_{1}>0$, this case being obvious). We consider a test function $f_{\eps}^{-}$ that takes its values in $[0,1]$, is continuous, and is equal to $1$ after $\alpha_{1}-\eps$ and to $0$ before $\alpha_{1}-2\eps$ (see Figure \ref{functionminus}).
\figcap{
\includegraphics{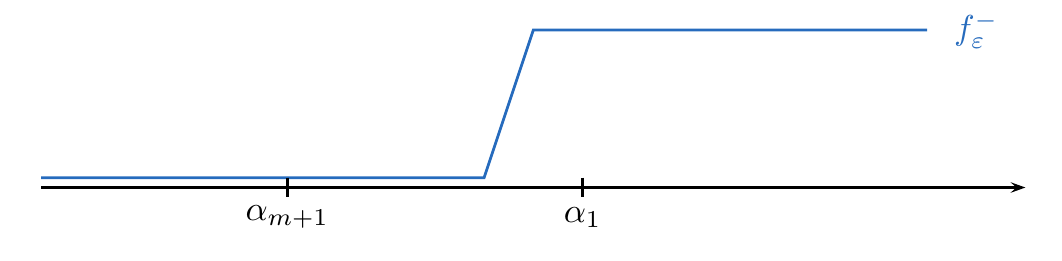}}{Test function $f_{\eps}^{-}$.\label{functionminus}}
For $\eps $ small enough, we may also suppose that $\alpha_{1}-2\eps>\alpha_{m+1}$. Then, $X_{\omega}(f_{\eps}^{-})=m\alpha_{1}$, whereas
$$X_{\lambda}(f_{\eps}^{-}) \leq (\alpha_{1}+2\eps) \,\,\card\left\{i \,\,\big|\,\, \frac{a_{i}}{n} \geq \alpha_{1}-2\eps\right\}$$
with big probability --- indeed, we already know that the $a_{i}$'s are smaller than $\alpha_{1}+2\eps$ with probability almost $1$. For $\eps$ small enough and as $n$ goes to infinity, this leads to the inequality $ \card\{i \,\,|\,\, \frac{a_{i}}{n} \geq \alpha_{1}-2\eps\} \geq m$, which is equivalent to ask that $\frac{a_{m}}{n} \geq \alpha_{1}-2\eps$.\vspace{2mm}
\end{enumerate}
Thus, for any $\eps>0$, the probability of the event
$$\alpha_{1}+2\eps\geq \frac{a_{1}}{n}\geq \cdots \geq \frac{a_{m}}{n} \geq \alpha_{1}-2\eps$$
goes to $1$ as $n$ goes to infinity, which is exactly Theorem \ref{llngeom} for the $m$ first rows of $\lambda$. Now, an induction enables us to treat the other rows, and the columns are treated exactly the same way.
\end{proof}
\begin{example}
When $\omega=\mathbf{0}$, the law of large numbers ensures that all the rows of $\lambda$ are some $o(n)$, and similarly for its columns. Actually, it is a well-known fact that the rows and the columns have order of magnitude $\sqrt{n}$, and that there is a limit shape for the random partitions under Plancherel measures after scaling the Young diagrams in both directions by the factor $\sqrt{n}$ (\emph{cf.} \cite{LS77,KV77}):
\figcap{
\includegraphics{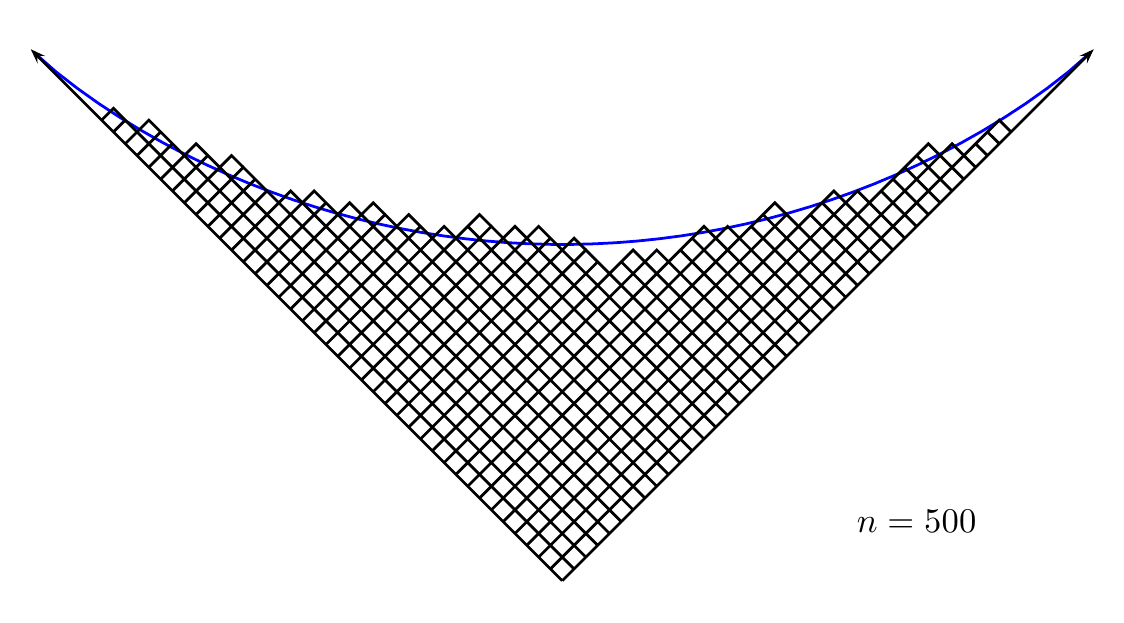}}{Random partition under the Plancherel measure $\proba_{n,\mathbf{0}} = \proba_{n}$.}

\noindent Hence, for this particular point of $\Omega$, the law of large numbers is not precise enough; however, an accurate law of large numbers and a central limit theorem can be proved in this case (and for other classes of measures on partitions) by a slight modification of the tools that we use here (see \cite{IO02,Mel11}). Another important example is when $\omega=(\alpha,\beta)$ with $\alpha$ and $\beta$ geometric progressions of same ratio $q$. However, these examples are better understood in the setting of the Hecke algebras, see Section \ref{qtplancherel}.
\end{example}
\bigskip\bigskip

To conclude this section, let us interpret the geometric law of large numbers in terms of properties of $\omega$-shuffles. We fix an integer $r$. Given a shuffle $\sigma$ of $n_{1}+n_{2}+\cdots+m_{1}+m_{2}+\cdots+l$ cards as in \S\ref{riffleshuffle}, it is obvious that the maximal sum 
$$\ell(iw_{1})+\ell(iw_{2})+\cdots+\ell(iw_{r})$$
of the lengths of $r$ disjoint increasing subwords in $\sigma$ is at least $n_{1}+n_{2}+\cdots+n_{r}$. Indeed, it suffices to consider the words 
$$iw_{j}=\lle n_{1}+\cdots+n_{j-1}+1,n_{1}+\cdots+n_{j}\rre,$$
because they stay in this order in $\sigma$. 
\begin{example}
Take $r=1$, $n=9=5+4$. Any shuffle of $12345$ and $6789$ contains an increasing subword of length at least $5$, namely, $12345$. However, for many shuffles, one can do better: for instance, $162378495$ is a shuffle of $12345$ with $6789$, and it contains an increasing subword of length $6$, namely, $123789$.
\end{example}
\noindent Thus, the $r$-th Greene invariant of $\sigma$, which we shall denote by $GI_{r}(\sigma)$, is bigger than $n_{1}+\cdots+n_{r}$. Now, if the sizes of the blocks are taken according to a multinomial law of parameter $(\alpha,\beta,\gamma)$, then the law of large numbers ensures that $\frac{n_{i}}{n}$ converges in probability to $\alpha_{i}$. Thus, with probability going to $1$, if $\sigma$ is a permutation obtained by an $\omega$-shuffle, then
$$\frac{GI_{r}(\sigma)}{n} \geq \alpha_{1}+\cdots + \alpha_{r}-\eps$$
for any $\eps>0$. However, one could imagine that ``many'' $\omega$-shuffles yield permutations with greater Greene invariants (this impression may be strengthened by the randomized block of size $l\simeq n\gamma$). The Kerov-Vershik law of large numbers and Propositions \ref{frompermtopar} and \ref{apocalyptica} ensure that it is not the case:
\begin{proposition}
Let $\sigma$ be a permutation of size $n$ chosen randomly according to $\qproba_{n,\omega}$. As $n$ goes to infinity, 
$$\frac{GI_{r}(\sigma)}{n} \longrightarrow_{\qproba_{n,\omega}} \alpha_{1}+\cdots+\alpha_{r}.$$
Of course, the same holds if one replaces the increasing subwords by the decreasing subwords and the $\alpha$'s by the $\beta$'s.
\end{proposition}
\begin{proof}
Because of Proposition \ref{frompermtopar}, the distribution of $GI_{r}(\sigma)$ under $\qproba_{n,\omega}$ is the same as the distribution of $\lambda_{1}+\cdots+\lambda_{r}$ under $\proba_{n,\omega}$. Then, it suffices to apply Theorem \ref{llngeom}.
\end{proof}\bigskip

Hence, $\omega$-shuffles do not create bigger increasing subwords than those given by the sizes of the blocks. In Section \ref{clt}, we are going to prove a central limit theorem that will improve this result. Hence, assuming that the $\alpha_{i}$'s are all distinct, up to order $\sqrt{n}$ (and not $n$), the asymptotic distribution of $GI_{r}(\sigma)$ is the same as the asymptotic distribution of the size $n_{1}+\cdots+n_{r}$ of the $r$ first blocks of the $\omega$-shuffle, that is to say
$$n\,(\alpha_{1}+\cdots+\alpha_{r})+\sqrt{n}\,\sqrt{(\alpha_{1}+\cdots+\alpha_{r})(1-\alpha_{1}-\cdots-\alpha_{r})}\,\,\mathcal{N}(0,1)+o(\sqrt{n})$$
with $\mathcal{N}(0,1)$ a standard gaussian variable.
\bigskip

\section{Central limit theorems for characters and for partitions}\label{clt}
In the previous paragraph, it was shown that all the coherent systems of measures $(\proba_{n,\omega})_{n \in \N} $ have a property of concentration, that can be seen at the level of characters or directly for the geometry of the random partitions. In this section, we shall prove that the concentration is always gaussian, which is quite a striking fact. However, the geometric version of our central limit theorem will be a little less precise than the algebraic one: indeed, one cannot study with our method of moments the fluctuations of rows or columns corresponding to equal coordinates of the point $\omega$ of the Thoma simplex. This will be done in \S\ref{oconnell} by entirely different methods, that relie essentially on the interpretation of the measures $\proba_{n,\omega}$ as push-forwards of the measures $\qproba_{n,\omega}$.\bigskip

\subsection{Central limit theorem for the irreducible characters} In this paragraph, we fix permutations $\sigma^{(1)},\ldots,\sigma^{(r)} \in \sym_{\infty}$ of cycle types $\mu^{(1)},\ldots,\mu^{(r)}$; the partitions $\mu^{(i)}$ are all defined up to parts of size $1$, but this will not change our results. For convenience, we shall denote $p_{\mu}(A-B+\gamma E)=p_{\mu}(\omega)$.
\begin{theorem}\label{mainalg}
We consider the random vector $(\Delta_{n,\omega}(\sigma^{(1)}),\ldots,\Delta_{n,\omega}(\sigma^{(r)}))$, with
$$\Delta_{n,\omega}(\sigma)=\sqrt{n}\,\left(\chi^{\lambda}(\sigma)-\chi^{\omega}(\sigma)\right)$$
and $\lambda$ chosen randomly according to $\proba_{n,\omega}$. As $n$ goes to infinity, this vector converges in joint law to a centered gaussian vector $(\Delta_{\infty,\omega}(\sigma^{(1)}),\ldots,\Delta_{\infty,\omega}(\sigma^{(r)}))$, and the covariance matrix of this vector is:
\begin{align*}
&k(\Delta_{\infty,\omega}(\sigma^{(i)}),\Delta_{\infty,\omega}(\sigma^{(j)}))\\
&=\sum_{a \in \mu^{(i)},\,\,
b\in \mu^{(j)} } ab\,\,p_{\mu_{i}\sqcup \mu_{j}\setminus \{a,b\}}(\omega)\,\,(p_{a+b-1}(\omega)-p_{a}(\omega)\,p_{b}(\omega)).
\end{align*}
In particular, $k(\Delta_{\infty,\omega}(\sigma),\Delta_{\infty,\omega}(\tau))=lm\,\,(p_{l+m-1}(\omega)-p_{l}(\omega)\,p_{m}(\omega))$ if $\sigma$ and $\tau$ are cycles of lengths $l$ and $m$.
\end{theorem}
\begin{proof}
For any partition $\mu$,
$$\sqrt{n}\,\left(\chi^{\lambda}(\mu\sqcup 1^{n-|\mu|})-p_{\mu}(\omega)\right)=\sqrt{n} \left(\frac{\varSigma_{\mu}(\lambda)-\esper_{n,\omega}[\varSigma_{\mu}]}{n^{|\mu|}}\right)+O\left(\frac{1}{\sqrt{n}}\right),$$
so it is sufficient to show that the random vector 
$$V=\sqrt{n}\left(\frac{\varSigma_{\mu^{(1)}}}{n^{|\mu^{(1)}|}},\ldots,\frac{\varSigma_{\mu^{(r)}}}{n^{|\mu^{(r)}|}}\right)$$
converges to a gaussian vector with the announced covariance matrix. However, for $s\geq 3$, by using Theorem \ref{estimate}, one sees that
\begin{align*}
k_{n,\omega}\left(\varSigma_{\mu^{(i_{1})}},\ldots,\varSigma_{\mu^{(i_{s})}}\right)&=O\left(n^{|\mu^{(i_{1})}|+\cdots+|\mu^{(i_{s})}|-s+1}\right)\\
k_{n,\omega}\left(\sqrt{n}\,\frac{\varSigma_{\mu^{(i_{1})}} }{n^{|\mu^{(i_{1})}|}},\ldots,\sqrt{n}\,\frac{\varSigma_{\mu^{(i_{s})}} }{n^{|\mu^{(i_{s})}|}} \right)&=O\left(n^{1-\frac{s}{2}}\right)\to 0.
\end{align*}
Thus, $V$ converges indeed towards a gaussian vector. To compute the covariances, we use the following trick:
\begin{align*}k_{n,\omega}(\varSigma_{\mu},\varSigma_{\nu})&=\esper_{n,\omega}[\varSigma_{\mu}\,\varSigma_{\nu}]-\esper_{n,\omega}[\varSigma_{\mu}]\,\esper_{n,\omega}[\varSigma_{\nu}]\\
&=\esper_{n,\omega}[\varSigma_{\mu}\,\varSigma_{\nu}-\varSigma_{\mu\sqcup \nu}]+\big(\esper_{n,\omega}[\varSigma_{\mu\sqcup\nu}]-\esper_{n,\omega}[\varSigma_{\mu}]\,\esper_{n,\omega}[\varSigma_{\nu}]\big).
\end{align*}
The term of higher degree in the first part of the right-hand side is exactly the expectation of the subdominant term of a product $\varSigma_{\mu}\,\varSigma_{\nu}$. We have calculated this subdominant term after Proposition \ref{matching}; hence,
\begin{align*}
\esper_{n,\omega}[\varSigma_{\mu}\,\varSigma_{\nu}-\varSigma_{\mu\sqcup \nu}]&=\sum_{a \in \mu,b \in \nu}ab\,\,\esper_{n,\omega}[\varSigma_{(\mu\sqcup\nu \sqcup \{a+b-1\}) \setminus \{a,b\}}]\\
&\simeq \left(\sum_{a \in \mu,\,b \in \nu}ab\,\, p_{\mu\sqcup\nu \setminus \{a,b\}}(\omega)\,\,p_{a+b-1}(\omega)\right)n^{|\mu|+|\nu|-1}.
\end{align*}
On the other hand, $\esper_{n,\omega}[\varSigma_{\mu\sqcup\nu}]-\esper_{n,\omega}[\varSigma_{\mu}]\,\esper_{n,\omega}[\varSigma_{\nu}]$ is easily computed to be equal to 
$$p_{\mu \sqcup \nu}(\omega) \,\,\left[n^{\downarrow |\mu|+|\nu|}-n^{\downarrow |\mu|}\,n^{\downarrow |\nu|}\right]\simeq - \,|\mu|\,|\nu|\,\,p_{\mu \sqcup \nu}(\omega) \,\,n^{|\mu|+|\nu|-1}$$
Dividing everything by $n^{|\mu|+|\nu|-1}$ and writing
$$|\mu|\,|\nu|\,\,p_{\mu \sqcup \nu}(\omega)=\sum_{a\in \mu,\,b\in \nu}ab\,\,p_{\mu\sqcup\nu \setminus \{a,b\}}(\omega)\,\,p_{a}(\omega)\,p_{b}(\omega),$$
we obtain exactly the announced formula.
\end{proof}
\bigskip
\bigskip

So, for any integer $k\geq 2$, if $\sigma$ is a $k$-cycle, then $$\sqrt{n}\,\left(\frac{\chi^{\lambda}(\sigma)-\chi^{\omega}(\sigma)}{k}\right)$$ converges under the laws $\proba_{n,\omega}$ to a centered gaussian variable of variance $p_{2k-1}(\omega)-p_{k}(\omega)^{2}$. Again, this result is not precise enough when $\omega=\mathbf{0}$ --- fortunately, this is the only problematic point. Actually, it has been shown in \cite[Theorem 6.1]{IO02} that under Plancherel measures, if $\sigma$ is a $k$-cycle, then $$n^{\frac{k}{2}}\,\left(\frac{\chi^{\lambda}(\sigma)}{\sqrt{k}}\right)$$ converges to a centered gaussian variable of variance $1$. We shall say a little bit more on these kind of results in the conclusion of our paper.

\subsection{Gaussian behaviour of the rows and of the columns}\label{gaussrow}
Now, let us look at the consequences of Theorem \ref{mainalg} for the geometry of the random partitions. First, let us give a central limit theorem for the random measures $X_{\lambda}$ under coherent systems of probability measures $(\proba_{n,\omega})_{n \in \N}$. In the following, $\mathscr{C}^{1}([-1,1])$ is endowed with the topology of uniform convergence of the functions and of their derivatives, and measures are now considered as elements of the dual space of $\mathscr{C}^{1}([-1,1])$ --- this dual space contains $\mathscr{P}([-1,1]) \subset \mathscr{M}([-1,1])=\mathscr{C}^{0}([-1,1])^{*}$.
\begin{theorem}\label{mainmeasure}
We consider the random measure $\nabla_{n,\omega}=\sqrt{n}\,(X_{\lambda}-X_{\omega})$, with $\lambda$ chosen randomly according to $\proba_{n,\omega}$. As $n$ goes to infinity, $\nabla_{n,\omega}$ converges to a centered gaussian element $\nabla_{\infty,\omega}$ in $\mathscr{C}^{1}([-1,1])^{*}$, and for any functions $f,g \in \mathscr{C}^{1}([-1,1])$, 
$$k(\nabla_{\infty,\omega}(f),\nabla_{\infty,\omega}(g))=X_{\omega}((xf(x))'(xg(x))')-X_{\omega}((xf(x))')\,X_{\omega}((xg(x))').$$
\end{theorem}
\begin{proof}
We have shown in our ``algebraic'' central limit theorem that the random vector 
$$\left(\sqrt{n}\,\left(\frac{\varSigma_{k}}{n^{k}} - p_{k}(\omega)\right)\right)_{k\geq 1}$$
converges to a centered gaussian vector with covariance matrix 
$$lm\,(p_{l+m-1}(\omega)-p_{l}(\omega)\,p_{m}(\omega)).$$ By using the same arguments as before, one can replace $\varSigma_{k}$ by $p_{k}$, and therefore, the same results holds for
\begin{align*}
\left(\sqrt{n}\,\left(\frac{p_{k}(\lambda)}{n^{k}}-p_{k}(\omega) \right)\right)_{k \geq 1}&=\left(\sqrt{n}\,\left(X_{\lambda}(x^{k})-X_{\omega}(x^{k}) \right)\right)_{k\geq 0}\\
&=(\nabla_{n,\omega}(x^{k}))_{k\geq 0},
\end{align*}
although with a shift of index: $$k(\nabla_{\infty,\omega}(x^{l}),\nabla_{\infty,\omega}(x^{m}))=(l+1)(m+1)\,(p_{l+m+1}(\omega)-p_{l+1}(\omega)\,p_{m+1}(\omega)).$$ 
Now, this computation of covariances can be rewritten as:
\begin{align*}k(\nabla_{\infty,\omega}(x^{l}),\nabla_{\infty,\omega}(x^{m}))&=(l+1)(m+1)\,(X_{\omega}(x^{l}\,x^{m})-X_{\omega}(x^{l})\,X_{\omega}(x^{m}))\\
&=X_{\omega}((x^{l+1})'(x^{m+1})')-X_{\omega}((x^{l+1})')\,X_{\omega}((x^{m+1})').
\end{align*}
By multilinearity, for any polynomials $f$ and $g$ in $\R[x]$, $\nabla_{n,\omega}(f)$ and $\nabla_{n,\omega}(g)$ converge towards centered gaussian variables, and the covariance of the limiting distibutions writes as
$$k(\nabla_{\infty,\omega}(f),\nabla_{\infty,\omega}(g))=X_{\omega}((xf(x))'(xg(x))')-X_{\omega}((xf(x))')\,X_{\omega}((xg(x))').$$ 
Finally, one uses the density of $\R[x]$ in $\mathscr{C}^{1}([-1,1])$ to get the same result for  any functions $f,g \in \mathscr{C}^{1}([-1,1])$.  However, notice that one has first to apply a random version of the Banach-Steinhaus theorem for the (random) linear forms $\nabla_{n,\omega} : \mathscr{C}^{1}([-1,1])\to \R$ (see \cite{Tha87} or \cite{VV95}). Indeed, because of the factor $\sqrt{n}$, it is not clear \emph{a priori} that one can get a uniform bound on expressions like
$$\sqrt{n}\,\left(X_{\lambda}(f-f_{k})-X_{\omega}(f-f_{k})\right)$$
with $(f_{k})_{k \in \N}$ sequence of polynomials approximating $f$ in $\mathscr{C}^{1}([-1,1])$ --- this point was left unexplained in \cite{FM10}.
\end{proof}
\bigskip
\bigskip

Unfortunately, Theorem \ref{mainmeasure} is sufficient to describe the fluctuations of the row $\lambda_{i}$ or the column $\lambda_{j}'$ only when $\alpha_{i-1}>\alpha_{i}>\alpha_{i+1}$ or when $\beta_{j-1}>\beta_{j}>\beta_{j+1}$. Indeed, our method of moments will not allow to distinguish at order $\sqrt{n}$ the parts corresponding to equal coordinates. We fix distinct coordinates $\alpha_{i_{1}}\neq \alpha_{i_{2}}\neq \cdots \neq \alpha_{i_{r}}\neq 0$ and $\beta_{j_{1}}\neq \beta_{j_{2}} \neq \cdots \neq \beta_{j_{s}}\neq 0$ of the point $\omega \in \Omega$, and we denote by $m_{k}$ (respectively, $w_{k}$) the multiplicity of $\alpha_{i_{k}}$ (resp., of $\beta_{j_{k}}$), so
$$\alpha_{i_{k}-1}>\alpha_{i_{k}}=\alpha_{i_{k}+1}=\cdots=\alpha_{i_{k}+m_{k}-1}>\alpha_{i_{k}+m_{k}}$$
and similarly for $\beta_{j_{k}}$. We then denote
\begin{align*}\Lambda_{i_{k},m_{k},n}&=\sqrt{n}\left( \frac{\lambda_{i_{k}}}{n}+ \frac{\lambda_{i_{k}+1}}{n}+\cdots+ \frac{\lambda_{i_{k}+m_{k}-1}}{n}-m_{k}\,\alpha_{i_{k}}\right);\\
\Lambda_{j_{k},w_{k},n}'&=\sqrt{n}\left( \frac{\lambda_{j_{k}}'}{n}+ \frac{\lambda_{j_{k}+1}'}{n}+\cdots+ \frac{\lambda_{j_{k}+w_{k}-1}'}{n}-w_{k}\,\beta_{j_{k}}\right).
\end{align*}

\begin{theorem}\label{almostgeom}
As $n$ goes to infinity, the random vector $$(\Lambda_{i_{1},m_{1},n},\ldots,\Lambda_{i_{r},m_{r},n},\Lambda_{j_{1},w_{1},n}',\ldots,\Lambda_{j_{s},w_{s},n}')$$
 converges to a centered gaussian vector $(\Lambda_{i_{1},m_{1},\infty},\ldots,\Lambda_{j_{s},l_{s},\infty}')$, with the following covariances:
\begin{align*}
k(\Lambda_{i_{k},m_{k},\infty},\Lambda_{i_{l},m_{l},\infty})&=\delta_{kl}\,m_{k}\,\alpha_{i_{k}}-m_{k}m_{l}\,\alpha_{i_{k}}\alpha_{i_{l}}\,;\\
k(\Lambda_{j_{k},w_{k},\infty}',\Lambda_{j_{l},w_{l},\infty}')&=\delta_{kl}\,w_{k}\,\beta_{j_{k}}-w_{k}w_{l}\,\beta_{j_{k}}\beta_{j_{l}}\,;\\
k(\Lambda_{i_{k},m_{k},\infty},\Lambda_{j_{l},w_{l},\infty}')&=-m_{k}w_{l}\,\alpha_{i_{k}}\beta_{j_{l}}.
\end{align*}
In particular, if $\alpha_{i}$ and $\alpha_{j}$ are two ``isolated'' coordinates of the sequence $\alpha$, then the random variables $X_{i,n}=\sqrt{n}\left(\frac{\lambda_{i}}{n}-\alpha_{i}\right)$ and $X_{j,n}=\sqrt{n}\left(\frac{\lambda_{j}}{n}-\alpha_{j}\right)$ converge jointly in law to centered gaussian variables $X_{i,\infty}$ and $X_{j,\infty}$, with
$$k(X_{i,\infty},X_{j,\infty})=\delta_{ij}\,\alpha_{i}-\alpha_{i}\alpha_{j}.$$
\end{theorem}
\begin{proof}
We consider test functions $f_{k}$ and $g_{l}$ such that:\vspace{2mm}
\begin{enumerate}
\item All these functions are in $\mathscr{C}^{1}([-1,1])$, non-negative and with values in $[0,1]$.\vspace{2mm}
\item For any $k$, there is a neighbourhood $W_{k}=[\alpha_{i_{k}}-\eps_{k},\alpha_{i_{k}}+\eps_{k}]$ of $\alpha_{i_{k}}$   such that $W_{k}$ does not contain any element of the set $\{\alpha_{i\geq 1},-\beta_{j\geq 1},0\} \setminus \{\alpha_{i_{k}}\}$, and the support of $f_{k}$ is contained in $W_{k}$. Moreover, there is a neighboorhood $V_{k}\subset W_{k}$ of $\alpha_{i_{k}}$ such that $f_{|V_{k}}=1$.\vspace{2mm}
\item Similarly, for any $l$, there is a neighbourhood $W_{l}'=[-\beta_{j_{l}}-\eps_{l}',-\beta_{j_{l}}+\eps_{l}']$ of $-\beta_{j_{l}}$   such that $W_{l}'$ does not contain any element of  $\{\alpha_{i\geq 1},-\beta_{j\geq 1},0\} \setminus \{\beta_{j_{l}}\}$, and the support of $g_{l}$ is contained in $W_{l}'$. Moreover, there is a neighboorhood $V_{l}'\subset W_{l}'$ of $-\beta_{j_{l}}$ such that $g_{|V_{l}'}=1$.
\end{enumerate}
\figcap{
\includegraphics{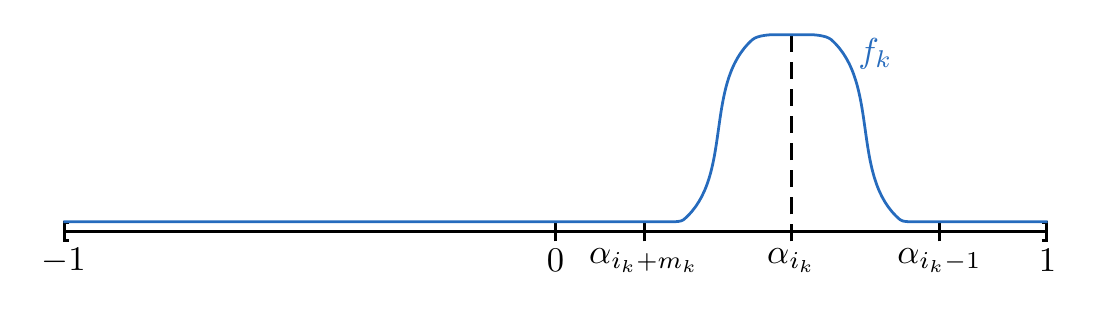}}{Test functions for the proof of the geometric central limit theorem.}
We fix an index $k$. Since $\frac{\lambda_{i}}{n} \to \alpha_{i}$ for all $i$, with probability going to $1$ as $n$ goes to infinity, the intersection of the support of $X_{\lambda}$ with $W_{k}$ consists exactly in 
$$\left\{\frac{\lambda_{i_{k}}}{n},\frac{\lambda_{i_{k}+1}}{n},\ldots,\frac{\lambda_{i_{k}+m_{k}-1}}{n}\right\},$$ that are even in $V_{k}$ with big probability. So,
$$X_{\lambda}(f_{k})=\frac{\lambda_{i_{k}}+ \lambda_{i_{k}+1}+\cdots+\lambda_{i_{k}+m_{k}-1}}{n}$$
with probability going to $1$; on the other hand, $X_{\omega}(f_{k})=m_{k}\,\alpha_{i_{k}}$. Hence,
$$\Lambda_{i_{k},m_{k},n}=\nabla_{n,\omega}(f_{k})$$
for any $k$ with probability going to $1$; similarly, $\Lambda_{j_{l},w_{l},n}'=\nabla_{n,\omega}(g_{l})$ with big probability. Then, by Theorem \ref{mainmeasure},  the random vector of the statement of our theorem converges in law to a centered gaussian vector $$(\Lambda_{i_{1},m_{1},\infty},\ldots,\Lambda_{i_{r},m_{r},\infty},\Lambda_{j_{1},l_{1},\infty}',\ldots,\Lambda_{j_{s},l_{s},\infty}'),$$ with $k(\Lambda_{i_{k},m_{k},\infty},\Lambda_{i_{l},m_{l},\infty})$ equal to
$$X_{\omega}((xf_{k}(x))' (xf_{l}(x))')-X_{\omega}((xf_{k}(x))')\,X_{\omega}((xf_{l}(x))'),$$
and the same kind of formula with variables $\Lambda_{j_{k},w_{k},\infty}'$. Now, notice that $f_{k}'(x)=0$ for any $x$ in the support of $X_{\omega}$, and any $k$. Hence, the previous covariance may be written as
$$X_{\omega}(f_{k}f_{l})-X_{\omega}(f_{k})\,X_{\omega}(f_{l})=\delta_{kl}\,m_{k}\,\alpha_{i_{k}}-m_{k}m_{l}\,\alpha_{i_{k}}\alpha_{i_{l}},$$
and the proof for the other covariance formulas is the same with functions $g_{l}$ instead of functions $f_{k}$.
\end{proof}
\bigskip

\section{Pitman operators and a link with conditioned random walks and brownian matrices}\label{oconnell}
In the previous paragraphs, we have shown that the rows and columns of a partition $\lambda$ of law $\proba_{n,\omega}$ have limit frequencies $\alpha_{i}$ and $\beta_{j}$, and gaussian fluctuations when gathered  for equal coordinates of the point $\omega$ of the Thoma simplex. To treat the case of equal parameters, we shall use O'Connell's theory (\cite{OC03}) that relates the Robinson-Schensted algorithm to an operation on multidimensional paths, and therefore random tableaux coming from $\omega$-shuffles to random walks conditioned to stay in a Weyl chamber. These random walks are also related to the processes of eigenvalues of brownian traceless Hermitian matrices. Thus, the fluctuations of a block of $d$ rows or columns corresponding to equal coordinates of $\omega$ are the same as those of the eigenvalues of a gaussian matrix in $\mathfrak{su}(d,\C)$, up to the additional gaussian term described by Theorem \ref{almostgeom}; see Theorem \ref{maingeom} for a precise statement. We shall prove this progressively, starting from the toy model of $2$-shuffles, then introducing \textbf{Pitman operators} in order to treat the case of $d$-shuffles, and finally studying the general case.

\subsection{The case of $2$-shuffles} To begin with, let us treat minutely the case when $$\omega=\left(\left(\frac{1}{2},\frac{1}{2},0,0,\ldots\right),(0,0,\ldots)\right).$$ The corresponding $\omega$-shuffles (which we shall call \textbf{$\mathbf{2}$-shuffles}) are easy to describe. Indeed, for any $n$, a random permutation of law $\qproba_{n,\omega}$ can be obtained by shuffling together two blocks of size $n_{1}$ and $n_{2}=n-n_{1}$ with $n_{1} \sim \mathcal{B}(n,\frac{1}{2})$, that is to say that
$$\proba[n_{1}=k]=\frac{1}{2^{n}}\,\binom{n}{k}.$$
We recall that all the $\binom{n}{k}$ shuffles of two blocks of sizes $k$ and $n-k$ are equiprobable. Thus, if $b_{1}, \ldots,b_{n}$ are independent identically distributed variables with $\proba[b_{i}=1]=\proba[b_{i}=2]=1/2$, then one can produce a permutation of law $\qproba_{n,\omega}$ by:
\begin{itemize}
\item setting $n_{1}=\card \{i \,\,|\,\, b_{i}=1\}$, $n_{2}=\card \{i \,\,|\,\, b_{i}=2\}$.
\item sending the integers of $\lle 1,n_{1}\rre$ in this order to the positions $i$ such that $b_{i}=1$, and then sending the integers of $\lle n_{1}+1,n_{1}+n_{2}\rre$ in this order to the positions $i$ such that $b_{i}=2$.
\end{itemize}
For instance, if the sequence of Bernoulli variables is $121121221$, then the corresponding permutation\footnote{With the vocabulary of \S\ref{randompermutation}, $\sigma$ is the standardization of the word whose letters are the Bernoulli variables.} will be $162374895$. This explicit construction of $2$-shuffles will enable us to relate them to random walks --- here, the random walk that we shall consider is:
\figcap{
\includegraphics{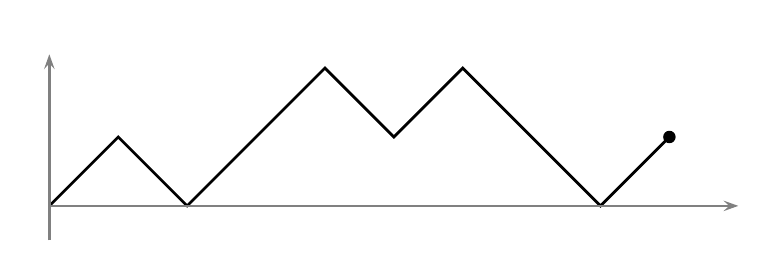}}{Random walk associated to the $2$-shuffle $162374895=\mathrm{Std}(121121221)$.}
\bigskip
\bigskip

Given a permutation $\sigma$ obtained by standardization from a word $w$ in $\{1,2\}^{*}$, the shape of the tableau $P(\sigma)$ is the same as the shape of the tableau $P(w)$. Actually, for any $k \leq n$, if $P_{k}(\sigma)$ (respectively, $P_{k}(w)$) is the tableau obtained after Schensted insertion of the $k$ first letters of $\sigma$ (resp., of the $k$ first letters of $w$), then $P_{k}(\sigma)$ has the same shape as $P_{k}(w)$. Hence, we will study the random process $(P(w_{k}))_{k \geq 0}$, where $w_{k}=b_{1}b_{2}\ldots b_{k}$ is a random word in $\{1,2\}^{k}$ with independent Bernoulli letters. This process takes its values in 
$$\mathrm{SST}\{1,2\}=\{\text{semi-standard tableaux with entries in }\{1,2\}\}.$$
A semi-standard tableau with entries in $\{1,2\}$ has at most two rows, and the entries in the second row are all equal to $2$. Thus, if one knows the size $n$ of the tableau, then it is entirely determined by the number $m_{11}$ of entries $1$ on the first row, and by the number $m_{12}$ of entries $2$ on the first row:
\begin{center}
\includegraphics{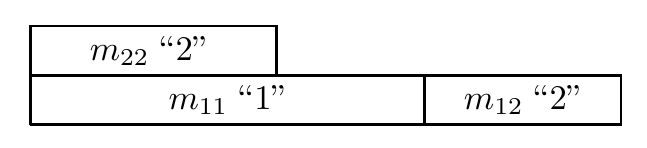}
\end{center}
with $m_{22}=n-m_{11}-m_{12}$. So, let us study in detail the random process $(m_{11}(k),m_{12}(k))$. For this purpose, we introduce the random walk $X_{k}=3k-2(b_{1}+b_{2}+\cdots+b_{k})$; this is a standard random walk on $\Z$, starting from $0$.

\begin{lemma}\label{2shuffle}
For any $k$, $m_{11}(k)=\frac{k+X_{k}}{2}$ and $m_{12}(k)=S_{k}-X_{k}$, where $S_{k}=\sup_{j \leq k}X_{j}$.
\end{lemma}
\begin{proof}
The quantity $m_{11}(k)$ is just the number of ``$1$'' in the sequence $b_{1},\ldots,b_{k}$. Hence, 
$$m_{11}(k)=\sum_{i=1}^{k} 2-b_{i}=2k-(b_{1}+b_{2}+\cdots+b_{k})=\frac{k+X_{k}}{2}.$$
As for $m_{12}(k)$, let us prove the result by induction on $k$. Supposing $m_{12}(k)=S_{k}-X_{k}$, let us compute $m_{12}(k+1)$:\vspace{2mm}
\begin{enumerate}
\item If $b_{k+1}=1$, one inserts a letter $1$ in the first row. Suppose that $m_{12}(k)=0$. Then, $m_{12}(k+1)$ is also equal to $0$, so $S_{k}=X_{k}$, and $X_{k+1}=X_{k}+1$, because $b_{k+1}=1$ means that we go up. This implies that $S_{k+1}=X_{k+1}$, so $m_{12}(k+1)=0=S_{k+1}-X_{k+1}.$ On the contrary, suppose that $m_{12}(k)\geq 1$. Then, the letter $1$ will replace a letter $2$, so 
$$m_{12}(k+1)=m_{12}(k)-1=S_{k}-X_{k}-1.$$ On the other hand, since $X_{k}=S_{k}-m_{12}(k) \leq S_{k}-1$, $X_{k+1}=X_{k}+1 \leq S_{k}$, so $S_{k}=S_{k+1}$. But then, $m_{12}(k+1)=S_{k}-(X_{k}+1)=S_{k+1}-X_{k+1}$, so the identity holds again at rank $k+1$.\vspace{2mm}
\item If $b_{k+1}=2$, one inserts a letter $2$ at the end of the first row, so $$m_{12}(k+1)=m_{12}(k)+1=S_{k}-X_{k}+1.$$ On the other hand, $b_{k+1}=2$ means that the random walk goes down, so $X_{k+1}=X_{k}-1$ and $S_{k}=S_{k+1}$. Then, $m_{12}(k+1)=S_{k}-(X_{k}-1)=S_{k+1}-X_{k+1}$.
 \end{enumerate}
Hence, if $\ell_{k}(\sigma)$ is the length of the longest increasing subword of $\sigma(1)\cdots \sigma(k)$, then 
$$\ell_{k}(\sigma)=m_{11}(k)+m_{12}(k)=\frac{k}{2}+ \frac{2S_{k}-X_{k}}{2} ,$$
and this for any $k\leq n$.
\end{proof}
\bigskip\bigskip

It is known since Pitman's works (see \cite{Pit75}) that if $(X_{k})_{k \in \N}$ is a standard random walk on $\Z$, then $(2S_{k}-X_{k})_{k \in \N}$ is a markovian random process on the non-negative integers, with
$$p(k,k+1)=\frac{k+2}{2(k+1)}\qquad;\qquad p(k,k-1)=\frac{k}{2(k+1)}.$$
After scaling by $n$ on the $x$-axis and $\sqrt{n}$ on the $y$-axis, the random process $(2S_{k}-X_{k})_{k \in \lle 0,n\rre}$ converges in law in $\mathscr{C}([0,1])$ to a \textbf{3-dimensional Bessel process} $(\mathrm{BES}^{3}_{t})_{t \in [0,1]}$, whose trajectorial laws are the same as $(2S_{t}-B_{t})_{t \in [0,1]}$, with $(B_{t})_{t \in [0,1]}$ standard brownian motion and $S_{t}=\sup_{s\leq t}B_{s}$. Alternatively, $(\mathrm{BES}^{3}_{t})_{t \in [0,1]}$ can be defined as the solution of the stochastic differential equation
$$dY_{t}=dB_{t}+\frac{dt}{Y_{t}},$$
see \cite[Chapter 6, \S3]{RY91}. For this reason, and using It\^{o}'s calculus, one can show that $(\mathrm{BES}^{3}_{t})_{t \in [0,1]}$ has the same trajectorial laws as the euclidian norm $(\sqrt{X_{t}^{2}+Y_{t}^{2}+Z_{t}^{2}})_{t \in [0,1]}$ of a $3$-dimensional brownian motion. In particular, this allows to compute the density of the law of $\mathrm{BES}^{3}_{t}$:
$$\mathrm{BES}^{3}_{t}\sim \mathbf{1}_{y \geq 0} \,\sqrt{\frac{2}{\pi\,t^{3}}}\,y^{2}\,\E^{-\frac{y^{2}}{2t}}\,dy.$$

\noindent With these classical results from probability, Lemma \ref{2shuffle} can be reinterpreted as follows:
\begin{proposition}\label{bessel3}
Let $\sigma$ be a random permutation obtained by a $2$-shuffle. We denote
$$L_{t \in [0,1[}^{(n)}=\frac{2}{\sqrt{n}}\left(\left[(\lfloor nt +1\rfloor-nt)\,\ell_{\lfloor nt \rfloor}(\sigma)+ (nt-\lfloor nt\rfloor)\,\ell_{\lfloor nt \rfloor +1}(\sigma)\right]-\frac{nt}{2}\right)$$
the random process obtained by renormalization and interpolation of the process $(\ell_{k}(\sigma) - \frac{k}{2})_{k \in \lle 0,n\rre}$. As $n$ goes to infinity, $(L_{t}^{(n)})_{t \in [ 0,1]}$ converges in law\footnote{As usual, the interpolation is only needed in order to work in $\mathscr{C}([0,1])$; if one uses instead step \emph{c\`adl\`ag} functions, one has the same result of convergence, but this time in Skorohod's space $\mathscr{D}([0,1])$, see \cite[Chapter 3, \S14]{Bil69}.} in $\mathscr{C}([0,1])$ to a 3-dimensional Bessel process $(\mathrm{BES}^{3}_{t})_{t \in [0,1]}$ (see Figure \ref{figbessel}). 
\figcap{
\includegraphics[scale=0.8]{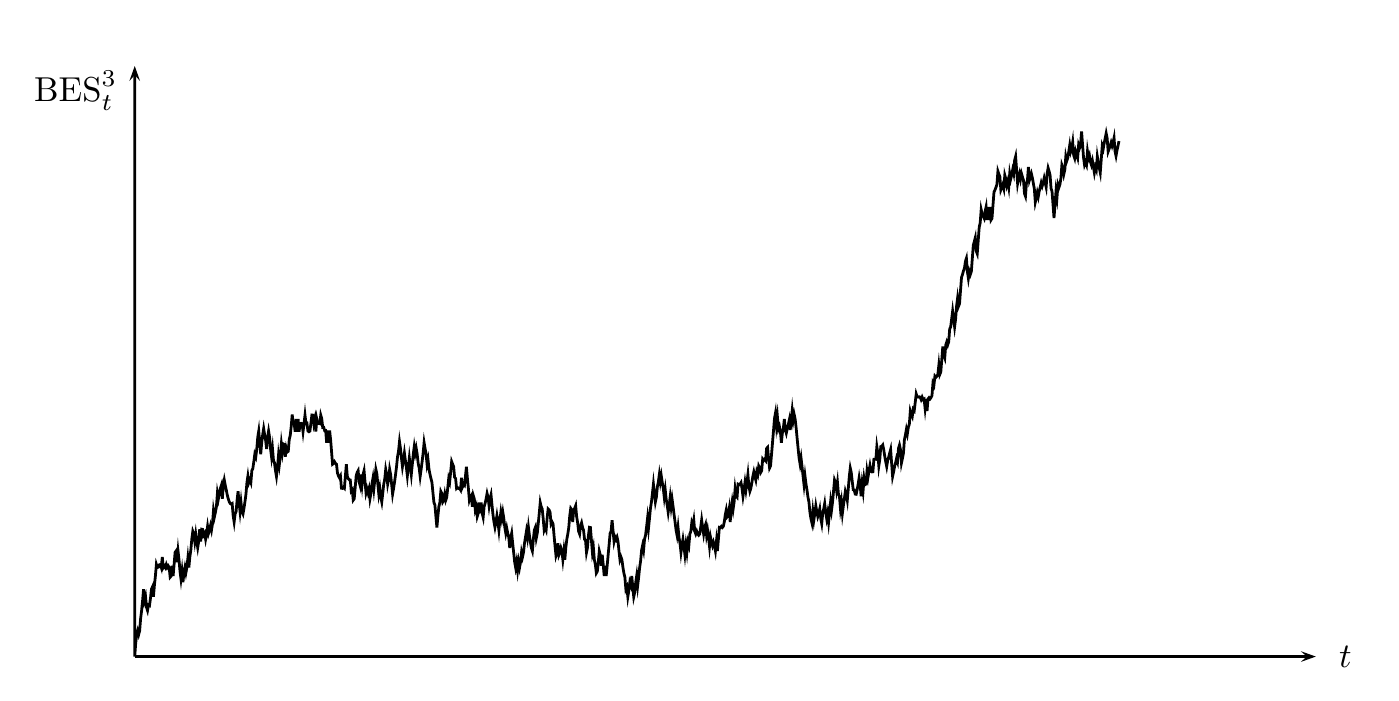}}{A trajectory of the $3$-dimensional Bessel process.\label{figbessel}}

\noindent In particular, 
$$\ell(\sigma)=\lambda_{1}=\frac{n}{2}+\frac{\sqrt{n}}{2}\,Y+o\left(\sqrt{n}\right), \quad\text{with }Y\sim \mathrm{BES}^{3}_{1}\sim \mathbf{1}_{y \geq 0} \,\sqrt{\frac{2}{\pi}}\,y^{2}\,\E^{-\frac{y^{2}}{2}}\,dy.$$
\end{proposition}

\bigskip
\bigskip

To conclude this paragraph, let us give yet another representation of the $3$-dimensional Bessel process. We denote by $\mathfrak{su}(2,\C)$ the space of traceless hermitian matrices of size $2\times 2$. An element of $\mathfrak{su}(2,\C)$ can be written as
$$M=\begin{pmatrix}m_{1} & m_{2}+\I m_{3}\\
m_{2}-\I m_{3} & -m_{1}\end{pmatrix},$$
and $\mathfrak{su}(2,\C)$ is an euclidian space for the scalar product $\scal{A}{B}=\tr AB^{*}=2(a_{1}b_{1}+a_{2}b_{2}+a_{3}b_{3})$. A brownian motion in $\mathfrak{su}(2,\C)$ is a continuous gaussian process $(M_{t})_{t \in \R_{+}}$ starting from the null matrix, and with covariances 
$$\esper[\mathrm{tr}((AM_{s})(BM_{t})^{*})]=K\,(s\wedge t)\,\mathrm{tr}(AB^{*}).$$
For instance, if $X,Y,Z$ are three independent brownian motions in $\R$, then 
$$\begin{pmatrix}X_{t} & Y_{t}+\I Z_{t}\\
Y_{t} - \I Z_{t}& -X_{t}\end{pmatrix}$$
is a brownian motion in $\mathfrak{su}(2,\C)$, with $K=3$. Now, the eigenvalues of such a matrix brownian motion are $\pm \sqrt{X_{t}^{2}+Y_{t}^{2}+Z_{t}^{2}}$, so they are described by a $3$-dimensional Bessel process. Hence, Proposition \ref{bessel3} can be restated as follows: if $\lambda=(\lambda_{1},\lambda_{2})$ is a random partition of law $\proba_{n,\omega}$ with $\omega = ((1/2,1/2),(0))$, then
$$\frac{2}{\sqrt{n}}\left(\lambda_{1}-\frac{n}{2},\lambda_{2}-\frac{n}{2}\right) \simeq_{n \to \infty} (x_{1},x_{2}),$$
where $x_{1}\geq x_{2}$ are the two eigenvalues of a brownian matrix in $\mathfrak{su}(2)$ at time $t=1$. This is the good formulation of Proposition \ref{bessel3} for a generalization to the case of $d$-shuffles.

\subsection{The case of $d$-shuffles}
With a similar framework as before, we want to understand the asymptotics of \textbf{d-shuffles}, meaning that we consider the case when
$$\omega=\left(\left(\frac{1}{d},\frac{1}{d},\ldots,\frac{1}{d},0,0,\ldots \right),(0)\right),$$
with $d$ coordinates equal to $\frac{1}{d}$ in $\alpha$. A generalization of the arguments of the previous paragraph to the case of $d$-shuffles has been proposed by N. O'Connell in \cite{OC03}, in connection with previous results from \cite{OCY02}. Let us denote by $\Pi_{d}$ the set of paths $W : \N \to \N^{d}$, such that $W(0)=(0,0,\ldots,0)$ and 
$$W(k+1)-W(k) \in \{0,e_{1},e_{2},\ldots,e_{d}\}$$
for any $k$, where $e_{i}=(0_{1},0_{2},\ldots,1_{i},\ldots,0_{d})$. 
If $w$ is a word in $\{1,2,\ldots,d\}^{*}$ of length $n$, we associate to it a path $(W^{(k)})_{k \leq n}$ in $\Pi_{d}$, with
$$W(k+1)-W(k)=e_{i}\quad\text{if }w_{k}=i.$$
Now, for any pair of sequences $(x(k))_{k\in \N}$ and $(y(k))_{k \in \N}$ in $\Pi_{1}$, we define two new sequences 
\begin{align*}
(x \Yup y)(k) &= \min \{x(j)-y(j) +y(k) \}_{j \leq k};\\
(x \Ydown y)(k) &= \max \{ x(j)-y(j)+y(k)\}_{j \leq k}.
\end{align*}
Notice that the rules $\Yup$ and $\Ydown$ are not associative. In the following, we will use the following convention: $x \Yup y \Yup z = (x \Yup y) \Yup z$ and $x \Ydown y \Ydown z = (x \Ydown y) \Ydown z$. We then define operators $D_{d} : \Pi_{d} \to \Pi_{d}$ and $T_{d} :  \Pi_{d} \to \Pi_{d-1}$:
\begin{align*}
D_{d}(x_{1},\ldots,x_{d})&=(x_{1},x_{1}\Yup x_{2},x_{1}\Yup x_{2}\Yup x_{3},\ldots,x_{1}\Yup x_{2}\Yup \cdots \Yup x_{d});\\
T_{d}(x_{1},\ldots,x_{d})&=(x_{2}\Ydown x_{1},x_{3}\Ydown (x_{1}\Yup x_{2}),\ldots,x_{d}\Ydown(x_{1}\Yup x_{2}\Yup \cdots \Yup x_{d-1})).
\end{align*}
Finally, from a path $W$ in $\Pi_{d}$, we define new paths $D^{i}W \in \Pi_{d-i+1}$ and $T^{i}W \in \Pi_{d-i}$ by the recursive rule 
\begin{align*} D^{1}W=D_{d}(W_{1},W_{2},\ldots,W_{d})\qquad&;\qquad T^{1}W=T_{d}(W_{1},\ldots,W_{d})\\
D^{i+1}W=D_{d-i}(T^{i}W)\qquad&;\qquad T^{i+1}W=T_{d-i}(T^{i}W).
\end{align*}
For any path $W$, the family $(D^{1}W,D^{2}W,\ldots,D^{d}W)$ is a sequence of triangular arrays. The link between this construction and the RSK algorithm is given by the following result of O'Connell:
\begin{theorem} \label{bradpitt}
For any word $w \in \{1,2,\ldots,d\}^{n}$, the semi-standard tableau of size $n$ obtained from $w$ by Schensted insertion is 
\begin{center}
\includegraphics{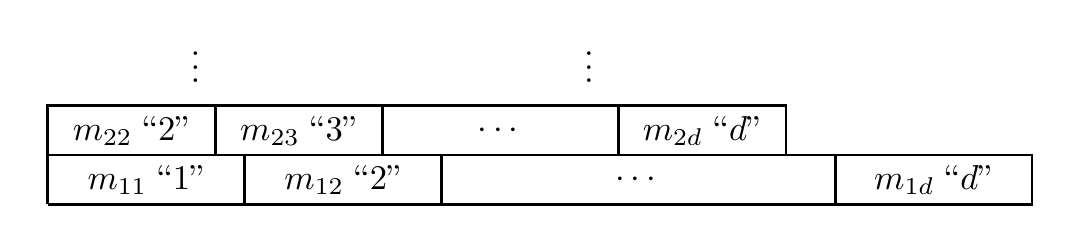}
\end{center}
where\footnote{By convention, $D^{0}W=(0,0,\ldots,0)$ for any path $W$, so $m_{ii}=(D^{1}W)_{i}(n)$ for any $i$.} $m_{ij}=(D^{j-i+1}W)_{i}(n)-(D^{j-i}W)_{i}(n)$ for any $1 \leq i \leq j \leq d$, $W$ being the path associated to the word $w$. In particular, for any $k\leq n$, the shape $(\lambda_{1},\lambda_{2},\ldots,\lambda_{d})$ of $P_{k}(w)$ is
$$\left((D^{d}W)_{1}(k),(D^{d-1}W)_{2}(k),\ldots,(D^{1}W)_{d}(k)\right).$$
\end{theorem}\bigskip
\bigskip

We refer to \cite[Theorem 3.1]{OC03} for a proof of this important result. Now, we have to understand the transformation
\begin{align*}G_{d} : \Pi_{d} &\to \Pi_{d}\\
W &\mapsto \left((D^{d}W)_{1},(D^{d-1}W)_{2},\ldots,(D^{1}W)_{d}\right),
\end{align*}
and what is the law of $G_{d}(W)$ when $W$ is a random path in $\N^{d}$ associated to a random word in $\{1,2,\ldots,d\}^{*}$. To begin with, we shall notice that the operators $\Yup$ and $\Ydown$ also make sense for any continuous or c\`adl\`ag functions $x(t)$ and $y(t)$. Thus, we can define operators
\begin{align*}D_{d} : \mathscr{C}_{0}(\R_{+},\R^{d}) &\to \mathscr{C}_{0}(\R_{+},\R^{d})\\
T_{d} : \mathscr{C}_{0}(\R_{+},\R^{d}) &\to \mathscr{C}_{0}(\R_{+},\R^{d-1}) \\
 G_{d} : \mathscr{C}_{0}(\R_{+},\R^{d}) &\to \mathscr{C}_{0}(\R_{+},\R^{d})
 \end{align*}
where $\mathscr{C}_{0}(\R_{+},\R^{d})$ denotes the set of continuous functions $w(t)$ from $\R_{+}$ to $\R^{d}$, with $w(0)=(0,0,\ldots,0)$. Moreover, $G_{d}$ has the following property, see \cite[Lemma 2.2]{OC03}: for any path $w \in \mathscr{C}_{0}(\R_{+},\R^{d})$ and any time $t \in \R_{+}$,
$$w_{1}(t)+w_{2}(t)+\cdots+w_{d}(t)=(G_{d}w)_{1}(t)+(G_{d}w)_{2}(t)+\cdots+(G_{d}w)_{d}(t).$$
In particular, $G_{d}$ stabilizes the subspace $\mathscr{C}_{0}(\R_{+},H)\subset \mathscr{C}_{0}(\R_{+},\R^{d})$ of functions taking their values in the hyperplane
$$H=\{(x_{1},x_{2},\ldots,x_{d})\,\,|\,\,x_{1}+x_{2}+\cdots+x_{d}=0\}.$$
On the other hand, for any $x,y,f$ functions in $\mathscr{C}_{0}(\R_{+},\R)$, one has 
$$
(x+f) \Yup (y+f)=(x\Yup y)+f\qquad;\qquad(x+f)\Ydown (y+f)=(x\Ydown y)+f.
$$
This implies immediately that:
$$G_{d}(x_{1}+f,x_{2}+f,\ldots,x_{d}+f)=G_{d}(x_{1},x_{2},\ldots,x_{d})+(f,f,\ldots,f).$$
In particular, set $f=\frac{x_{1}+x_{2}+\cdots+x_{d}}{d}$. Then,
$$G_{d}(x_{1},x_{2},\ldots,x_{d})=(f,f,\ldots,f)+G_{d}(x_{1}-f,x_{2}-f,\ldots,x_{d}-f),$$
and the last term is a function in $\mathscr{C}_{0}(\R_{+},H)$; hence, it suffices to study the restriction of $G_{d}$ to the subspace $\mathscr{C}_{0}(\R_{+},H)$.\bigskip
\bigskip

The operators $G_{d}$ are better understood in terms of the \textbf{Pitman transforms} attached to a root system (\emph{cf.} \cite{BBO05}). The hyperplane $H\subset \R^{d}$ contains a root system of type $A_{d-1}$, whose simple roots are $e_{1}-e_{2},e_{2}-e_{3},\ldots,e_{d-1}-e_{d}$. To each simple root $\delta_{i}=e_{i}-e_{i+1}$, we attach an operator $P_{i} : \mathscr{C}_{0}(\R_{+},\R^{d}) \to \mathscr{C}_{0}(\R_{+},\R^{d})$, defined by the following formula:\label{pitmanoperator}
$$(P_{i}w)(t)=w(t)-\left(\min\{ \scal{\delta_{i}}{w(s)}\}_{s \leq t}\right)\,\delta_{i}.$$
Notice that each $P_{i}$ stabilizes the subspace $\mathscr{C}_{0}(\R_{+},H)$. In \cite{BBO05}, it is shown that the operators $P_{i}$ satisfy the braid relations of the generators $s_{i}=(i,i-1)$ of the Coxeter group $\sym_{d}$:
\begin{align*}
&\forall i \leq d-2,\,\,\,P_{i}P_{i+1}P_{i}=P_{i+1}P_{i}P_{i+1};\\
&\forall i,j \leq d-1,\,\,\,|i-j|\geq 2 \Rightarrow P_{i}P_{j}=P_{j}P_{i}.
\end{align*}
It enables us to define $P_{\sigma}$ for any $\sigma \in \sym_{d}$, and it is not hard to see that $P_{\sigma}|_{\mathscr{C}_{0}(\R_{+},H)}$ takes its values in
$$\mathscr{C}_{0}\left(\R_{+},\bigcap_{s_{i} \in L_{\sigma}} H_{i}^{+}\right),$$
where $L_{\sigma}= \{s_{i} \,\,|\,\, \ell(s_{i}\sigma)<\ell(\sigma)\}$, and $H_{i}^{+} =\{x \in H\,\,|\,\,\scal{\delta_{i}}{x}\geq 0\}$ --- \emph{cf.} \cite[Proposition 2.8]{BBO05}. In particular, if $\omega_{0}$ is the permutation of maximal length in $\sym_{d}$, that is to say $(d,d-1,\ldots,1)$, then $P_{\omega_{0}}|_{\mathscr{C}_{0}(\R_{+},H)}$ is an idempotent operator taking its values in $\mathscr{C}_{0}(\R_{+},C),$ where $C=\bigcap_{i=1}^{d-1} H_{i}^{+}$ is the closed Weyl chamber in $H$ corresponding to the simple roots. In terms of coordinates, notice that $x =(x_{1},\ldots,x_{d})$ is in $C$ if and only if
$$x_{1}+\cdots+x_{d}=0 \quad\text{and}\quad x_{1}\geq x_{2}\geq \cdots \geq x_{d}.$$
Now, if $\phi_{k}^{d}$ is the projection $(x_{1},\ldots,x_{d}) \in \R^{d}\mapsto (x_{1},\ldots,x_{k}) \in \R^{k}$, then it is easily seen that
\begin{align*}
T_{d}(W)&=\phi_{d-1}^{d}\,P_{d-1}P_{d-2}\cdots P_{2}P_{1}(W);\\
T^{i}W&=\phi_{d-i}^{d}\,(P_{d-i}\cdots P_{1})\,(P_{d-i+1}\cdots P_{1})\,\cdots\, (P_{d-1}\cdots P_{1})(W).
\end{align*}
From this, one sees that $G_{d}=P_{1}(P_{2}P_{1})\cdots(P_{d-1}\cdots P_{1})= P_{\omega_{0}}$. In particular, $G_{d}|_{\mathscr{C}_{0}(\R_{+},H)}$ takes its values in $\mathscr{C}_{0}(\R_{+},C)$, which was not clear from the first definition. \bigskip
\bigskip

Then, let us consider the path $W$ associated to a random word $w \in \{1,2,\ldots,d\}^{*}$, where each letter of $w$ is independent of the other letters, and uniformly distributed on $\lle 1,d\rre$. If $$\widetilde{W}(k)=W(k)-\left(\frac{k}{d},\frac{k}{d},\ldots,\frac{k}{d}\right),$$ then $\widetilde{W}$ lives in the hyperplane $H$, and by (multi-dimensional) Donsker's theorem (\emph{cf.} \cite[Chapter 2]{Bil69}), after time-space scaling, $\widetilde{W}$ converges to a brownian motion $(w_{t})_{t \in \R}$ in $H$, that is to say a continuous gaussian process in $H$ with covariances
$$\esper[\scal{a}{w_{s}}\,\scal{b}{w_{t}}]=\frac{1}{d}\,(s\wedge t)\,\scal{a}{b}.$$
Indeed, if $X$ is a random step $\widetilde{W}(n+1)-\widetilde{W}(n)$, then $\scal{\delta_{i}}{X}$ equals $1$ if $X=e_{i}$, $-1$ if $X=e_{i+1}$, and $0$ otherwise. From this, one sees that $\esper[\scal{\delta_{i}}{X}\,\scal{\delta_{j}}{X}]$ equals $\frac{2}{d}$ if $i=j$, $-\frac{1}{d}$ if $i=j\pm 1$, and $0$ otherwise. But this is exactly $\frac{\scal{\delta_{i}}{\delta_{j}}}{d}$, whence the result after applying Donsker's theorem. \bigskip
\bigskip

Let us forget for the moment the factor $\frac{1}{d}$, and consider a standard brownian motion $w$ in $H$. Then, it can be shown that $P_{\omega_{0}}(w)$ is a brownian motion conditioned in Doob's sense to stay in the Weyl chamber $C$, \emph{cf.} \cite[Theorem 5.6]{BBO05} or \cite{Bia09}.  We have drawn on Figure \ref{pitmanfig} a random walk in $H \subset \R^{3}$ corresponding to a random word (in blue), and its image by Pitman's transform $P_{\omega_{0}}=P_{1}P_{2}P_{1}$ (in violet); it approximates a brownian motion conditioned to stay in a Weyl chamber\footnote{On the drawing, it appears that an important part of the trajectory of the initial brownian motion is simply translated by Pitman's transform $P_{\omega_{0}}$. This observation can be made precise, and one obtains a generalization of Williams' decomposition of the trajectories of the $\mathrm{BES}^{3}$ process, \emph{cf.} \cite[Chapter 6, Theorem 3.11]{RY91}. }. 
\figcap{
\includegraphics[scale=1.3]{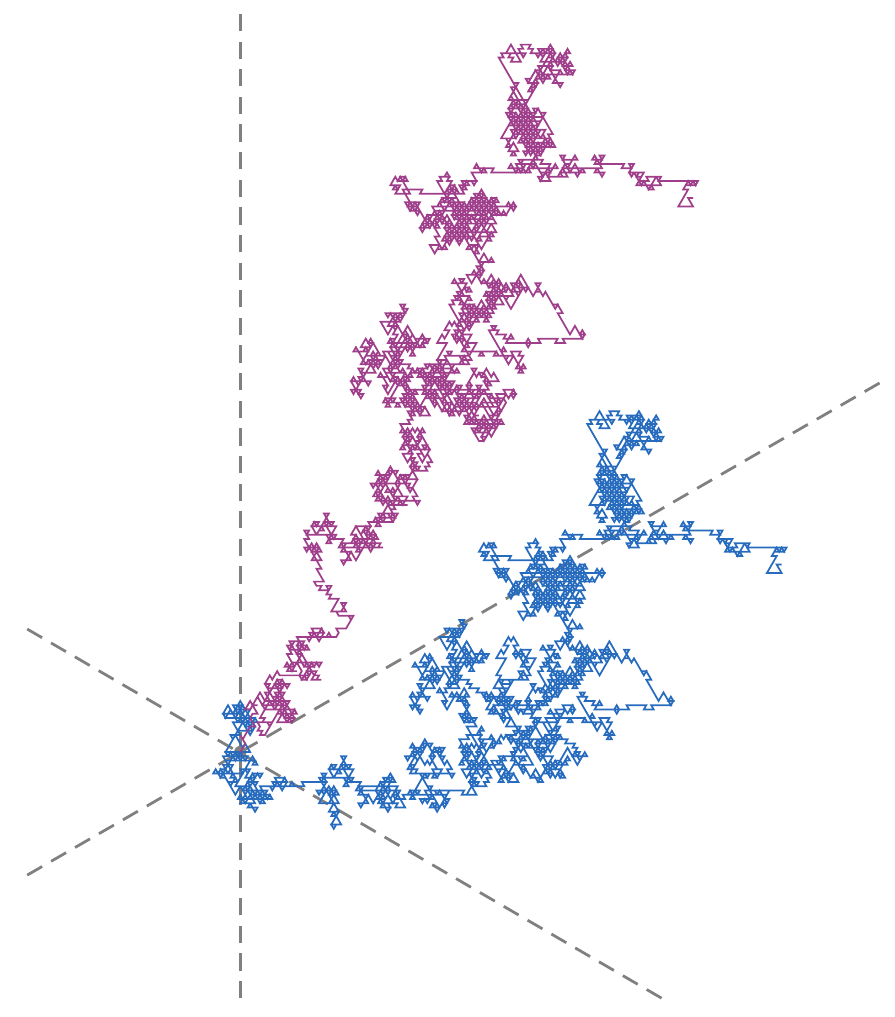}}{Action of Pitman's transform on a brownian motion when $d=3$.\label{pitmanfig}}

This conditioned brownian motion is the Markov process on $C$ with transition probabilities
$$q_{t}(x,y)=\frac{\Delta(y)}{\Delta(x)}\,\sum_{\sigma \in \sym_{d}} \eps(\sigma)\,p_{t}(x,\sigma(y))$$
where $p_{t}(x,y)=(2\pi t)^{-\frac{d-1}{2}}\,\exp\left(-\frac{\|x-y\|^{2}}{2t}\right)$ is the transition kernel of the standard brownian motion in $H$, and 
$$\Delta(x)=\prod_{\alpha \text{ positive root}} \scal{\alpha}{x}=\prod_{1\leq i<j\leq d}x_{i}-x_{j}$$
is (up to a multiplicative constant) the unique positive function on $C$ that is harmonic with respect to the symmetrized kernel and that vanishes at the boundary of $C$. The singularity of the transition kernel at $x=(0,0,\ldots,0)$ is solvable\footnote{This is easily seen by expanding terms $\exp(-\|\eta_{i}-x_{\sigma(i)}\|^{2}/2t)$ with $\eta \to 0$, and by using the theory of antisymmetric polynomials, see \cite[\S1.3]{Mac95}.}, with
$$q_{t}(0,y)=\frac{\Delta(y)^{2} }{1!\,2!\,\cdots\,d-1!\,t^{\frac{d(d-1)}{2}}}\,p_{t}(0,y).$$
This Markov process corresponds to the process of the eigenvalues $$(x_{t,1}\geq x_{t,2}\geq \cdots \geq x_{t,d})$$ of a brownian matrix in $\mathfrak{su}(d,\C)$ with covariances
$$\esper[\mathrm{tr}((AM_{s})(BM_{t})^{*})]=\frac{d^{2}-1}{d}\,(s\wedge t)\,\mathrm{tr}(AB^{*}),$$
see \cite{Dys62,Bia09}, and \cite[Theorem 2.5.2]{AGZ09} for checking that the multiplicative constant is indeed $\frac{d^{2}-1}{d}$. Taking account of the factor $\frac{1}{d}$ put aside before, we get:
\begin{theorem}
In the following, $(M_{t})_{t \in \R_{+}}$ is a brownian matrix in $\mathfrak{su}(d,\C)$ starting from the null matrix and with normalized covariances
$$\esper[\mathrm{tr}((AM_{s})(BM_{t})^{*})]= \frac{d^{2}-1}{d^{2}}\,(s\wedge t)\,\mathrm{tr}(AB^{*});$$
and $(x_{1,t}\geq x_{2,t}\geq \cdots \geq x_{d,t})$ is the set of eigenvalues of $M_{t}$. On the other hand, we consider a random permutation $\sigma$ obtained by a $d$-shuffle, and we denote
$$L_{i,t}^{(n)}= \frac{1}{\sqrt{n}}\left(\left[(\lfloor nt +1\rfloor-nt)\,\Lambda_{i,\lfloor nt \rfloor}(\sigma)+ (nt-\lfloor nt\rfloor)\,\Lambda_{i,\lfloor nt \rfloor +1}(\sigma)\right]-\frac{nt}{d}\right),$$
where $(\Lambda_{1,k}(\sigma),\ldots,\Lambda_{d,k}(\sigma))$ is the shape of the tableau $P_{k}(\sigma)$ obtained by Schensted insertion from the $k$ first letters of $\sigma$. As $n$ goes to infinity,
$$\left(L_{1,t}^{(n)},L_{2,t}^{(n)},\ldots,L_{d,t}^{(n)}\right)_{t \in [0,1]} \to  (x_{1,t},x_{2,t},\ldots,x_{d,t})_{t \in [0,1]},$$
where the convergence is in law and takes place in $\mathscr{C}([0,1],\R^{d})$. In particular, if $\lambda \sim \proba_{n,\omega}$ with $\omega = ((1/d,1/d,\ldots,1/d),(0))$, then\footnote{Here, notice that $dy$ denotes the $(d-1)$-dimensional Lebesgue measure on $H$.}
$$\lambda_{i}=\frac{n}{d}+\sqrt{\frac{n}{d}}\,Y_{i}+o\left(\sqrt{n}\right)$$
with
$$(Y_{1},\ldots,Y_{d})\sim \mathbf{1}_{\substack{y_{1}\geq y_{2}\geq  \cdots \geq y_{d}\quad \\ y_{1}+y_{2}+\cdots+y_{d}=0}}\,\,\,\frac{\Delta(y_{1},y_{2},\ldots,y_{d})^{2}}{1!\,2!\,\cdots\,d-1!\,(2\pi)^{\frac{d-1}{2}}}\,\E^{-\frac{\|y\|^{2}}{2}}\,dy.$$
\end{theorem}

\subsection{The general case of equal coordinates}
We pursue the study of the fluctuations of rows of a tableau obtained by Schensted insertion from a generalized shuffle, and we consider now the case when
$\omega=((p_{1}^{d_{1}},p_{2}^{d_{2}}, \ldots,p_{r}^{d_{r}}),(0))$ with 
$$p_{1}>p_{2}>\cdots>p_{r}>0 \,\,\,\text{ and }\,\,\,d_{1}p_{1}+d_{2}p_{2}+\cdots+d_{r}p_{r}=1.$$
Here, by $p_{i}^{d_{i}}$, we mean that $p_{i}$ appears $d_{i}$ times in $\alpha$. It will be clear from our reasoning that one could even consider an infinite sequence of $p_{i}$'s with the same property. So, if one understands the behaviour of rows with a finite number of parameters, the study of the general case when $\omega=(\alpha,0)$ and $\gamma=0$ will be completed. In the following, we set $d=d_{1}+d_{2}+\cdots+d_{r}$.\bigskip\bigskip

As before, an $\omega$-shuffle of size $n$ corresponding to parameters $p_{1}^{d_{1}},\ldots,p_{r}^{d_{r}}$ can be obtained by standardizing a random word $w=b_{1}b_{2}\ldots b_{n}$ in $\lle 1,d\rre^{n}$, where the letters $b_{k}$ are independent and of law
$$\proba[b_{k}=i]=p_{j}\quad\text{if }d_{1}+\cdots+d_{j-1}+1\leq i \leq d_{1}+\cdots+d_{j}.$$
To such a (random) word, we attach as before a random walk $W\in \Z^{d}$, with $W(k+1)-W(k)=e_{i}$ if $b_{k}=i$. O'Connell's Theorem \ref{bradpitt} holds again, so $G_{d}(W)(k)$ is the form $\Lambda_{k}(\sigma)$ of the tableau $P_{k}(\sigma)$ (or $P_{k}(w)$) obtained by Schensted insertion from the $k$ first letters of $\sigma$. Now, we shall use the following technical lemma:
\begin{lemma}\label{technical}
If $W=(x_{1},\ldots,x_{d})$ is a random walk obtained from a shuffle of parameter $p_{1}^{d_{1}},\ldots,p_{r}^{d_{r}}$, we introduce random functions $f_{1},\ldots,f_{d}$ such that $(f_{1},\ldots,f_{d})$ is the vector
$$
G_{d}(x_{1},\ldots,x_{d})-(G_{d_{1}}(x_{1},\ldots,x_{d_{1}}),\ldots,G_{d_{r}}(x_{d_{1}+\cdots+d_{r-1}+1},\ldots,x_{d_{1}+\cdots+d_{r}})).
$$
These functions are almost surely bounded (as functions of $\N$).
\end{lemma}
\begin{example}
Before proving Lemma \ref{technical}, let us see study the simplest example where it applies, namely, when $\omega=((a,1-a),(0))$ with $a>1/2$. The random walk to consider is two dimensional and will be denoted $w_{n}=(x_{n},y_{n})_{n \in \N}$; moreover,
$$(G_{2}w)_{n}= (x_{n}-\min\{x_{k}-y_{k}\}_{k \leq n},y_{n}+\min\{x_{k}-y_{k}\}_{k \leq n}).$$ However, $(z_{n}=x_{n}-y_{n})_{n \in \N}$ is a random walk on $\Z$ with $\proba[z_{n+1}=z_{n}+1]=a$, $\proba[z_{n+1}=z_{n}-1]=1-a$, and $z_{0}=0$ almost surely. It is well-known that such a random walk is transient, with
\begin{align*}\proba\left[\inf_{n \in \N}z_{n}\leq -m\right]&=\left(\frac{1-a}{a}\right)^{m}\\
\proba\left[\min_{n \in \N}z_{n}=-m\right]&=\left(\frac{1-a}{a}\right)^{m}-\left(\frac{1-a}{a}\right)^{m+1}=\frac{2a-1}{a}\left(\frac{1-a}{a}\right)^{m}.
\end{align*}
Thus, $(G_{2}(x,y))-(x,y)=(-z,z)$, with $|z|$ almost surely bounded by a (random) constant of expectation $\frac{1-a}{2a-1}$.
\end{example}
\begin{proof}[Proof of Lemma \ref{technical}]
The example above can be generalized by using the language of Pitman operators: if $P_{i}$ acts on the random walk with $i=d_{1}+\cdots+d_{j}$ for some $j \in \lle 1,r-1\rre$, then $P_{i}w$ differs from $w$ by almost surely bounded functions on the coordinates of index $d_{1}+\cdots+d_{j}$ and $d_{1}+\cdots+d_{j}+1$. This is exactly for the same reason: the difference of the two random walks $x_{d_{1}+\cdots+d_{j}}$ and $x_{d_{1}+\cdots+d_{j}+1}$ is almost surely bounded from below. Quite obviously, this is also true if one replaces $w$ by $w+\delta$ with $\delta$ almost surely bounded vector, so one can use this argument several times. More generally, this works also if $w$ is replaced by $P_{\sigma}(w)$, where $\sigma$ is a permutation in the Young subgroup $\sym_{d_{1}}\times \sym_{d_{2}}\times \cdots \times \sym_{d_{r}} \subset \sym_{d}$; indeed, such an operator won't change the fact that the coordinates $x_{d_{1}+\cdots+d_{j}}$ and $x_{d_{1}+\cdots+d_{j}+1}$ have different rates of growth to infinity. \bigskip

In the operator $G_{d}=P_{\omega_{0}}=P_{1}(P_{2}P_{1})(P_{3}P_{2}P_{1})\cdots(P_{d-1}\cdots P_{1})$, it allows to neglict all the factors $P_{i}$ with $i=d_{1}+\cdots+d_{j}$: they will simply modify the random path by almost surely bounded functions. So, $G_{d}w\simeq J_{d}w$, where $J_{d}$ is obtained from $G_{d}$ by deleting in the reduced decomposition of $\omega_{0}$ all the $s_{i}$'s with $i=d_{1}+\cdots+d_{j}$, $j \in \lle 1,r-1\rre$. Then, $J_{d}$ acts separately on each block of coordinates $x_{d_{1}+\cdots+d_{j-1}+1},\ldots,x_{d_{1}+\cdots+d_{j}}$, and the action corresponding to the block of size $d_{j}$ is given by the operator
$$G_{d_{j}}\,(P_{d_{j}-1}\cdots P_{2}P_{1})^{k},$$
where $k=d-d_{j}$. Now, each $P_{i}$ is an idempotent operator (this is almost obvious from the definition given on page \pageref{pitmanoperator}), and on the other hand, for any $i$, there is a reduced decomposition of $\omega_{0}$ in $\sym_{d_{j}}$ that ends by $s_{i}$. Hence, 
$$G_{d_{j}}\,(P_{d_{j}-1}\cdots P_{2}P_{1})^{k}=G_{d_{j}},$$
whence the result.
\end{proof}
\bigskip
\bigskip

Now, for each block $x_{d_{1}+\cdots+d_{j-1}+1},\ldots,x_{d_{1}+\cdots+d_{j}}$, one can follow the same reasoning as in the case of $d$-shuffles, with some precaution. First, notice that if $g_{j}=\frac{1}{d_{j}}\sum_{i=1}^{d_{j}} x_{d_{1}+\cdots+d_{j-1}+i}$, then:\vspace{2mm}
\begin{enumerate}
\item In $\N^{r}$, $(d_{1}g_{1},\ldots,d_{r}g_{r})$ is a random walk of parameter $d_{1}p_{1},\ldots,d_{r}p_{r}$.\vspace{2mm}
\item The vectors
$$\widetilde{W}_{j}=(x_{d_{1}+\cdots+d_{j-1}+1}-g_{j},\ldots,x_{d_{1}+\cdots+d_{j-1}+d_{j}}-g_{j})$$
with $j \in \lle 1,r\rre$ are random walks in $H \subset \R^{d_{j}}$, and they are all independent, and independent from the vector $(d_{1}g_{1},\ldots,d_{r}g_{r})$.\vspace{2mm}
\end{enumerate}
One takes account of the first vector by the standard central limit theorem: 
$$d_{i}g_{i}=n(d_{i}p_{i})+\sqrt{n}\,X_{i}+o\left(\sqrt{n}\right)$$
where $(X_{1},\ldots,X_{r})$ is a gaussian vector of covariance matrix $k(X_{i},X_{j})=\delta_{ij}\,d_{i}p_{i}-d_{i}d_{j}\,p_{i}p_{j}$. More precisely, one can describe asymptotically 
$$(d_{1}g_{1}(k)-k\,d_{1}p_{1},\ldots,d_{r}g_{r}(k)-k\,d_{r}p_{r})_{k\leq n}$$ 
by using the multidimensional Donsker theorem. As for the vectors $\widetilde{W}_{j}$, they can be studied with the framework of the previous paragraph on $d$-shufflles. The only difference is the multiplicative constant in the covariances of the brownian motion that is the limit of $\widetilde{W}_{j}$ after time-space scaling: it is $p_{j}$ instead of $\frac{1}{d_{j}}$. In the end, one obtains the following:
\begin{theorem}\label{maingeom}
With the same notations as before, we fix independent brownian motions $(M_{t,j})_{t \in \R_{+}, 1\leq j \leq r}$ taking values in the spaces $\mathfrak{su}(d_{j},\C)$, with covariances 
$$\esper[\mathrm{tr}((AM_{s,j})(BM_{t,j})^{*})]=\frac{d_{j}^{2}-1}{d_{j}}\,p_{j}\,(s\wedge t)\,\mathrm{tr}(AB^{*});$$
$(x_{1,j,t}\geq x_{2,j,t}\geq \cdots \geq x_{d_{j},j,t})$ will denote the set of eigenvalues of $M_{t,j}$. We also consider an independent brownian motion $(B_{1,t},\ldots,B_{r,t})$ in $\R^{d}$, with covariances 
$$\esper[\scal{X}{B_{s}}\,\scal{Y}{B_{t}}]=(s \wedge t)\,X^{t}QY,\quad \text{where }Q=\left(\delta_{ij}\,\frac{p_{i}}{d_{i}}-p_{i}p_{j}\right)_{1\leq i,j\leq r}. $$
On the other hand, we consider a random permutation obtained by a shuffle of parameters $p_{1}^{d_{1}},\ldots,p_{r}^{m_{r}}$, and we denote as before $\Lambda_{k}(\sigma)$ the shape of the standard tableau $P_{k}(\sigma)$, and 
$$L_{i,t}^{(n)}= \frac{1}{\sqrt{n}}\left(\left[(\lfloor nt +1\rfloor-nt)\,\Lambda_{i,\lfloor nt \rfloor}(\sigma)+ (nt-\lfloor nt\rfloor)\,\Lambda_{i,\lfloor nt \rfloor +1}(\sigma)\right]-p_{j}nt\right),$$
assuming that $d_{1}+\cdots+d_{j-1}+1\leq i\leq d_{1}+\cdots+d_{j}$. As $n$ goes to infinity, $\left(L_{1,t}^{(n)},\ldots,L_{i,t}^{(n)},\ldots,L_{d,t}^{(n)}\right)_{t \in [0,1]}$ converges in law to
$$\left(B_{1,t}+x_{1,1,t},\ldots,B_{j,t}+x_{i,j,t},\ldots,B_{r,t}+x_{d_{r},r,t}\right)_{t \in [0,1]},$$
in $\mathscr{C}([0,1],\R^{d})$. In particular, if $\lambda=\Lambda_{n}(\sigma)=\Lambda(\sigma)$ with $\sigma$ chosen randomly as before, then $$\lambda_{i}=np_{j}+\sqrt{np_{j}}\,(X_{j}+Y_{i,j})+o\left(\sqrt{n}\right),$$ with
$$(Y_{1,j},\ldots,Y_{d_{j},j})\sim \mathbf{1}_{\substack{y_{1,j}\geq y_{2,j}\geq  \cdots \geq y_{d_{j},j}\quad \\ y_{1,j}+y_{2,j}+\cdots+y_{d_{j},j}=0}}\,\,\,\frac{\Delta(y_{1,j},y_{2,j},\ldots,y_{d_{j},j})^{2}}{1!\,2!\,\cdots\,d_{j}-1!\,(2\pi)^{\frac{d_{j}-1}{2}}}\,\E^{-\frac{\|y_{j}\|^{2}}{2}}\,dy_{j}$$
and $(X_{1},\ldots,X_{r})$ gaussian vector of covariance matrix  $$Q'=\left(\frac{\delta_{ij}}{d_{i}}-\sqrt{p_{i}p_{j}}\right)_{1\leq i,j \leq r}.$$
\end{theorem}
\begin{remark}
At the end of Section \ref{largenumbers}, we have seen that the law of large numbers can be informally restated by saying that the longest increasing subsequences in a $\alpha$-shuffle are not bigger at order $n$ than the lengths of the blocks shuffled together. Theorem \ref{maingeom} ensures that it is also true at order $\sqrt{n}$, unless some blocks have the same frequency $p_{j}= \alpha_{i}=\alpha_{i'}$. Then, some of the increasing subwords become (stochastically) bigger than the sizes of the blocks to be shuffled, because of the terms $\sqrt{np_{j}} \,Y_{i,j}$.
\end{remark}
\bigskip
\bigskip

Theorem \ref{maingeom} seems to remain true when $\beta\neq 0, \gamma \neq 0$, and in this case, the fluctuations of the first renormalized columns seem to be given by analogous formulas as those given above, with fluctuations ``of type $X$'' described by Theorem \ref{almostgeom}, and fluctuations ``of type $Y$'' (for equal coordinates) independent of the fluctuations of the rows. We don't know how to prove this, but we can give arguments in favor of this conjecture. These arguments can be made rigorous in some cases, for instance when $\gamma=0$ and $\beta$ has only a finite number of non-zero coordinates. When performing the Schensted insertion from a permutation $\sigma$ obtained by an $\omega$-shuffle, let us mark by a point $\bullet$ the boxes containing a number coming from a block of size $m_{j}$, and by a star $\star$ the boxes containing a number coming from the block of size $l$.
\begin{example}
Suppose $n=16$, $n_{1}=n_{2}=3$, $m_{1}=m_{2}=3$, $l=4$. We denote the ``cards'' to be shuffled by $\texttt{0,1,\ldots,9,A,B,\ldots,F}$. After cutting the deck, the five blocks to be shuffled are
$$\texttt{012,\,\,345,\,\,678,\,\,9AB,\,\,CDEF}.$$
The blocks $\texttt{678}$ and $\texttt{9AB}$ are reversed, and $\texttt{CDEF}$ is randomized, becoming from instance $\texttt{DFCE}$. Then, all the blocks are shuffled together, and one can for instance obtain
$$\texttt{8B0D3A97FC41E562}.$$
Now, the partitions appearing in the RSK algorithm are:
$$\young(\bullet)\quad\young(\bullet\bullet)\quad\young(\bullet,~\bullet)\quad\young(\bullet,~\bullet\star)\quad\young(\bullet\bullet,~~\star)\quad\young(\bullet\bullet\star,~~\bullet)\quad\young(\bullet,\bullet\bullet\star,~~\bullet)\quad\young(\bullet,\bullet,\bullet\bullet\star,~~\bullet)$$

$$\young(\bullet,\bullet,\bullet\bullet\star,~~\bullet\star)\quad\young(\bullet,\bullet,\bullet\bullet\star\star,~~\bullet\star)\quad \young(\bullet,\bullet,\bullet,\bullet\bullet\star\star,~~~\star)\quad \young(\bullet,\bullet,\bullet,\bullet,~\bullet\star\star,~~~\star)\quad \young(\bullet,\bullet,\bullet,\bullet,~\bullet\star\star,~~~\star\star)$$

$$ \young(\bullet,\bullet,\bullet,\bullet\star,~\bullet\star\star,~~~~\star)\quad \young(\bullet,\bullet,\bullet,\bullet\star\star,~\bullet\star\star,~~~~\bullet)
\quad\young(\bullet,\bullet,\bullet\star,\bullet\bullet\star,~~\star\star,~~~~\bullet) $$\vspace{1mm}

\noindent Notice that the partition of unmarked boxes 
$$\yng(2,4)$$
 is the one corresponding to $034152$; more generally, during the Schensted insertion, the partition of unmarked boxes evolves as the partition obtained \emph{via} RSK from the subword containing only the letters coming from the blocks of sizes $n_{i}$. Thus, to prove that Theorem \ref{maingeom} holds again when $\beta\neq 0,\gamma \neq 0$ amounts to control the number of marked boxes appearing at the right of the first rows, and to show that it is a $o(\sqrt{n})$ --- then, by symmetry of the problem, we would  obtain the fluctuations of the columns, and see that they are analogous to those of the rows. 
\end{example}
\bigskip

One should notice that each time a unmarked box appears in a row, it pushes a marked box $$\young(\bullet)\quad\text{or}\quad\young(\star)$$ to the next row. Moreover, marked boxes push themselves to upper rows: for instance, if $\gamma=0$ and $\beta=(\beta_{1},\ldots,\beta_{d})$, then one can not  find more than $d$ marked boxes on the same row. So the idea would be to show that marked boxes are ``sufficiently pushed up''. Unfortunately, to do this rigorously, it seems that we would need precise estimates on the probabilities of creating a ``long'' increasing subword in a shuffle of parameter $\alpha=0$; and we don't have for the moment such estimates.\bigskip
\bigskip

To conclude this section, let us look again at the partition-valued random process $(\Lambda_{k}(\sigma))_{k \in \lle 0,n\rre}$, where $\sigma$ is obtained from an $\omega$-shuffle, and $\Lambda_{k}(\sigma)$ is the shape of the tableau obtained by Robinson-Schensted insertion from the $k$ first letters of $\sigma$. We have shown previously that after time-space scaling, and at least when $\beta=\gamma=0$, the coordinates of this process converge jointly to a markovian process that can be written in terms of (conditioned) brownian motions. It raises the following question: is $(\Lambda_{k}(\sigma))_{ k \in \lle 0,n\rre}$ itself a markovian process in $\ym=\bigsqcup_{n \in \N}\ym_{n}$? The answer is yes, and more precisely:
\begin{theorem}\label{markovpartition}
The random process $(\Lambda_{k}(\sigma))_{ k \in \lle 0,n\rre}$ is markovian, with transition probabilities equal to
$$p(\lambda,\Lambda)=\mathbf{1}_{\lambda \nearrow \Lambda}\,\frac{s_{\Lambda}(\omega)}{s_{\lambda}(\omega)}.$$
\end{theorem}
\begin{proof}
This is a generalization of \cite[Theorem 6.1]{OC03}, and the proof is exactly the same. The statement of Theorem \ref{markovpartition} amounts to ask that for any sequence of partitions $\lambda^{(1)}\nearrow \cdots \nearrow \lambda^{(k)}$,
$$\proba[\Lambda_{1}(\sigma)=\lambda^{(1)},\cdots,\Lambda_{k}(\sigma)=\lambda^{(k)}]=s_{\lambda^{(k)}}(\omega).$$
However, one can expand $s_{\lambda^{(k)}}(\omega)$ by using the structure of Hopf algebra of $\Sym$:
\begin{align*}
s_{\lambda^{(k)}}(\omega)&=\left((1 \otimes \upsilon \otimes 1)\circ \Delta^{3}\right)s_{\lambda^{(k)}}(\alpha,\beta,\gamma E)\\
&=\sum_{ \nu \subset \mu \subset \lambda^{(k)}} s_{\nu}(\alpha)\,s_{\mu'\setminus \nu'}(\beta)\,s_{\lambda^{(k)}\setminus \mu}(\gamma E).
\end{align*}
Then, the arguments are exactly the same as in \cite{OC03}.
\end{proof}
\bigskip

\section{Jones-Ocneanu traces of Hecke algebras and (q,t)-Plancherel measures}\label{qtplancherel} 
In this last section, we return to the ``algebraic'' study of the probability measures $\proba_{n,\omega}$, and we introduce new observables in $\obs$ related to the representation theory of the Hecke algebras $\hecke(\sym_{n})$, that are deformations of the symmetric group algebras $\C\sym_{n}$. We then determine their asymptotic behaviour under the probability measures $\proba_{n,\omega}$, thereby proving a $q$-analog of Theorem \ref{mainalg}. Finally, we study the systems of coherent measures coming from the Jones-Ocneanu traces of the Hecke algebras (\emph{cf.} \S\ref{jonesocneanu}); they form a $(q,t)$-deformation of the system of Plancherel measures described on page \pageref{csspezialschur}, see \S\ref{qtasymptotic}.

\subsection{Hecke algebras of the symmetric groups} Let us fix a non-negative integer $n$. We recall that the (generic) \textbf{Iwahori-Hecke algebra} of the symmetric group $\sym_{n}$ is the $\C(q)$-algebra $\hecke(\sym_{n})$ generated by elements $T_{1},T_{2},\ldots,T_{n-1}$, with presentation:
$$
\begin{cases}
&\forall i \leq n-1,\,\,\,(T_{i}-q)(T_{i}+1)=0\,;\\
&\forall i \leq n-2,\,\,\,T_{i}T_{i+1}T_{i}=T_{i+1}T_{i}T_{i+1}\,;\\
&\forall i,j,\,\,\,|j-i| \geq 2 \Rightarrow T_{i}T_{j}=T_{j}T_{i}.
\end{cases}
$$
A linear basis of $\hecke(\sym_{n})$ consists in the $T_{\sigma}=T_{i_{1}}T_{i_{2}}\cdots T_{i_{\ell(\sigma)}}$, where $\sigma$ runs over $\sym_{n}$ and $\sigma=s_{i_{1}}s_{i_{2}}\cdots s_{i_{\ell(\sigma)}}$ is any reduced decomposition of $\sigma$ in elementary transpositions. Thus, $\hecke(\sym_{n})$ is a quantization of the symmetric group algebra $\C\sym_{n}$. In fact, it has the same representation theory: it is a semisimple $\C(q)$-algebra, and its irreducible modules $V^{\lambda}(q)$ are labelled by the integer partitions of size $n$ and have the same dimensions as the irreducible representations $V^{\lambda}$ of $\sym_{n}$, and the same branching rules (see \cite{Mat99}, and \cite{GP00} for more details on the Iwahori-Hecke algebras in the general setting of Coxeter groups).\bigskip\bigskip

As in the case of the symmetric group $\sym_{n}$, the character theory of $\hecke(\sym_{n})$ can be encoded in the algebra $\Sym$ of symmetric functions, \emph{cf.} \cite{Ram91}. A normalized trace on $\hecke(\sym_{n})$ is a function $\tau:\hecke(\sym_{n}) \to \C(q)$ that satisfies the two hypotheses:\vspace{2mm}
\begin{enumerate}
\item $\tau$ is tracial: $\forall a,b,\,\,\tau(ab)=\tau(ba)$.\vspace{2mm}
\item $\tau$ is normalized so that $\tau(1)=1$. \vspace{2mm}
\end{enumerate} 
The condition of positivity is a little more subtle than in the case of the symmetric group. First, for any value of $q \in \C$, we denote $\hecke_{q}(\sym_{n}) = \hecke(\sym_{n})\otimes_{\C(q)}\C$ the specialized Hecke algebra obtained by sending the variable $q$ to its value. This specialized Hecke algebra is a semisimple algebra isomorphic to $\C\sym_{n}$ (whence with the same representation theory) as soon as $q\neq0$ and $q$ is not a non-trivial root of unity; in particular, it is the case when $q \in \R_{+}^{*}$. Then, assuming $q \in \R_{+^{*}}$, a positive normalized trace, or character on $\hecke_{q}(\sym_{n})$ is a normalized trace $\hecke_{q}(\sym_{n})\to \C$ such that (see \cite{KV89}):\vspace{2mm}
\begin{enumerate}
\setcounter{enumi}{2}
\item For all $\sigma_{1},\ldots,\sigma_{r}$ in $\sym_{n}$, $(\tau(T_{\sigma_{i}}T_{\sigma_{j}^{-1}}))_{1\leq i,j\leq r}$ is hermitian and non-negative definite.
\end{enumerate}\bigskip
\bigskip

As in the group case, any trace on $\hecke(\sym_{n})$ is a $\C(q)$-linear combination of the irreducible characters $\zeta^{\lambda}(q)$ of the modules $V^{\lambda}(q)$ of $\hecke(\sym_{n})$; we shall also denote by $\chi^{\lambda}(q)$ the normalized irreducible character of $V^{\lambda}(q)$. A trace on $\hecke_{q}(\sym_{n})$ with $q \in \R_{+}^{*}$ is positive if and only if it is a positive linear combination of the $\zeta^{\lambda}(q)$.  If $\mu=(\mu_{1},\ldots,\mu_{r})$ is a partition in $\ym_{n}$, let us denote
$$
T_{\mu}=(T_{1}\cdots T_{\mu_{1}-1})(T_{\mu_{1}+1}\cdots T_{\mu_{1}+\mu_{2}-1})\cdots(T_{\mu_{1}+\cdots+\mu_{r-1}+1}\cdots T_{\mu_{1}+\cdots+\mu_{r}-1}).
$$
Then, any element $h \in \hecke(\sym_{n})$ is conjugated modulo $[\hecke(\sym_{n}),\hecke(\sym_{n})]$ to a unique linear combination of elements $T_{\mu}$, $\mu \in \ym_{n}$ (\emph{cf.} \cite[\S8.2]{GP00} or \cite[\S5]{Ram91}). So, a trace $\tau : \hecke(\sym_{n}) \to \C(q)$ is entirely determined by the values $\tau(T_{\mu})$, and the character theory of $\hecke(\sym_{n})$ is given by the character table
$$\left(\zeta^{\lambda}(q,\mu)=\zeta^{\lambda}(q,T_{\mu})\right)_{\lambda,\mu \in \ym_{n}}.$$
 Now, using the notations introduced in \S\ref{symfunc}, for any partition $\mu$, we set
$$q_{\mu}(q,X)=\frac{h_{\mu}(qX-X)}{(q-1)^{\ell(\mu)}},$$
where $h_{\mu}=\prod_{i=1}^{\ell(\mu)}s_{\mu_{i}}$ is the homogeneous symmetric function of type $\mu$; $q_{\mu}(q,X)$ is an element of $\Sym \otimes_{\C}\C(q)$.
 A generalization of the Frobenius formula \ref{frobenius} reads as follows (see \cite[Theorem 4.14]{Ram91}):
\begin{proposition}\label{qfrobenius}
For any partition $\mu \in \ym_{n}$,
$$q_{\mu}(q,X)=\sum_{\lambda \in \ym_{n}} \zeta^{\lambda}(q,\mu)\,s_{\lambda}(X) .$$
\end{proposition}\bigskip\bigskip

Then, using exactly the same arguments as in the case of $\sym_{\infty}$, we get the $q$-analog of Thoma's theorem \ref{thomatheorem}:
\begin{theorem}
Suppose $q\in \R_{+}^{*}$. The set of extremal characters $\mathcal{X}^{*}(\hecke_{q}(\sym_{\infty}))$ can be identified (homeomorphically) with the Thoma simplex $\Omega$, and the character $\chi^{\omega}(q)$ corresponding to a point $\omega=(\alpha,\beta) \in \Omega$ writes as
$$\chi^{\omega}(q,T_\mu)=q_{\mu}(q,A-B+\gamma E).$$
The decomposition of ${\chi^{\omega}(q)}_{|\hecke_{q}(\sym_{n})}$ in irreducible characters is
$$\chi^{\omega}(q)=\sum_{\lambda\in \ym_{n}}\proba_{n,\omega}[\lambda]\,\chi^{\lambda}(q),$$
where $\proba_{n,\omega}[\lambda]=(\dim\lambda)\,s_{\lambda}(A-B+\gamma E)$ exactly as before.
\end{theorem}\bigskip

This result provides ``new'' coherent systems of probability measures on $\ym=\bigsqcup_{n \in \N}\ym_{n}$ coming from characters of $\hecke_{q}(\sym_{\infty})$ instead of characters of $\sym_{\infty}$. Moreover, letting the parameter $q$ vary in $\R_{+}^{*}$, we will obtain one-parameter families of coherent systems. In the following, we will investigate the (asymptotic) behaviour of certain families coming from the so-called Jones-Ocneanu traces of $\hecke(\sym_{\infty})$.  The theory developed in the previous paragraphs holds again; however, in the setting of the Hecke algebras, it seems more appriopriate to enounce the algebraic central limit theorem \ref{mainalg} for $q$-characters instead of characters. The end of this paragraph is devoted to this problem.\bigskip\bigskip

For $\lambda$ and $\mu$ partitions, we introduce as in \S\ref{polfunc} the symbol
$$\varSigma_{\mu}(q,\lambda)=\begin{cases}
n(n-1)\ldots (n-|\mu|+1)\,\chi^{\lambda}(q,\mu \sqcup 1^{n-|\mu|}) &\text{if }n=|\lambda|\geq |\mu|,\\
0&\text{otherwise}.
\end{cases}$$
At first sight, it is not clear that the functions $\varSigma_{\mu}(q)$ are observables of diagrams in the algebra $\obs$. However, one can expand the symmetric functions $q_{\mu}(q,X)$ as combinations of power sums $p_{\mu}(X)$, and therefore write the $q$-characters as linear combinations of characters, see \cite[Theorem 5.4]{Ram91} and \cite[Proposition 10]{FM10}:
$$\forall \lambda,\mu,\,\,\,(q-1)^{\ell(\mu)}\,\zeta^{\lambda}(q,T_{\mu})=\sum_{|\nu|=|\mu|} (q^{\nu}-1)\,\frac{\scal{h_{\mu}}{p_{\nu}}}{\scal{p_{\nu}}{p_{\nu}}}\,\zeta^{\lambda}(\nu)$$
Here, for a partition $\nu$, $q^{\nu}-1$ is an abbreviation for $\prod_{i=1}^{\ell(\nu)}q^{\nu_{i}}-1$. From this, one deduces that the $\varSigma_{\mu}(q)$ are indeed observables of diagrams, and that they are related to the $\varSigma_{\mu}$ by the following relations (notice that the change of basis is triangular):
\begin{align*}
(q-1)^{\ell(\mu)}\,\varSigma_{\mu}(q) &= \sum_{|\nu|=|\mu|} (q^{\nu}-1)\,\frac{\scal{h_{\mu}}{p_{\nu}}}{\scal{p_{\nu}}{p_{\nu}}}\,\varSigma_{\nu};\\
(q^{\mu}-1)\,\varSigma_{\mu}&=\sum_{|\nu|=|\mu|} (q-1)^{\ell(\nu)}\,\scal{p_{\mu}}{m_{\nu}}\,\varSigma_{\nu}(q).
\end{align*}
Let us fix a parameter $\omega \in \Omega$. Theorem \ref{mainalg} ensures that the variables 
$$\Delta_{n,\omega}(\mu)=\sqrt{n}\left(\frac{\varSigma_{\mu}(\lambda)}{n^{|\mu|}}-p_{\mu}(\omega)\right)$$
with $\lambda \sim \proba_{n,\omega}$ converge jointly to a gaussian vector; since the $\varSigma_{\mu}(q)$'s are linear combinations of the $\varSigma_{\mu}$'s, the same result holds for the variables 
 $$\Delta_{n,\omega}(q,\mu)=\sqrt{n}\left(\frac{\varSigma_{\mu}(q,\lambda)}{n^{|\mu|}}-q_{\mu}(q,\omega)\right).$$ 
We denote by $\Delta_{\infty,\omega}(q,\mu)$ the limiting gaussian variables; their covariances $k(\Delta_{\infty,\omega}(q,\mu^{(1)}),\Delta_{\infty,\omega}(q,\mu^{(2)}))$ are equal by bilinearity to
 \begin{align*}
&\sum_{|\nu^{(i)}|=|\mu^{(i)}|} C_{\mu^{(1)},\mu^{(2)},\nu^{(1)},\nu^{(2)}}(q)\, \,k(\Delta_{\infty,\omega}(\nu^{(1)}),\Delta_{\infty,\omega}(\nu^{(2)}))\\
=& \sum_{\substack{|\nu^{(i)}|=|\mu^{(i)}| \\ a \in \nu^{(1)}\,,\, b \in \nu^{(2)}}} ab\,\, C_{\mu^{(1)},\mu^{(2)},\nu^{(1)},\nu^{(2)}}(q) \,\,\left(\frac{p_{a+b-1}(\omega)}{p_{a}(\omega)\,p_{b}(\omega)}-1\right),
 \end{align*}
where  $$C_{\mu^{(1)},\mu^{(2)},\nu^{(1)},\nu^{(2)}}(q) =\frac{(q^{\nu^{(1)}\sqcup\nu^{(2)}}-1)}{(q-1)^{\ell(\mu^{(1)})+\ell(\mu^{(2)})}}\,\frac{\scal{h_{\mu^{(1)}}}{p_{\nu^{(1)}}}\,\scal{h_{\mu^{(2)}}}{p_{\nu^{(2)}}} }{z_{\nu^{(1)}}\,z_{\nu^{(2)}} }.$$
\begin{theorem}
As $n$ goes to infinity, if $\lambda \sim \proba_{n,\omega}$, then the variables $$\sqrt{n}\left(\chi^{\lambda}(q,\mu)-\chi^{\omega}(q,\omega)\right)$$ converge jointly to a gaussian vector, whose covariance matrix is given by the formula above.
\end{theorem}
\noindent In particular, when $\mu^{(1)}=l$ and $\mu^{(2)}=m$ are partitions of length $1$, the scalar products $\scal{h_{\mu^{(i)}}}{p_{\nu^{(i)}}}$ are all equal to $1$ by Frobenius formula \ref{frobenius}, and thus, $k(\Delta_{\infty,\omega}(q,l),\Delta_{\infty,\omega}(q,m))$ is equal to
$$\frac{1}{(q-1)^{2}}\sum_{\substack{\nu^{(1)} \in \ym_{l} \\ \nu^{(2)} \in \ym_{m} }}  \frac{p_{\nu^{(1)}\sqcup \nu^{(2)}}(q\omega - \omega)}{z_{\nu^{(1)}} \,z_{\nu^{(2)}}}\left(\sum_{\substack{ a \in \nu^{(1)}\\ b \in \nu^{(2)}}} ab\,\left(\frac{p_{a+b-1}(\omega)}{p_{a}(\omega)\,p_{b}(\omega)}-1\right)\right).$$
In the general case, it does not seem possible to simplify this formula; however, we obtained a simplification in the case of the $q$-Plancherel measures, see \cite{FM10} and \cite{Mel10}.

\subsection{Jones-Ocneanu traces and their weights}\label{jonesocneanu}
A \textbf{Markov trace} (also-called \textbf{Jones-Ocneanu trace}, \cite{Jon87}) on $\hecke(\sym_{n})$ is a normalized trace that satisfies the additional condition:
$$
\forall k \in \lle 1,n-1\rre,\,\,\,\forall a \in \hecke(\sym_{k}), \tau(T_{k}\,a)=\tau(T_{k})\,\tau(a)=z\,\tau(a).
$$
If $z$ is a fixed parameter, there is a unique corresponding Markov trace on $\hecke(\sym_{n})$, characterized by the values $\tau_{z}(T_{\mu})=z^{|\mu|-\ell(\mu)}$.
\begin{proposition}
Let $X_{q,z}=\frac{[1-q+z]-[z]}{1-[q]}$ denote the formal alphabet that corresponds to the following specialization of $\Sym$: 
$$
\forall k \geq 1,\,\,\,p_{k}(X_{q,z})=\frac{(1-q+z)^{k}-z^{k}}{1-q^{k}}.
$$
The decomposition of $\tau_{z}$ in irreducible $q$-characters is
$$
\tau_{z}=\sum_{\lambda \in \ym_{n}}s_{\lambda}(X_{q,z})\,\zeta^{\lambda}(q).
$$
\end{proposition}
\begin{proof}
The result is due to Ocneanu (unpublished notes) and Wenzl (see \cite{Wen88}); we give here a very short proof. If $Y_{q,z}=qX_{q,z}-X_{q,z}$, then $p_{k}(Y_{q,z})=(q^{k}-1)\,p_{k}(X_{q,z})=z^{k}-(1-q+z)^{k}$. Consequently,
\begin{align*}
P(u)&=\sum_{k=1}^{\infty} \frac{p_{k}(Y_{q,z})}{k}\,u^{k}=\log\left( \frac{1-u(1-q+z)}{1-uz}\right)\\
H(u)&=\sum_{k=0}^{\infty}h_{k}(Y_{q,z})\,u^{k}=\exp P(u)=\frac{1-u(1-q+z)}{1-uz}.
\end{align*}
Hence, $h_{k}(Y_{q,z})=(q-1)z^{k-1}$, and $$q_{\mu}(q,X_{q,z})=\frac{h_{\mu}(Y_{q,z})}{(q-1)^{\ell(\mu)}}=\prod_{i=1}^{\ell(\mu)} z^{\mu_{i}-1} =z^{|\mu|-\ell(\mu)}=\tau_{z}(T_{\mu}).$$
Now, because of the $q$-Frobenius formula \ref{qfrobenius}, for any partition $\mu$,
$$
\tau_{z}(T_{\mu})=q_{\mu}(q,X_{q,z})=\sum_{\lambda \in \ym_{n}} s_{\lambda}(X_{q,z})\,\zeta^{\lambda}(q,T_{\mu}).
$$
Since a trace is characterized by its values on the $T_{\mu}$'s, the result follows immediately.
\end{proof}
\bigskip
\bigskip

There is a hook length formula for the specializations $s_{\lambda}(X_{q,z})$ of the Schur functions, see for instance \cite[\S1.3]{Mac95}:
$$s_{\lambda}(X_{q,z})=q^{n(\lambda)}\,\prod_{\oblong \in \lambda} \frac{(1-q)+z(1-q^{c(\oblong)})}{1-q^{h(\oblong)}}
$$
where $n(\lambda)=\sum_{i=1}^{\ell(\lambda)}(i-1)\lambda_{i}$, and if $\oblong =(i,j)$ is a box of the Young diagram of $\lambda$, $c(\oblong)=i-j$ is the \textbf{content} of $\oblong$. Now, let us make the change of variables:
$$
z=-(1-q)(1-t).
$$
Then, assuming that $q \in \R_{+}^{*}$ and $t \in [0,1]$, one sees that $s_{\lambda}(X_{q,z})$ is positive, because
$$
s_{\lambda}(X_{q,z})=\frac{q^{b(\lambda)}}{\prod_{\oblong \in \lambda} \{h(\oblong)\}_{q}}\,\left(\prod_{\oblong \in \lambda} t+(1-t)q^{c(\oblong)}\right)
$$
with $\{i\}_{q}=\frac{1-q^{i}}{1-q}$ for any positive integer $i$. Another way to see that $s_{\lambda}(X_{q,z})$ is positive in this situation is to notice that $$X_{q,z}=A-\overline{B},$$ where $A$ and $B$ are the alphabets associated to the point of the Thoma simplex $\omega_{q,t}=(\alpha_{q,t},\beta_{q,t}) $, with
\begin{align*}
\alpha_{q,t} &= \frac{[t(1-q)]}{1-[q]} = (t(1-q),t(1-q)q,t(1-q)q^{2},\ldots);\\
\beta_{q,t} &= \frac{[(1-t)(1-q)]}{1-[q]} = ((1-t)(1-q),(1-t)(1-q)q,(1-t)(1-q)q^{2},\ldots).
\end{align*}
In particular:
\begin{proposition}
If $z=-(1-q)(1-t)$ with $t \in [0,1]$ and $q \in \R_{+}^{*}$, then the Markov trace $\tau_{z}$ is an irreducible (positive) character of $\hecke_{q}(\sym_{\infty})$, and it is equal to $\chi^{\omega_{q,t}}$ with $\omega_{q,t}$ defined by the formulas above.
\end{proposition}

\begin{example}
Suppose that $t=1$. Then, $z=0$ and $\tau=\tau_{0}$ is the \textbf{regular trace} of $\hecke_{q}(\sym_{\infty})$:
$$\forall \omega \in \sym_{\infty},\,\,\,\tau(T_\omega)=\begin{cases}
1 & \text{if }\omega=1,\\
0 & \text{otherwise}.
\end{cases}
$$
Although the definition of the regular does not seem to depend on the value of $q$, the weights of $\tau$ in its decomposition in irreducible characters vary with $q$:
$$\forall n,\,\,\,\tau_{|\hecke_{q}(\sym_{n})}=(1-q)^{n}\sum_{\lambda \in \ym_{n}} (\dim \lambda)\,s_{\lambda}(1,q,q^{2},\ldots)\,\chi^{\lambda}(q).$$
The probability measures corresponding to this decomposition have been studied extensively in \cite{FM10}. In the following, we generalize the results of this paper by using the general theorems proved in Sections \ref{largenumbers} and \ref{clt}.
\end{example}

\subsection{$(q,t)$-Plancherel measures and their asymptotics}\label{qtasymptotic}~
For $q \in \R_{+}^{*}$ and $t \in [0,1]$, we call \textbf{(q,t)-Plancherel measure} of size $n$ the probability measure $\proba_{n,q,t}$ on $\ym_{n}$ corresponding to the decomposition of the Markov trace $\tau_{z}$ with $z=-(1-q)(1-t)$. In other words, $\proba_{n,q,t}=\proba_{n,\omega_{q,t}}$, and
$$
\proba_{n,q,t}[\lambda]=(1-q)^{n}\,(\dim \lambda)\,\, s_{\lambda}\left(\frac{[t]-[t-1]}{1-[q]}\right).$$
The $(q,t)$-Plancherel measures have the following symmetry: $$\proba_{n,q,t}[\lambda]=\proba_{n,q^{-1},t}[\lambda'].$$
Hence, we shall always suppose that $q \in\, ]0,1[$, the other case following by a symmetry. 
\begin{remark}
When $q$ goes to $1$, one recovers the usual Plancherel measures defined on page \pageref{csspezialschur} 
$$\proba_{n,q=1,t}[\lambda]=\proba_{n}[\lambda]=\frac{(\dim \lambda)^{2}}{n!}.$$
Hence, the $(q,t)$-Plancherel measures form a two-parameter deformation of the coherent system of Plancherel measures. They are also a generalization of the $q$-Plancherel measures defined in \cite{Ker92,Stra08} and studied in \cite{FM10}.
\end{remark}
\bigskip\bigskip

Using the results of \S\ref{randompermutation}, one can interpret the $(q,t)$-Plancherel measures as push-forwards of probability measures on the symmetric groups. Let us detail a little bit these measures $\qproba_{n,q,t}$ such that $\Lambda_{\star}\qproba_{n,q,t}=\proba_{n,q,t}$. We could of course take $\qproba_{n,q,t}[\sigma]=\qproba_{n,\omega_{q,t}}[\sigma]$, but it will be a little bit easier\footnote{In return, we lose the direct interpretation of the random permutations as generalized shuffles.} to describe the probability measures
$$\qproba_{n,q,t}[\sigma]=\qproba_{n,\omega_{q,t}}[\sigma^{-1}]=F_{\sigma}(\omega_{q,t})=L_{c(\sigma)}(\omega_{q,t}).$$
Taking $\sigma^{-1}$ instead of $\sigma$ does not change the fact that $\Lambda_{\star}\qproba_{n,q,t}=\proba_{n,q,t}$, because the bilateral Kazhdan-Lusztig cells $\{\sigma\in \sym_{n}\,\,|\,\,\Lambda(\sigma)=\lambda\}$ are stabilized by the inversion $\sigma \mapsto \sigma^{-1}$. Now, for any composition $c \in \comp_{n}$, if $\mathrm{comaj}(c)=\sum_{d \in D(c)} n-d$, then
$$\forall q \in\, ]0,1[,\,\,\,L_{c}(1,q,q^{2},\ldots)=\frac{q^{\mathrm{comaj}(c)}}{(1-q)(1-q^{2})\cdots(1-q)^{n}},$$
see \cite[\S7.19]{Stan91}. Using the definition of the specialization $L_{c}(\omega)$ (\emph{cf.} Paragraph \ref{pushforward}), one sees that:
\begin{align*}\qproba_{n,q,t}[\sigma]&=L_{c(\sigma)}(\omega_{q,t})=\sum_{i=0}^{n} L_{c(\sigma)_{\lle1,i\rre}}(\alpha_{q,t}) \,L_{\overline{c(\sigma)_{\lle i+1,n \rre}}}(\beta_{q,t})\\
&=\sum_{i=0}^{n} L_{c(\sigma)_{\lle1,i\rre}}(\alpha_{q,t}) \,L_{c(\sigma\omega_{0})_{\lle 1,n-i \rre}}(\beta_{q,t})\\
&=\sum_{i=0}^{n} \frac{t^{i}\,(1-t)^{n-i}}{\{i!\}_{q}\,\{n-i!\}_{q}}\,q^{\mathrm{comaj}(c(\sigma)_{\lle 1,i\rre})+\mathrm{comaj}(c(\sigma\omega_{0})_{\lle 1,n-i\rre})}.\end{align*}
For $w$ word of size $n$, we set $\qproba_{n,q}[w]=\frac{q^{\mathrm{comaj}(w)}}{\{n!\}_{q}}$; for any $q \in \R_{+}^{*}$, $\qproba_{n,q}$ is a probability measure on $\sym_{n}$. 
\begin{proposition}\label{interpolation}
As $t$ varies in $[0,1]$, the measures $\qproba_{n,q,t}$ provide an interpolation between $\qproba_{n,q}[\sigma]$ and $\qproba_{n,q}[\sigma\omega_{0}]$:
$$\qproba_{n,q,t}[\sigma]=\sum_{i=0}^{n} t^{i}\,(1-t)^{n-i}\,\qproba_{i,q}[\sigma_{\lle 1,i\rre}]\,\qproba_{n-i,q}[(\sigma\omega_{0})_{\lle 1,n-i\rre}]$$
where $\sigma_{\lle 1,i\rre}$ denotes the word $\sigma(1)\sigma(2)\cdots \sigma(i)$, and similarly for $(\sigma\omega_{0})_{\lle 1,n-i\rre}$.
\end{proposition}\bigskip

 It was \emph{a priori} a non trivial fact that these interpolations correspond \emph{via} RSK to the weights of the Markov traces $\tau_{-(1-q)(1-t)}$. On the other hand, let us notice that the parts $\mathbb{Q}_{i,q}[\sigma_{\lle 1,i\rre}]\,\mathbb{Q}_{n-i,q}[(\sigma\omega_{0})_{\lle 1,n-i\rre}]$ of $\mathbb{Q}_{n,q,t}[\sigma]$ have a simple geometric interpretation. If $\sigma$ is a permutation, let us draw its corresponding ribbon, see \S\ref{fqsym}. We draw a line after the $i$-th box, and we define $m_{i}(\sigma)$ as the sum of the following quantities:
\begin{itemize}
\item for each box $\young(\vee)$ before the line and with another box just under, the distance between $\young(\vee)$ and the line in the ribbon;
\item for each box $\young(<)$ after the line and with another box just at the left, the distance between $\young(<)$ and the line in the ribbon.
\end{itemize}
Then, the $i$-th weight is simply $\frac{q^{m_{i}(\sigma)}}{\{i!\}_{q}\,\{n-i!\}_{q}}$; for instance, if  $i=5$ and $\sigma=64182357$, then $m_{i}(\sigma)=(4_{6}+3_{4}+1_{8})+(0_{3}+1_{5}+2_{7})=11$.
$$\young(6,4,18,:2357)\quad=\quad\young(\vee,\vee,~\vee,:~<<<)\,.$$
\bigskip
\bigskip

Now, let us describe the asymptotics of the $(q,t)$-Plancherel measures. First, we give the algebraic central limit theorem, noticing that
$$p_{k}(\omega_{q,t})=\frac{(t^{k}-(t-1)^{k})(1-q)^{k}}{1-q^{k}}$$
for any $k\geq 1$.
\begin{proposition}
If $\lambda \sim \proba_{n,q,t}$, then as $n$ goes to infinity, the random variables 
$$\sqrt{n}\left(\chi^{\lambda}(q,T_{l})-(t-1)^{l-1}(1-q)^{l-1}\right)$$ converge jointly to a gaussian vector $(\Delta_{\infty,q,t}(q,l))_{l \geq 1}$ with covariances 
\begin{align*}
&(1-q)^{-l-m+2}\,k(\Delta_{\infty,q,t}(q,l),\Delta_{\infty,q,t}(q,m))
\\
&=\sum_{\substack{\nu^{(1)}\in \ym_{l}\\ \nu^{(2)}\in \ym_{m} \\ a \in \nu^{(1)} \,,\, b \in \nu^{(2)} }} \frac{(t-1)^{\nu} - t^{\nu } }{z_{\nu^{(1)}}\,z_{\nu^{(2)}} } \left(\frac{1-u^{a+b-1}}{1-q^{a+b-1}}\,\frac{1-u}{1-q}\,\frac{1-q^{a}}{1-u^{a}}\,\frac{1-q^{b}}{1-u^{b}}-1\right)
\end{align*}
where  $u= \frac{t-1}{t}$, $\nu=\nu^{(1)}\sqcup \nu^{(2)}$, and $(t-1)^{\nu}-t^{\nu}=\prod_{i=1}^{\ell(\nu)} ((t-1)^{\nu_{i}}-t^{\nu_{i}})$.
\end{proposition}
\begin{remark}
When $t=1$, this formula can be simplified by using a M\"obius inversion formula, and it becomes
$$k(\Delta_{\infty,q,t}(q,l),\Delta_{\infty,q,t}(q,m))=\frac{(q-q^{2})^{l+m-3}\,(1-q^{2})\,\{l-1\}_{q}\,\{m-1\}_{q}}{\{l+m-1\}_{q}\,\{l+m-2\}_{q}\,\{l+m-3\}_{q}}$$
see \cite{Mel10}.
\end{remark}\bigskip

Now, let us give the geometric version of the central limit theorem. We have drawn below a random Young diagram taken according to the $(q,t)$-Plancherel measure of parameters $n=300$, $q=\frac{1}{2}$ and $t=\frac{2}{3}$. 
\figcap{
\includegraphics[scale=0.8]{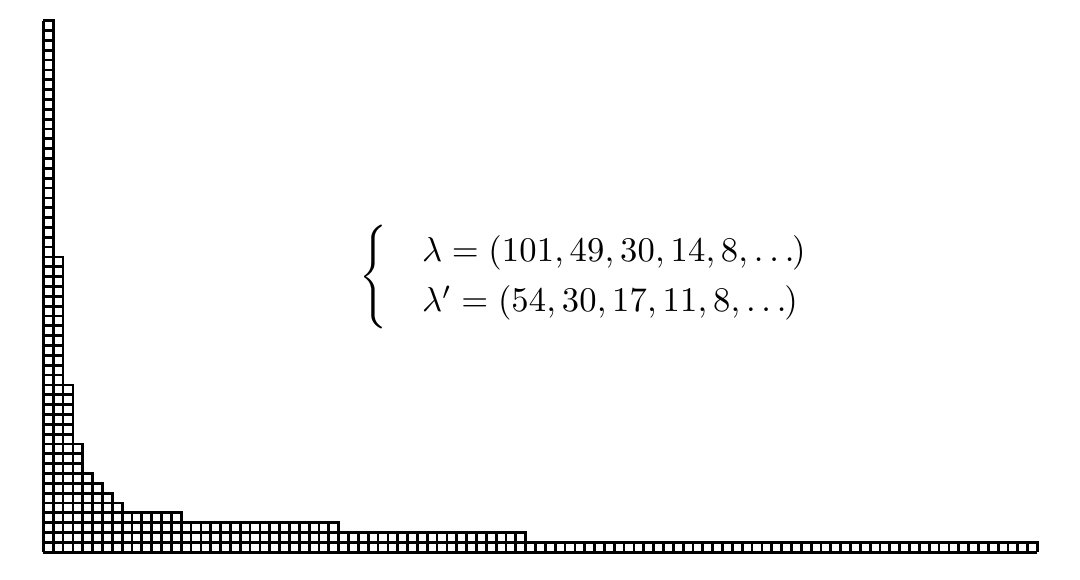}}{A random Young diagram under a $(q,t)$-Plancherel measure.}
With $q=\frac{1}{2}$, it is not hard to see that the first rows and the first columns seem to follow geometric progressions of reason $\frac{1}{2}$. Moreover, with $t=\frac{2}{3}$, the ratio between the first row and the first column seems to be close to $\frac{t}{1-t}=2$. This is in adequation with Theorems \ref{llngeom} and \ref{almostgeom} that state the following (notice that we are in the case of inequal coordinates):
\begin{proposition}\label{lastone}
If $\lambda \sim \proba_{n,q,t}$, then 
\begin{align*}
\lambda_{i} &= n\,t(1-q)q^{i-1}+n^{1/2}\,X_{i}+o(n^{1/2});\\
\lambda_{j}'&=n\,(1-t)(1-q)q^{j-1}+n^{1/2}\,Y_{j}+o(n^{1/2}),
\end{align*}
 where $(X_{i},Y_{j})_{i\geq 1,j\geq 1}$ is a gaussian vector with covariances
\begin{align*}
k(X_{i},X_{j})&=\delta_{i,j} \,t(1-q)q^{i-1} -t^{2}(1-q)^{2}q^{i+j-2};\\
k(Y_{i},Y_{j})&=\delta_{i,j} \,(1-t)(1-q)q^{i-1} -(1-t)^{2}(1-q)^{2}q^{i+j-2};\\
k(X_{i},Y_{j})&=-t(1-t)(1-q)^{2}q^{i+j-2}.
\end{align*}
\end{proposition}
\noindent In particular, for $t=1$, we recover Theorems 1 and 2 of \cite{FM10}. Proposition \ref{lastone} also gives the asymptotic behaviour of the length of a longest increasing or decreasing subword in a random permutation taken according to the probability measure described in Proposition \ref{interpolation}.
\bigskip

\section*{Conclusion}
\noindent Let us summarize the results obtained in this paper:\vspace{2mm}
\begin{enumerate}
\item To any irreducible character of $\sym_{\infty}$ or of $\hecke_{q}(\sym_{\infty})$ with $q \in \R_{+}^{*}$, we have associated a point $\omega$ in the Thoma simplex $\Omega$, and two random models: a model $(\qproba_{n,\omega})_{n\in \N}$ of random permutations obtained by generalized riffle shuffles, and a model $(\proba_{n,\omega})_{n \in \N}$ of random partitions related to the restriction of the irreducible character to the finite symmetric groups or their Hecke algebras.\vspace{2mm}
\item The computation of the probabilities of both models can be done using specializations of combinatorial Hopf algebras: the Hopf algebra $\FQSym$ for the permutations, and the Hopf algebra $\Sym$ for the partitions. This interpretation enabled us to see that the two models are related by the RSK correspondence.\vspace{2mm}
\item The model of random partitions satisfies a law of large numbers and a central limit theorem. In both cases, the asymptotic result can be given an algebraic flavour (for the values of the characters), or a geometric flavour (for the first rows and columns of the partitions). The geometric central limit theorem is much more involved than the other results, and cannot be proved by using only techniques of moments. Hence, one had to relate random permutations and partitions to some conditioned random walks, and to the behaviour of the eigenvalues of brownian hermitian matrices.\vspace{2mm}
\end{enumerate}
When $\omega=\mathbf{0}$, our methods do not describe with sufficient precision the asymptotic behaviour of the random partitions; however, they can be adapted to this purpose, see \cite{IO02}. One obtains again a law of large numbers and a central limit theorem, and this CLT seems to have a universality property; \emph{cf.} \cite{Mel11}. Hence, the next step in the asymptotic study of the models presented in this article would be to understand what happens ``in the neighbourhood'' of the point $\mathbf{0}$ of the Thoma simplex. More precisely, we would like to understand what happens when $\omega$ \emph{varies with $n$ and converges to $\mathbf{0}$}. Let us give a concrete example --- this is one of the two cases studied in \cite{Mel11}, and we refer also to \cite{Bia01}. Suppose that
$$\omega_{n}=\left(\left(\frac{1}{N},\ldots,\frac{1}{N},0,0,\ldots\right),(0,0,\ldots)\right),$$
with $N$ coordinates equal to $\frac{1}{N}$, and $cN \simeq \sqrt{n}$. Then, there is a limit shape for the random partitions $\lambda \sim \proba_{n,\omega_{n}}$; it has been found by P. Biane in \cite{Bia01}, and it is drawn on Figure \ref{schurweyl} in the case when $c=1$.
\figcap{
\includegraphics{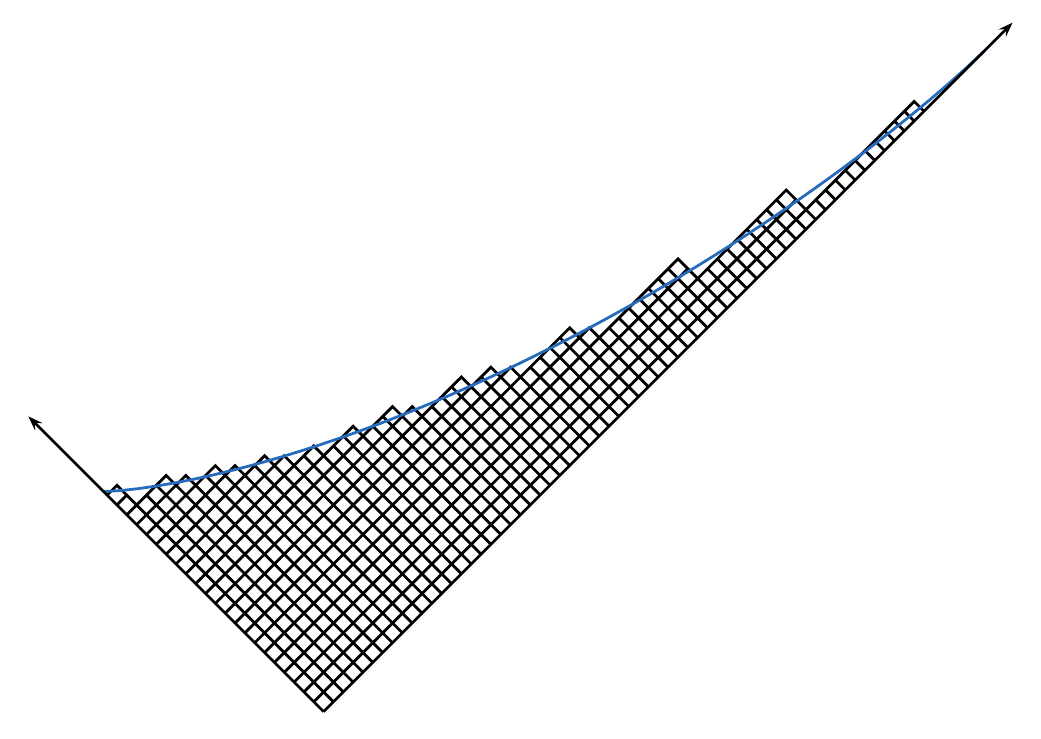}
}{Random partition under the Schur-Weyl measure of parameter $c=1$.\label{schurweyl}}

Now, the interesting fact is that the fluctuations of the partitions under these measures (called \textbf{Schur-Weyl measures}) are described by a gaussian free field that does not depend on $c$ (or, to be more precise, only by a translation of the fluctuations on the $x$-axis). More generally, we conjecture the following (the statement is willingly left unprecise):
\begin{conjecture}\label{lastconjecture}
This asymptotic behaviour is universal among models of random partitions associated to parameters $(\omega_{n})_{n \in \N}$ converging to $\mathbf{0}$ at speed $1/\sqrt{n}$ as $n$ goes to infinity.
\end{conjecture}
\bigskip

There is one example whose study would be very interesting in order to prove this conjecture --- it has been suggested by M. Bo\.{z}ejko. In some sense, the Schur-Weyl measures provide another one-parameter deformation of the Plancherel measures: indeed, when $c$ goes to $0$, the limit shape of the partitions under Schur-Weyl measures converges to the one corresponding to the Plancherel measures. On the other hand, the limit shapes under Schur-Weyl measures correspond \emph{via} the Markov-Krein correspondence (see \cite{Ker98}) to the Mar\v{c}enko-Pastur measures, that are a one-parameter deformation of the Wigner semicircle law (corresponding \emph{via} Markov-Krein to the limit shape of the random partitions under Plancherel measures). Now, the \textbf{free Meixner laws} (see \emph{e.g.} \cite{BB06}) form a two-parameter deformation of the Wigner measure, and they generalize the Mar\v{c}enko-Pastur laws. So it would be extremely interesting to construct models of random partitions corresponding to these laws, and to study their fluctuations. One can hope that with two parameters, the computations will be sufficiently generic in order to understand the behaviour of random partitions in the general situation of Conjecture \ref{lastconjecture}. The results of this paper have been obtained in a similar way: starting from the $q$-Plancherel measures already studied in \cite{FM10}, we have gone to the case of $(q,t)$-Plancherel measures, and we have then realized that our results hold in fact for any model of random partitions associated to a point of the Thoma simplex.

\bigskip
\bigskip

\bibliographystyle{alpha}
\bibliography{cltinfinity}

\begin{thebibliography}{KOV04}

\bibitem[AGZ09]{AGZ09}
G.~W. Anderson, A.~Guionnet, and O.~Zeitouni.
\newblock {\em An Introduction to Random Matrices}, volume 118 of {\em
  Cambridge Studies in Advanced Mathematics}.
\newblock Cambridge University Press, 2009.

\bibitem[AS05]{AS05}
M.~Aguiar and F.~Sottile.
\newblock Structure of the {M}alvenuto-{R}eutenauer {H}opf algebra of
  permutations.
\newblock {\em Adv. Math.}, 191(2):225--275, 2005.

\bibitem[BB06]{BB06}
M.~Bo\.{z}ejko and W.~Bryc.
\newblock On a class of free {L\'evy} laws related to a regression problem.
\newblock {\em J. Funct. Anal.}, pages 59--77, 2006.

\bibitem[BBO05]{BBO05}
P.~Biane, P.~Bougerol, and N.~O'Connell.
\newblock Littelmann paths and brownian paths.
\newblock {\em Duke Math. J.}, 130(1):127--167, 2005.

\bibitem[Bia01]{Bia01}
P.~Biane.
\newblock Approximate factorization and concentration for characters of
  symmetric groups.
\newblock {\em Internat. Math. Res. Notices}, 4:179--192, 2001.

\bibitem[Bia09]{Bia09}
P.~Biane.
\newblock From {P}itman's theorem to crystals.
\newblock {\em Advanced Studies in Pure Mathematics}, 55:1--13, 2009.

\bibitem[Bil69]{Bil69}
P.~Billingsley.
\newblock {\em Convergence of Probability Measures}.
\newblock Wiley, 1969.

\bibitem[Bri69]{Bri69}
D.~Brillinger.
\newblock The calculation of cumulants via conditioning.
\newblock {\em Ann. Inst. Statist. Math.}, 21:375--390, 1969.

\bibitem[DF{\'S}10]{DFS10}
M.~Do{\l}ega, V.~F{\'e}ray, and P.~{\'S}niady.
\newblock Explicit combinatorial interpretation of {K}erov character
  polynomials as numbers of permutation factorizations.
\newblock {\em Adv. Math.}, 225(1):81--120, 2010.

\bibitem[DHT01]{NCSF6}
G.~Duchamp, F.~Hivert, and J.-Y. Thibon.
\newblock Noncommutative symmetric functions {VI}: {F}ree quasi-symmetric
  functions and related algebras.
\newblock \texttt{arXiv:math/0105065v1}, 2001.

\bibitem[Dys62]{Dys62}
F.~J. Dyson.
\newblock A brownian-motion model for the eigenvalues of a random matrix.
\newblock {\em J. Math. Phys.}, 3:1191--1198, 1962.

\bibitem[FM10]{FM10}
V.~F{\'e}ray and P.-L. M{\'e}liot.
\newblock Asymptotics of $q$-{P}lancherel measures.
\newblock To appear in \emph{Probability Theory and Related Fields},
  \texttt{arXiv:1001.2180v1 [math.RT]}, 2010.

\bibitem[Ful97]{Ful97}
W.~Fulton.
\newblock {\em Young Tableaux with Applications to Representation Theory and
  Geometry}, volume~35 of {\em London Mathematical Society Students Texts}.
\newblock Cambridge University Press, 1997.

\bibitem[Ful02]{Ful02}
J.~Fulman.
\newblock Applications of symmetric functions to cycle and increasing
  subsequence structure after shuffles.
\newblock {\em J. Alg. Combin.}, 16:165--194, 2002.

\bibitem[Ges94]{Ges84}
I.~M. Gessel.
\newblock Multipartite {$P$}-partitions and inner products of skew {S}chur
  functions.
\newblock In {\em Combinatorics and algebra (Boulder, Colo., 1983)}, pages
  289--317. Amer. Math. Soc., 1994.

\bibitem[GK10]{GK10}
R.~Gohm and C.~K{\"o}stler.
\newblock Noncommutative independence from characters of the infinite symmetric
  group $s\infty$.
\newblock \texttt{arXiv:1005.5726}, 2010.

\bibitem[GP00]{GP00}
M.~Geck and G.~Pfeiffer.
\newblock {\em Characters of Finite Coxeter Groups and Iwahori-Hecke Algebras},
  volume~21 of {\em London Mathematical Society Monographs}.
\newblock Oxford University Press, 2000.

\bibitem[IK99]{IK99}
V.~Ivanov and S.~Kerov.
\newblock The algebra of conjugacy classes in symmetric groups, and partial
  permutations.
\newblock In {\em Representation Theory, Dynamical Systems, Combinatorial and
  Algorithmical Methods III}, volume 256 of {\em Zapiski Nauchnyh Seminarov
  POMI}, pages 95--120, 1999.

\bibitem[IO02]{IO02}
V.~Ivanov and G.~Olshanski.
\newblock Kerov's central limit theorem for the {P}lancherel measure on {Y}oung
  diagrams.
\newblock In {\em Symmetric Functions 2001: Surveys of Developments and
  Perspectives}, volume~74 of {\em NATO Science Series II. Mathematics, Physics
  and Chemistry}, pages 93--151, 2002.

\bibitem[JK81]{JK81}
G.~D. James and A.~Kerber.
\newblock {\em The Representation Theory of the Symmetric Group}, volume~16 of
  {\em Encyclopedia of Mathematics and its Applications}.
\newblock Addison-Wesley, 1981.

\bibitem[Jon87]{Jon87}
V.~F.~R. Jones.
\newblock Hecke algebra representations of braid groups and link polynomials.
\newblock {\em Ann. of Math.}, 126(2):335--388, 1987.

\bibitem[Ker92]{Ker92}
S.~V. Kerov.
\newblock $q$-analogue of the hook walk algorithm and random {Y}oung tableaux.
\newblock {\em Funct. Anal. Appl.}, 26(3):179--187, 1992.

\bibitem[Ker93]{Ker93}
S.~V. Kerov.
\newblock Gaussian limit for the {P}lancherel measure of the symmetric group.
\newblock {\em Comptes Rendus Acad. Sci. Paris, S\'erie I}, 316:303--308, 1993.

\bibitem[Ker98]{Ker98}
S.~V. Kerov.
\newblock Interlacing measures.
\newblock {\em Amer. Math. Soc. Transl.}, 181(2):35--83, 1998.

\bibitem[KO94]{KO94}
S.~V. Kerov and G.~Olshanski.
\newblock Polynomial functions on the set of {Y}oung diagrams.
\newblock {\em Comptes Rend. Acad. Sci. Paris, S\'erie I}, 319:121--126, 1994.

\bibitem[KOV04]{KOV04}
S.~V. Kerov, G.~Olshanski, and A.~M. Vershik.
\newblock Harmonic analysis on the infinite symmetric group.
\newblock {\em Invent. Math.}, 158:551--642, 2004.

\bibitem[KV77]{KV77}
S.~V. Kerov and A.~M. Vershik.
\newblock Asymptotics of the {P}lancherel measure of the symmetric group and
  the limiting form of {Y}oung tableaux.
\newblock {\em Doklady AN SSSR}, 233(6):1024--1027, 1977.

\bibitem[KV81a]{KV81a}
S.~V. Kerov and A.~M. Vershik.
\newblock Asymptotic theory of characters of the symmetric group.
\newblock {\em Funct. Anal. Appl.}, 15(4):246--255, 1981.

\bibitem[KV81b]{KV81b}
S.~V. Kerov and A.~M. Vershik.
\newblock Characters and factor representations of the infinite symmetric
  group.
\newblock {\em Soviet Math. Doklady}, 23:389--392, 1981.

\bibitem[KV84]{KV84}
S.~V. Kerov and A.~M. Vershik.
\newblock Characters, factor representations and {K}-functor of the infinite
  symmetric group.
\newblock In {\em Proc. Intern. Conf. on Operator Algebras and Group
  Representations 1980}, volume~2 of {\em Monographs and Studies in
  Mathematics}, pages 23--32. Pitman, London, 1984.

\bibitem[KV89]{KV89}
S.~V. Kerov and A.~M. Vershik.
\newblock Characters and realizations of representations of an
  infinite-dimensional {H}ecke algebra, and knot invariants.
\newblock {\em Soviet Math. Doklady}, 38(1):134--137, 1989.

\bibitem[Lit95]{Litt95}
P.~Littelmann.
\newblock Paths and root operators in representation theory.
\newblock {\em Ann. Math.}, 142(3):499--525, 1995.

\bibitem[LLT02]{LLT02}
A.~Lascoux, B.~Leclerc, and J.-Y. Thibon.
\newblock The plactic monoid.
\newblock In {\em Algebraic Combinatorics on Words}. Cambridge University
  Press, 2002.

\bibitem[LS77]{LS77}
B.~F. Logan and L.~A. Shepp.
\newblock A variational problem for random {Y}oung tableaux.
\newblock {\em Adv. Math.}, 26:206--222, 1977.

\bibitem[Mac95]{Mac95}
I.~G. Macdonald.
\newblock {\em Symmetric functions and {H}all polynomials}.
\newblock Oxford Mathematical Monographs. Oxford University Press, 2nd edition,
  1995.

\bibitem[Mat99]{Mat99}
A.~Mathas.
\newblock {\em {I}wahori-{H}ecke algebras and {S}chur algebras of the symmetric
  group}, volume~15 of {\em University Lecture Series}.
\newblock Amer. Math. Soc., 1999.

\bibitem[M{\'e}l10]{Mel10}
P.-L. M{\'e}liot.
\newblock Gaussian concentration of the q-characters of the {H}ecke algebras of
  type {A}.
\newblock \texttt{arXiv:1009.4288v1 [math.RT]}, 2010.

\bibitem[M{\'e}l11]{Mel11}
P.-L. M{\'e}liot.
\newblock Kerov's central limit theorem for {S}chur-{W}eyl and {G}elfand
  measures.
\newblock To appear in the proceedings of the 23th International Conference on
  Formal Power Series and Algebraic Combinatorics (Reykjavik), 2011.

\bibitem[MR95]{MR95}
C.~Malvenuto and C.~Reutenauer.
\newblock Duality between quasi-symmetric functions and {S}olomon descent
  algebra.
\newblock {\em J. Algebra}, 177:967--982, 1995.

\bibitem[O'C03]{OC03}
N.~O'Connell.
\newblock A path-transformation for random walks and the {R}obinson-{S}chensted
  correspondence.
\newblock {\em Trans. Amer. Math. Soc.}, 355:3669--3697, 2003.

\bibitem[Oko97]{Oko97}
A.~Okounkov.
\newblock On representations of the infinite symmetric group. {R}epresentation
  theory, dynamical systems, combinatorial and algorithmic methods {II} ({A}.
  {M}. {V}ershik, ed.).
\newblock {\em Zapiski Seminarov POMI}, 240:167--229, 1997.

\bibitem[OV04]{OV04}
A.~Okounkov and A.~M. Vershik.
\newblock A new approach to the representation theory of the symmetric groups,
  {II}.
\newblock {\em Zapiski Seminarov POMI}, 307:57--98, 2004.

\bibitem[OY02]{OCY02}
N.~O'Connell and M.~Yor.
\newblock A representation for non-colliding random walks.
\newblock {\em Electron. Comm. Probab.}, 7:1--12, 2002.

\bibitem[Pit75]{Pit75}
J.~W. Pitman.
\newblock One-dimensional {B}rownian motion and the three-dimensional {B}essel
  process.
\newblock {\em Advances in Applied Probability}, 7:511--526, 1975.

\bibitem[Ram91]{Ram91}
A.~Ram.
\newblock A {F}robenius formula for the characters of the {H}ecke algebras.
\newblock {\em Invent. Math.}, 106:461--488, 1991.

\bibitem[Rob38]{Rob38}
G.~Robinson.
\newblock On the representations of the symmetric group.
\newblock {\em Amer. J. Math.}, 60(3):745--760, 1938.

\bibitem[RY91]{RY91}
D.~Revuz and M.~Yor.
\newblock {\em Continuous Martingales and Brownian Motion}, volume 293 of {\em
  Grundlehren der mathematischen Wissenschaften}.
\newblock Springer-Verlag, 3rd edition, 1991.

\bibitem[Sch61]{Sch61}
C.~Schensted.
\newblock Longest increasing and decreasing subsequences.
\newblock {\em Canadian Journal of Mathematics}, 13:179--191, 1961.

\bibitem[{\'S}ni06a]{Sni06a}
P.~{\'S}niady.
\newblock Asymptotics of characters of symmetric groups, genus expansion and
  free probability.
\newblock {\em Discrete Math.}, 306(7):624--665, 2006.

\bibitem[{\'S}ni06b]{Sni06b}
P.~{\'S}niady.
\newblock Gaussian fluctuations of characters of symmetric groups and of
  {Y}oung diagrams.
\newblock {\em Probability Theory and Related Fields}, 136(2):263--297, 2006.

\bibitem[Sta91]{Stan91}
R.~P. Stanley.
\newblock {\em Enumerative combinatorics}.
\newblock Cambridge University Press, 1991.

\bibitem[Sta01]{Stan01}
R.~P. Stanley.
\newblock A generalized riffle shuffle and quasisymmetric functions.
\newblock {\em Annals of Combinatorics}, 5:479--491, 2001.

\bibitem[Str08]{Stra08}
E.~Strahov.
\newblock A differential model for the deformation of the {P}lancherel growth
  process.
\newblock {\em Adv. Math}, 217(6):2625--2663, 2008.

\bibitem[Tha87]{Tha87}
D.~H. Thang.
\newblock Random operators in {B}anach spaces.
\newblock {\em Probability and Mathematical Statistics}, 8:155--167, 1987.

\bibitem[Tho64]{Tho64}
E.~Thoma.
\newblock Die unzerlegbaren, positive-definiten {K}lassenfunktionen der
  abzählbar unendlichen symmetrischen {G}ruppe.
\newblock {\em Math. Zeitschrift}, 85:40--61, 1964.

\bibitem[VV95]{VV95}
M.~V. Velasco and A.~R. Villena.
\newblock A random {B}anach-{S}teinhaus theorem.
\newblock {\em Proc. Amer. Math. Soc.}, 123(8):2489--2497, 1995.

\bibitem[Was81]{Was81}
A.~J. Wassermann.
\newblock Automorphic actions of compact groups on operator algebras.
\newblock Ph.D. Thesis, University of Pennsylvania, 1981.

\bibitem[Wen88]{Wen88}
H.~Wenzl.
\newblock Hecke algebras of type ${A}_n$ and subfactors.
\newblock {\em Invent. Math.}, 92:349--383, 1988.

\bibitem[Zel81]{Zel81}
A.~Zelevinsky.
\newblock {\em Representations of finite classical groups: A Hopf algebra
  approach}, volume 869 of {\em Lecture Notes in Mathematics}.
\newblock Springer-Verlag, 1981.

\end{thebibliography}

\end{document}